\documentclass[10pt]{amsart}
\usepackage{fullpage}
\usepackage{amsfonts, amssymb, amsthm, amsmath, amscd, euscript}
\usepackage[all]{xypic}
\usepackage{hyperref}
\usepackage{comment}
\usepackage{enumerate}
\usepackage{rotating}
\usepackage{multicol}

\numberwithin{equation}{subsection}
\theoremstyle{plain}
\newtheorem{theorem}[equation]{Theorem}
\newtheorem{proposition}[equation]{Proposition}
\newtheorem{corollary}[equation]{Corollary}
\newtheorem{lemma}[equation]{Lemma}

\newtheorem{conjecture}[equation]{Conjecture}

\newtheorem{itheorem}{Theorem}
\numberwithin{itheorem}{section}
\newtheorem{icorollary}[itheorem]{Corollary}

\theoremstyle{definition}
\newtheorem{definition}[equation]{Definition}

\newtheorem{notation}[equation]{Notation}

\newtheorem{example}[equation]{Example}

\newtheorem{remark}[equation]{Remark}

\newcommand{\beq}{\begin{equation}}
\newcommand{\eeq}{\end{equation}}

\newcommand{\mb}{\mathbb}

\newcommand{\mf}{\mathfrak}

\newcommand{\kk}{{\Bbbk}}

\renewcommand{\AA}{\mb{A}}

\newcommand{\CC}{{\Bbbk}}

\newcommand{\NN}{\mathbb Z_{\ge0}} 
\newcommand{\OO}{\mb{O}}

\newcommand{\QQ}{\mb{Q}}
\newcommand{\RR}{\mb{R}}

\newcommand{\ZZ}{\mb{Z}}

\newcommand{\DMO}{\DeclareMathOperator}
\DMO{\GK}{GKdim}
\DMO{\PGK}{PGKdim}
\DMO{\gr}{gr}
\DMO{\Der}{Der}
\DMO{\trdeg}{trdeg}
\DMO{\Kdim}{Kdim}
\DMO{\codim}{codim}

\DMO{\Res}{Res}
\DMO{\MSpec}{MSpec}
\DMO{\Sa}{S}
\DMO{\Ua}{U}
\DMO{\Core}{Core}
\DMO{\rk}{rk}
\DMO{\Ker}{\ker}
\DMO{\End}{End}
\DMO{\Aut}{Aut}
\DMO{\im}{im}
\DMO{\Ann}{Ann}
\DMO{\supp}{supp}
\DMO{\ldeg}{ldeg}
\DMO{\sdeg}{gaps}
\DMO{\CoreP}{P}

\newcommand{\Shifts}{{\it Shifts}}
\newcommand{\Shift}{{\it Shift}}
\newcommand{\DLoc}{{\it DLoc}}
\newcommand{\Loc}{{\it Loc}}
\newcommand{\Dil}{{\it Dil}}
\DMO{\lie}{lie}
\DMO{\PSpec}{PSpec}
\DMO{\len}{len}
\DMO{\Spec}{Spec}

\newcommand{\Vir}{V\!ir}
\newcommand{\VP}{V_P}

\newcommand{\del}{\partial}
\newcommand{\ssm}{\setminus}

\DMO{\ev}{ev}

\title[ Poisson primes in the symmetric algebra of the Virasoro algebra]{The Poisson spectrum of the symmetric algebra of the Virasoro algebra}
\author{Alexey V. Petukhov and Susan J. Sierra}
\date{\today}

\address{Sierra: School of Mathematics, The University of Edinburgh, Edinburgh EH9 3FD, United Kingdom}

\email{s.sierra@ed.ac.uk}

\address{Petukhov: Institute for Information Transmission Problems, Bolshoy Karetniy 19-1, Moscow 127994, Russia}
\email{alex-{}-2@yandex.ru}


\begin{document}

\begin{abstract}
    Let $W = \mathbb{C}[t,t^{-1}]\del_t$ be the {\em Witt algebra } of algebraic vector fields on $\mathbb{C}^\times$ and let $\Vir$ be the {\em Virasoro algebra}, the unique nontrivial central extension of $W$.
    In this paper, we study the Poisson ideal structure of the symmetric algebras of $\Vir$ and $W$, as well as several related Lie algebras.
    We classify prime Poisson ideals and Poisson primitive ideals of $\Sa(\Vir)$ and $\Sa(W)$.
    In particular, we show that the only functions in $W^*$ which vanish on a nontrivial Poisson ideal (that is, the only maximal ideals of $\Sa(W)$ with a nontrivial Poisson core) are given by linear combinations of derivatives at a finite set of points; we call such functions {\em local}.  
   Given a local function $\chi\in W^*$, we construct the associated Poisson primitive ideal through computing the algebraic symplectic leaf of $\chi$, which gives a notion of coadjoint orbit in our setting.
   
    As an application, we prove a structure theorem for subalgebras of $\Vir$ of finite codimension and show in particular that any such subalgebra of $\Vir$ contains the central element $z$, substantially generalising a result of Ondrus and Wiesner on subalgebras of codimension 1.
    As a consequence, we deduce that $\Sa(\Vir)/(z-\zeta)$ is Poisson simple if and only if $ \zeta \neq 0$.
    \end{abstract}
    
    \maketitle
    
\tableofcontents

\section{Introduction}\label{INTRO}


Let $G$ be a connected algebraic group over $\mathbb C$ with Lie algebra $\mf g$, and consider the coadjoint action of $G$ on $\mf g^*$.
This is a beautiful classical topic, with profound connections to areas from geometric representation theory to combinatorics to physics.
Algebraic geometry tells us that  coadjoint orbits in $\mf g^*$ correspond to $G$-invariant radical ideals in the symmetric algebra $\Sa(\mf g)$.  

These can  also be defined using the Kostant-Kirillov {\em Poisson bracket} on $\Sa(\mf g)$:
$$\{f, g\}=\sum_{i,j}\frac{\partial f}{\partial e_i}\frac{\partial f}{\partial e_j}[e_i, e_j]$$
where $\{e_i\}$ is a basis of $\mf g$. 
Recall that an ideal $I$ of $\Sa(\mf g)$ is a {\em Poisson ideal} if $I$ is also a Lie ideal for the Poisson bracket.  
A basic fact is that $I$ is $G$-invariant if and only if $I$ is Poisson.

Thus to compute the closure of the coadjoint orbit  of 
 $\chi \in \mf g^*$, let $\mf m_\chi$ \label{ind:mchi} 
 be the kernel of the 
evaluation morphism $${\rm ev}_\chi: {\rm S}(\mf g)\to\mb \CC,$$ 
and let $\CoreP(\chi)$ 
be the {\em Poisson core} \label{ind:poissoncore}
of $\mf m_\chi$:  
the maximal Poisson ideal contained in $\mf m_\chi$.  
By definition, an ideal of the form $\CoreP(\chi)$ is called {\em Poisson primitive}; 
\label{ind:poissonprimitive}
by a slight abuse of notation, we refer to $\CoreP(\chi)$ as the {\em Poisson core of $\chi$}. 
The closure of the coadjoint orbit of $\chi$ is defined by $\CoreP(\chi)$:
\beq\label{orbit} \overline{G \cdot \chi} = V(\CoreP(\chi)) := \{ \nu \in \mf g^* \ | \ \ev_\nu(\CoreP(\chi))=0\},\eeq
and so $\chi,  \nu \in \mf g^*$ are in the same $G$-orbit if and only if $\CoreP(\chi) = \CoreP(\nu)$.
In the case of algebraic Lie algebras over $\mathbb C$ or $\RR$, coadjoint orbits are symplectic leaves for the respective Poisson structure. 

In this paper, we investigate how this theory extends to the {\em Witt algebra} $W= \mb C[t, t^{-1}]\del_t$ of algebraic vector fields on $\mb C^\times$, and to its central extension the {\em Virasoro algebra} $\Vir = \mb C[t, t^{-1}] \del_t \oplus \mb C z$, with Lie bracket given by 
\[ [f \del_t, g \del_t] = (f g' - f'g) \del_t + \Res_0(f' g''-g'f'') z, \quad  \mbox{$z$ is central.} \]
(We also consider some important Lie subalgebras of $W$.)
These infinite-dimensional Lie algebras, of fundamental importance in representation theory and in physics, have no adjoint group \cite{Lempert}, but one can still study the Poisson cores of maximal ideals, and more generally the Poisson ideal structure of $\Sa(W)$ and $\Sa(\Vir)$.
Motivated by \eqref{orbit}, we will say that functions $\chi, \nu \in \Vir^*$ or in $W^*$ are in the same {\em pseudo-orbit} if $\CoreP(\chi) = \CoreP(\nu)$.
These (coadjoint) pseudo-orbits can be considered as algebraic symplectic leaves in $\Vir^*$ or $W^*$.

Taking the discussion above as our guide, we focus on prime Poisson ideals and Poisson primitive ideals of $\Sa(\Vir)$ and $\Sa(W)$.
Important  questions here, which for brevity we ask in the introduction only for $\Vir$, include:
\begin{itemize}
    \item Given $\chi \in \Vir^*$, can we compute the Poisson core $\CoreP(\chi)$ and the pseudo-orbit of $\chi$?  When is $\CoreP(\chi)$ nontrivial?
    \item How can we understand prime Poisson ideals of $\Sa(\Vir)$?  Can we parameterise them in a reasonable fashion, ideally in a way which gives us further information about the ideal?  How does one distinguish Poisson primitive ideals from other prime Poisson ideals?  
    \item It is known, see \cite[Corollary~5.1]{LSS}, that $\Sa(\Vir)$ satisfies the ascending chain condition on prime Poisson ideals.  The augmentation ideal of $\Sa(\Vir)$, that is, the ideal generated by $\Vir \subset \Sa(\Vir)$, is clearly a maximal Poisson ideal.  What are the others?   Conversely, does any nontrivial prime Poisson ideal have finite height?
    \item Do prime Poisson ideals induce any reasonable algebraic geometry on the uncountable-dimensional vector space $\Vir^*$?
\end{itemize}
We  answer all of these questions,   almost completely working out the structure of the Poisson spectra of $\Sa(\Vir)$ and $\Sa(W)$. 

Let us begin by discussing the idea of algebraic geometry on $\Vir^*$.  
A priori, this seems completely intractable as $\Vir^*$ is an uncountable-dimensional affine space;  little interesting can be said about $\Sa(\mf a)$ where $\mf a$ is a countable-dimensional {\em abelian} Lie algebra.
However,  $\Vir$ and $W$ are extremely noncommutative and so Poisson ideals in their symmetric algebras are very large:  in particular, by  a  result of Iyudu and the second author \cite[Theorem~1.3]{IS},  if $I$ is a nontrivial Poisson ideal of $\Sa(W)$ (respectively, a non-centrally generated Poisson ideal of $\Sa(\Vir)$), then $\Sa(W)/I$ (respectively, $\Sa(\Vir)/I$) has polynomial growth. 
This suggests that we might hope that a Poisson primitive ideal, and more generally a prime Poisson ideal, would correspond to a finite-dimensional algebraic subvariety of $\Vir^*$, which we might be able to investigate using tools from affine algebraic geometry.  We will see that this is indeed the case.

From the discussion above, it is important to characterise which functions  $\chi \in \Vir^*$ have nontrivial Poisson cores.  
One striking result, proved in this paper, is  that  such $\chi$ must vanish on the central element $z$. 
Further, the induced function $\overline\chi \in W^*$ is given by evaluating {\em local} behaviour on a  proper (that is, finite) subscheme of $\mb C^\times$.  
We have:
\begin{itheorem}[Theorem~\ref{Tlocvir}]  \label{iTlocvir}
Let $\chi \in \Vir^*$. The following are equivalent:
\begin{enumerate}
    \item[$(1)$] The Poisson core of $\chi$ is nontrivial:  that is, $\CoreP(\chi) \supsetneqq (z - \chi(z))$.
    \item[$(2)$] $\chi(z) = 0$ and the induced function $\overline{\chi} \in W^*$ is a linear combination of functions of the form
    \[f\del_t\mapsto \alpha_0f(x)+\ldots+\alpha_nf^{(n)}(x)\]
    where $x \in \mb C^\times$ and $\alpha_0, \dots, \alpha_n \in \mb C$.
    \item[$(3)$] The isotropy subalgebra $\Vir^\chi$ of $\chi$ has finite codimension in $\Vir$.
    \end{enumerate}
\end{itheorem}
We call functions $\chi \in \Vir^*$ satisfying the equivalent conditions of Theorem~\ref{iTlocvir} {\em local functions} as by condition $(2)$ they are defined by local data.

Motivated by condition $(3)$ of Theorem~\ref{iTlocvir}, we 
investigate subalgebras of $\Vir$ of finite codimension.  
We prove:
\begin{itheorem}[Proposition~\ref{prop:4.15}]
\label{ithm:4.15}
Let $\mf k \subseteq \Vir$ be a subalgebra of finite codimension.  Then there is $f \in \mb C[t, t^{-1}] \setminus \{0\}$ so that $\mf k \supseteq \mb C z + f \mb C[t, t^{-1}] \del_t$.  In particular, any finite codimension subalgebra of $\Vir$   contains $z$.
\end{itheorem}

As an immediate corollary of 
Theorem~\ref{ithm:4.15}, we show: 
\begin{icorollary}[Corollary~\ref{cor:Psimple}] \label{icor:Psimple}
If $0 \neq \zeta \in \mb C$, then $\Sa(\Vir)/(z-\zeta)$ is {\em Poisson simple}:  it has no nontrivial Poisson ideals. 
\end{icorollary}

We then study the pseudo-orbits of local functions on $\Vir$, $W$, and related Lie algebras;  we describe our results for $\Vir$ in the introduction.
If $\chi \in \Vir^*$ is local, then by combining Theorem~\ref{iTlocvir} and \cite[Theorem~1.3]{IS} $\Sa(\Vir)/\CoreP(\chi)$ has polynomial growth and we thus expect the pseudo-orbit of $\chi$ to be finite-dimensional.  
We show that  pseudo-orbits of local functions in $\Vir^*$ are in fact orbits of a finite-dimensional solvable algebraic (Lie) group acting on an affine variety which maps injectively to $\Vir^*$, and we describe these orbits   explicitly 
(Section~\ref{SSexd}).
This allows us to completely determine the pseudo-orbit of an arbitrary local function in $ \Vir^*$ (Theorem~\ref{Tlfpoi}) and thus also determine the Poisson primitive ideals of $\Sa(\Vir)$ (Remark~\ref{rem:6}). 
We also classify maximal Poisson ideals in $\Sa(\Vir)$ (Corollary~\ref{cor:augmentation}):  they are  the augmentation ideal, the ideals $(z-\zeta)$ for $\zeta \in \mb C^\times$, and the defining ideals of all but one of the two-dimensional pseudo-orbits.

Through this analysis, we obtain a nice combinatorial description of pseudo-orbits in $W^*$:  pseudo-orbits of local functions on $W$, and thus Poisson primitive ideals of $\Sa(W)$, correspond to a choice of a partition $\lambda$ and a point in an open subvariety of a finite-dimensional affine space $\AA^k$, where $k$ can be calculated from $\lambda$.
(See Remark~\ref{rem:dagdag}.)  In Theorem~\ref{thm:paramprimes} and Remark~\ref{rem:paramprimesVir}, we expand this correspondence to obtain a parameterisation of all prime Poisson ideals of $\Sa(W)$ and $\Sa(\Vir)$. 
We also study the related Lie algebra $W_{\geq -1} = \mb C[t]\del_t$, and prove (Corollary~\ref{Csameprimespec}) that Poisson primitive  and prime Poisson ideals of $\Sa(W_{\geq -1})$ are induced by restriction from $\Sa(W)$.

Our understanding of prime Poisson ideals  allows us to determine  exactly which prime Poisson ideals of $\Sa(\Vir)$ obey the {\em Poisson Dixmier-Moeglin equivalence}\label{ind:PDME}, which generalises the characterisation of primitive ideals in enveloping algebras of finite-dimensional Lie algebras due to Dixmier and Moeglin.  
The central question is when a Poisson primitive ideal of $\Sa(\Vir)$ is {\em Poisson locally closed}:  that is, locally closed in the Zariski topology on Poisson primitive ideals. 
(If $\dim \mf g < \infty $ then a prime Poisson ideal of $\Sa(\mf g)$ is Poisson primitive if and only if it is Poisson locally closed \cite[Theorem~2]{LLS}.)
We show (Theorem~\ref{thm:PDME}) that $(z)$ is the only Poisson primitive ideal of $\Sa(\Vir)$ which is not Poisson locally closed.
We further prove (Corollary~\ref{Cinfo}) that $\Sa(W)$ has no nonzero prime Poisson ideals of finite height.


One part of the proof of Theorem~\ref{ithm:4.15} is to show that, given a subalgebra $\mf k$ of finite codimension in $W$, there are a finite collection of points $S$ (the ``support'' of $\mf k$) and $n\in\NN$ such that all vector fields in $\mf k$ vanish at all points of $S$ and so that $\mf k$ contains all vector fields vanishing to order $\ge n$ at every point of $S$. 
Based on this result, we classify subalgebras of $\Vir$ of codimension $\leq 3$ in Section~\ref{Svirfd}. 

Our original motivation for studying Poisson ideals of $\Sa(\Vir)$ was to study two-sided ideals in the universal enveloping algebra $\Ua(\Vir)$, and we turn to enveloping algebras in Section~\ref{SUg}.  
In the finite-dimensional setting, Kirillov's orbit method gives a correspondence between primitive ideals of $\Ua(\mf g)$ and coadjoint orbits in $\mf g^*$.  
We conjecture that a similar correspondence exists for $\Vir$ and related Lie algebras, and in subsection~\ref{SSorbit} we show that   pseudo-orbits of dimension 2 in $\Vir^*$ quantise to give a family of primitive ideals, kernels of well-known maps from $\Ua(\Vir)$ through $\Ua(W)$ to the localised Weyl algebra.
We end the paper with some conjectures about (two-sided) ideals of the universal enveloping algebra $\Ua(\Vir)$ which are motivated by our work on $\Sa(\Vir)$; these are the subject of further research. 

\smallskip

To end the introduction, let us briefly discuss the classical (continuous) version of the theory of coadjoint orbits of the Virasoro algebra.
If one considers the real Lie algebra of continuous vector fields on the circle and its central extension the real Virasoro algebra $\Vir_{\RR}$,  then the group of diffeomorphisms of the circle acts on $\Vir_{\RR}$ and its continuous dual.  There are of course notions of coadjoint orbits in this context,  see~\cite{K, W}.
However, it  happens  that the corresponding orbits do not define an interesting (Poisson) ideal in the  symmetric algebra $\Sa(\Vir_{\RR})$; in fact, this can be viewed as a somewhat informal result of our paper.
Note that the local functions which appear in Theorem~\ref{iTlocvir} and which we study in this paper can be thought of as a product of a point-based distribution with a vector field --- and hence  are very far from being continuous.

\smallskip

\noindent {\bf Acknowledgements: } We thank Jason Bell,  Ken Brown, Lucas Buzaglo, Iain Gordon, Ivan Frolov, Omar Le\'on S\'anchez, and Toby Stafford for interesting discussion and questions.  
We would also like to thank the anonymous referee for their comments, and in particular for suggesting Remark~\ref{rem:GK3two} to us.

The first author was supported by the Russian Foundation for Basic Research, grant no. 20--01--00091--a. 
The second author is  partially   
supported by  EPSRC grants  EP/M008460/1 and EP/T018844/1. 
In addition, we are grateful to the Edinburgh Mathematical Society and the University of Edinburgh School of Mathematics for supporting a 2019 mini-workshop on enveloping algebras of infinite-dimensional Lie algebras.

\section{Poisson ideals and pseudo-orbits}\label{BACKGROUND}

In this section we recall the general notions of Poisson algebra, Poisson ideal, and Poisson primitive ideal.   We  then consider  how these concepts behave for the symmetric algebras of the infinite-dimensional Lie algebras in which we are interested. 

Throughout,  let $\kk$ be an   uncountable algebraically closed field of characteristic 0.
Let us define the Lie algebras of interest in this paper.
The {\em Witt  algebra} 
$W =\CC[t, t^{-1}]\partial$ 
\label{ind:W}
is the Lie algebra of vector fields  on $\CC^\times:=\CC\setminus\{0\}$; here $\del = \del_t = \frac{d}{dt}$.
It is graded by setting $\deg t^n \del = n-1$.  
Then $W_{\ge-1}$ 
\label{ind:W-1}
is the subalgebra $\CC[t]\del$ of $W$, and $W_{\ge 1}$, 
\label{ind:W1}
sometimes called the {\em positive Witt algebra}, is  the subalgebra $t^2\CC[t]\del$; $W_{\ge0}$ stands for $t\CC[t]\del$.

The {\em Virasoro algebra } $\Vir$
\label{ind:Vir}
 is isomorphic as a  vector space to $\CC[t, t^{-1}]\del\oplus\CC z$.  It is endowed with a Lie algebra  structure  by the formula
$$[f\del + c_1 z, g\del +  c_2z ]=(fg'-f'g)\del + \Res_0(f'g''-f''g')z.$$
(Here  $\Res_0(f)$ stands for the algebraic residue of $f$ at 0, i.e. the coefficient  of $t^{-1}$ in the Laurent expansion of $f$ at 0.)
It is well-known that $\Vir$ is the unique nontrivial  one-dimensional central extension of $W$.  There is a canonical Lie algebra homomorphism  $\Vir\to W$ given by factoring out $z$.

Let $V$ be a $\kk$-vector space.  We will use 
 $\Sa(V)$ to denote the symmetric algebra of $V$; that is, polynomial functions on $V^*$.  
For $\chi\in V^*$ we denote by $\ev_\chi$ 
\label{ind:evchi}
the induced homomorphism $\Sa(V)\to\CC$  defined by $\ev_\chi(f) = f(\chi)$. 

Let $\mf g$ be one of the Lie algebras $\Vir$, $W$, $W_{\geq -1}$, or $W_{\geq 1}$.  Our assumption on the cardinality of the  field $\kk$ means that the  following extended Nullstellensatz applies to $\Sa(\mf g)$; see~\cite[Corollary~9.1.8, Lemma 9.1.2]{MR}. 
\begin{theorem}[Extended Nullstellensatz] \label{ENSS}
Let $A$ be a commutative $\kk$-algebra such that $\dim_\kk A<|\kk|.$ Then:

$\bullet$ $A$ is a Jacobson ring: every radical ideal is an intersection of a family of maximal ideals;

$\bullet$ if $\mf m\subseteq A$ is a maximal ideal then the canonical map $\kk\to A/\mf m$ is an isomorphism;

$\bullet$ if $A = \Sa(V)$ for a vector space $V$ with  $\dim V<|\kk|$, then the maximal ideals of $A$ are all of the form $\mf m_\chi=\Ker {\rm ev}_\chi$ for some $\chi \in V^*$. 
\end{theorem}

Let $A$ be a commutative $\kk$-algebra. (Note that we make no noetherianity or finiteness assumption on $A$.)
Denote by $\MSpec A$
\label{ind:Mspec}
 the set of maximal ideals of $A$, which we consider  as a (potentially infinite-dimensional) variety.  
We denote by $\mathfrak{m}_x$ the maximal ideal corresponding to a point $x \in {\rm MSpec}A$. 
\label{ind:mchi2}
As usual,  $\MSpec A$  is a topological space under the Zariski topology. 
Given an ideal $N$ of $A$, we denote the corresponding closed subset of $\MSpec A$ by 
\[ V(N) := \{ \mf m \in \MSpec A \ |\  N \subseteq \mf m\}.\]
\label{ind:V}
Given $X \subseteq \MSpec A$, we denote the corresponding radical ideal of $A$ by
\[ I(X) := \bigcap_{x \in X} \mf m_x.\]
\label{ind:I}

As we have enlarged the class of varieties somewhat beyond the usual, we make the convention that for us an {\em algebraic variety} is a classical variety:  a (nonempty) integral separated scheme of finite type over $\kk$, as in \cite[p. 105]{Hartshorne}. 
If $A, B$ are commutative $\CC$-algebras which are domains, an algebraic map or {\em morphism of varieties} $\phi: \MSpec A \to \MSpec B$ has the usual meaning:  a function so that the pullback $\phi^*$ defines an algebra homomorphism $B \to A$.

Let $\{\cdot\ , \cdot\}: A\times A\to A$ be a skew-symmetric $\CC$-bilinear map. 
We say that $(A, \{\cdot\ , \cdot\})$ is a {\it Poisson algebra}
\label{ind:Poissonalg}
 if $\{\cdot\ , \cdot\}$ satisfies the Leibniz rule on each input and the Jacobi identity.  
An ideal $I$ of $A$ is {\it Poisson} if $\{I, A\}\subseteq I$. 
\label{ind:Poissonideal}
It is clear that the sum of all Poisson subideals of any ideal $I$ of $A$ is the maximal Poisson ideal inside $I$; we denote this ideal by ${\rm Core}(I)$ and refer to it as the {\em Poisson core} 
\label{ind:poissoncore2}
of $I$. Note that if $I$ is radical, respectively prime, then ${\rm Core}(I)$ inherits this property; see \cite[Lemmata~2.6 and 2.8]{PS}.

A Poisson ideal $I$ is called {\it Poisson primitive}
\label{ind:poissonprim2}
 if $I={\rm Core}(\mf m)$ for a maximal ideal $\mf m$ of $A$; Poisson primitive ideals are prime.
We denote the set of Poisson primitive ideals of $A$ by  $\PSpec_{\rm prim}A$, 
\label{ind:Pspec}
and the set of prime Poisson ideals by $\PSpec A$.  Both are given the Zariski topology, where the closed subsets are defined by Poisson ideals of $A$.

Consider a Lie algebra $\mf g$ with $\dim \mf g < |\kk|$. 
It is well known that $\Sa(\mf g)$ possesses a canonical Poisson algebra structure, induced by defining $\{u, v\} = [u, v]$ for any $u, v \in \mf g$. 
As $\dim\mf g<|\CC|$ 
then by the Nullstellensatz $\MSpec(\Sa(\mf g))$ can be canonically identified with~$\mf g^*:$
$$\chi\in\mf g^*\leftrightarrow \mathfrak{m}_\chi:=\Ker {\rm ev}_\chi\in \MSpec(\Sa(\mf g)).$$
Thus any  Poisson primitive ideal of $\Sa(\mf g)$ is equal to  $\Core(\mathfrak{m}_\chi)$ for some $\chi\in\mf g^*$.
Set $\CoreP (\chi):=\Core(\mathfrak{m}_\chi)$.
\label{ind:poissoncore3}
For any Lie algebra $\mf g$, the ideal $\mathfrak m_0 = \Core(\mathfrak{m}_0)=\CoreP(0)$ (the augmentation ideal)  is the Poisson core of $0 \in \mf g^*$ and so is Poisson primitive.

Even in the absence of an adjoint group to $\mf g$, the Poisson primitives  $\CoreP(\chi)$ give analogues of coadjoint orbits.  

\begin{definition} \label{def:pseudoorbit}
Let $\mf g$ be any Lie algebra. 
The {\em pseudo-orbit} of $\chi \in \mf g^*$ is
\[ \mb O(\chi) :=  \{ \nu \in \mf g^* | \CoreP(\nu) = \CoreP(\chi)\} .\]
The {\em dimension} of $\mb O(\chi)$ is defined to be $\GK \Sa(\mf g)/\CoreP(\chi)$.  (Here if $R$ is a $\kk$-algebra then $\GK R$ denotes the {\em Gelfand-Kirillov dimension} of $R$; see \cite{KL}.)
\end{definition}

\begin{remark}\label{rem:orbit}
If $\mf g = {\rm Lie} \ G$ is the Lie algebra of a connected algebraic group and $\chi \in \mf g^*$, then $\mb O(\chi)$ is the coadjoint orbit of $\chi$.  
In our setting, $\mf g$ is not the Lie algebra of any algebraic group.  
However, we will see that we can still define algebraic group actions on pieces of $\mf g^*$ that allow us to recover pseudo-orbits as actual orbits.
\end{remark}

Part (b) of the next result is an analogue of \eqref{orbit} for pseudo-orbits.

\begin{lemma}\label{lem:basics}
Let $\mf g $ be a Lie algebra with $\dim \mf g < | \kk | $.  
\begin{itemize}
    \item[(a)]
 Any radical Poisson ideal in $\Sa(\mf g) $ is equal to an intersection of Poisson primitive ideals.
Explicitly, given a radical Poisson ideal $I$ we have
\[ I = \bigcap \{\CoreP(\chi) \ | \ \chi \in \mf g^*,~\ev_\chi(I) = 0\} .\]
\item[(b)]
Assume now that $\dim \mf g$ is countable,
and let $\chi \in \mf g^*$ be  nonzero.
Then
\[ \CoreP(\chi) = \bigcap_{\nu \in \mb O(\chi)} \mf m_\nu.\]
\end{itemize}
\end{lemma}
\begin{proof}
(a) is well-known, but we give a proof for completeness.
 As $I$ is radical and Poisson, 
 $$I \subseteq \bigcap \{  \CoreP(\chi) \ | \ \chi \in \mf g^*,~\ev_\chi(I) = 0\},$$ 
 and this is contained by definition in $\bigcap_{\chi \in \mf g^*, \ev_\chi(I) = 0} \mf m_\chi $. 
 By the  Nullstellensatz, this last is equal to~$I$.  
 
For (b), set $R:=\Sa(\mf g)/\CoreP(\chi)$. 
Thanks to the proof of~\cite[Theorem 6.3]{LSS}, which uses only that $\mf g$ is countably generated (see also~\cite[Theorem 3.2]{BLLSM}), there is a sequence of non-zero Poisson ideals $L_1, L_2, \ldots$ of $R$ such that if  $\mu \in \mf g^*$ with $\CoreP(\mu)$ strictly containing $\CoreP(\chi)$, then $\CoreP(\mu)/\CoreP(\chi)$ contains $L_i$ for at least one $i$. 

This is equivalent to the following statement: there is a sequence $f_1, f_2, \ldots\in R \setminus \{0\}$
such that if $\CoreP(\mu)$ strictly contains $\CoreP(\chi)$ then $f_i(\mu)=0$ for at least one $i$. 
Therefore $R'=(\Sa(\mf g)/\CoreP(\chi))[f_1^{-1}, f_2^{-1},\ldots]$ is Poisson simple. 
On the other hand $R'$ is clearly at most countable-dimensional and hence $(0)\subset R'$ is an intersection of a family of maximal ideals of $R'$ by the Nullstellensatz. The definition of $R'$ and $f_i'$ guarantees that $\CoreP(\mu)=\CoreP(\chi)$ for each $\mu \in \mf g^*$ defining such an ideal.
We thus obtain a family of maximal ideals of $\Sa(\mf g)$
contained in the pseudo-orbit of $\chi$ whose intersection is $\CoreP(\chi)$.  
This implies the desired result. 
\end{proof}

\subsection{Some results on pseudo-orbits} \label{SSfrpo}
We now give several results which are well known for  finite-dimensional algebraic Lie algebras, but  require  proof in our setting.   Throughout this subsection, let $\mf g$ be a Lie algebra with $\dim \mf g < |\kk|$.

For every $\chi \in\mf g^*$ define the skew-symmetric bilinear form $B_\chi(x, y):=\chi([x, y])$.
\label{ind:Bchi}
The kernel of $B_\chi$ is a Lie subalgebra of $\mf g$, which we denote by $\mf g^\chi$.
\label{ind:isotropy}
Observe that $\mf g^\chi$ is precisely $\{v \in \mf g \ |\  v \cdot \chi = 0\}$; that is $\mf g^\chi$ is the isotropy subalgebra of $\chi$ under the (coadjoint) action of $\mf g$ on $\mf g^*$. 
As $B_\chi$ induces a nondegenerate skew-symmetric form on $\mf g / \mf g^\chi$, we have $\rk B_\chi = \dim \mf g/\mf g^\chi$.

\begin{lemma}\label{Lrk+} The dimension of  $\mb O(\chi)$ is at least $\rk B_\chi$, i.e. $${\rm GKdim}({\rm S}(\mf g)/\CoreP(\chi))\ge \rk B_\chi.$$\end{lemma}

Before proving Lemma~\ref{Lrk+}, we establish some notation, which we will need for several results.  For $u_1,\ldots, u_n,$ $ v_1, \ldots, v_n\in \mf g$ set $D(u_1,  \ldots, u_n; v_1,  \ldots, v_n)\in\Sa(\mf g)$ 
\label{ind:D}
to be the determinant 
\begin{equation}\begin{vmatrix}[u_1, v_1]&\cdots&[u_1, v_{n}]\\
\cdots&\cdots&\cdots\\
{[u_{n}, v_1]}&\cdots&{[u_{n}, v_{n}]}\\
\end{vmatrix}. \label{Ednp1}\end{equation} 
\begin{proof}[Proof of Lemma~\ref{Lrk+}]
Pick $r\in\NN$ such that $r\le \rk B_\chi$. 
Then there is an $r$-dimensional subspace $V$ of $\mf g$ so that 
 $\rk ( B_\chi|_{V})=r$;
 that is, if $u_1, \ldots, u_r$ is  a basis of $V$ then $\ev_{\chi}(D(u_1, \dots, u_r, u_1, \dots, u_r)) \neq 0$.
We will show that  $u_1, \ldots, u_r$ are algebraically independent in ${\rm S}(\mf g)/\CoreP(\chi)$; that is, that $\GK (\Sa(\mf g)/\CoreP(\chi) ) \geq r$.

Assume to the contrary that 
the $u_i$ are not algebraically independent modulo $\CoreP(\chi)$.  
This means that there is some nonzero $P \in \kk[u_1, \dots, u_r] \cap \CoreP(\chi)$, and we may assume $P$ is of minimal degree among such elements.
As $\CoreP(\chi)$ is Poisson, 
\beq \label{biff}
\{ u_i, P \} = \sum_j \{u_i, u_j\} \del_j P \in \CoreP(\chi),
\eeq
where $\del_j P = \frac{\del P}{\del u_j}$.

Let $Q$ be the field of fractions of $\Sa(\mf g)/\CoreP(\chi)$ (recall that $\CoreP(\chi)$ is prime).
We may rewrite \eqref{biff} as the matrix equation
\[ \Bigl( \{u_i, u_j\} \Bigr)_{ij} \begin{pmatrix} \del_1 P \\ \vdots \\ \del_r P \end{pmatrix} = \vec{0}\]
over $Q$.  By minimality of $\deg P$, the vector $\begin{pmatrix} \del_1 P \\ \vdots \\ \del_r P \end{pmatrix}  \in Q^r$ is nonzero.  
Thus
$\Bigl( \{u_i, u_j\} \Bigr)_{ij} \in M_{r\times r}(Q)$ is singular, so
$D(u_1, \dots, u_r; u_1, \dots, u_r) \in \CoreP(\chi) \subseteq \mathfrak{m}_\chi$.  This contradicts the first paragraph of the proof.
\end{proof}

It is not necessarily easy to calculate  $\CoreP(\chi)$, but we can sometimes use GK-dimension to show that a Poisson ideal contained in $\mathfrak{m}_\chi$ is in fact equal to $\CoreP(\chi)$.  This is given by the following lemma. 

\begin{lemma} \label{Lprtop}
Let $\mf g $ be a Lie algebra and $A$ a commutative $\kk$-algebra that is a domain.  Let ${\phi}: \Sa(\mf g) \to A$ be an algebra homomorphism so that $\ker {\phi} $ is Poisson.  
Let $\chi \in \mf g^*$ be such that 
\begin{equation}\Ker\phi\subseteq \mathfrak{m}_\chi\mathrm{~and~}\rk B_\chi \geq \GK A.\label{Epo}\end{equation}
Then $\Ker\phi= \CoreP(\chi)$ and thus is Poisson  primitive; further, $\rk B_\chi = \GK A = \dim \OO(\chi)$.\end{lemma}
\begin{proof}
Certainly $\ker {\phi} \subseteq \CoreP(\chi)$.  
If the containment is strict, then as $A$ is a domain, 
$$\GK(\Sa (\mf g) / \CoreP(\chi)) < \GK A \leq \rk B_\chi,$$ contradicting
Lemma~\ref{Lrk+}.
The final equality is direct from Lemma~\ref{Lrk+}.
\end{proof}

\begin{corollary}\label{Cpodscr}
If both $\chi_1, \chi_2$ satisfy \eqref{Epo} for the same ${\phi}$ then $\CoreP(\chi_1)=\CoreP(\chi_2)$. \qed \end{corollary}

We will use Lemma~\ref{Lprtop} and Corollary~\ref{Cpodscr} to describe primitive ideals in $\Sa(\mf g)$ and the corresponding pseudo-orbits  in $\mf g^*$ explicitly. 
In the situation of Lemma~\ref{Lprtop}, $A$ may not necessarily be a Poisson algebra.
We abuse notation slightly, however, and say that if $\ker {\phi}$ is a Poisson ideal,
then ${\phi}$ is a {\em Poisson morphism}.  
\label{ind:poissonmor}
Further, 
given a morphism of varieties 
\[\psi: \MSpec(A) \to \MSpec(\Sa(\mf g)) = \mf g^*,\]
 we say that $\psi$ is {\em Poisson} if $\psi^*:  \Sa(\mf g) \to A$ is Poisson.

It is important to know under which conditions maps  are Poisson and the answer is given in Proposition~\ref{Prpoi}. 
To explain the setup we need to use some concepts from (affine) algebraic geometry.  
As mentioned, we will, somewhat loosely, refer to infinite-dimensional vector spaces such as $\mf g^*$ as  varieties, enlarging the class  from  standard usage.  
For any variety $X$ and any point $x\in X$ let ${\rm T}_xX = (\mathfrak{m}_x/\mathfrak{m}_x^2)^*$ 
\label{ind:tangent}
denote the tangent space  to $X$ at $x$; this definition makes sense for $X= \mf g^*$ as well.  We will, without further comment,  identify ${\rm T}_x X$ with elements of $\mathfrak{m}_x^*$ which vanish on $\mathfrak{m}_x^2$.
If there is a map $\psi: X\to Y$ between algebraic varieties $X$ and $Y$ then, for all $x\in X$, there is an induced map $\psi_x: {\rm T}_xX\to {\rm T}_{\psi(x)}Y$. 

We canonically identify the tangent space to $\mf g^*$ at $\chi\in\mf g^*$ with $\mf g^*$. 
For any derivation $D$ of $\Sa(\mf g)$ and any $\chi\in\mf g^*$ we denote by $D_{\chi}$
\label{ind:Dchi}
 the tangent vector defined by $D$ at $\chi$; 
that is $D_\chi(f) = \chi(Df)$ for any $f \in \mathfrak{m}_\chi$.
Now,  $\mf g$ acts by derivations on $\Sa(\mf g)$.  For $u \in \mf g$ and $\chi\in\mf g^*$ let $u_\chi$ 
be the tangent vector defined by $u$ at $\chi$.  
\label{ind:uchi}
If $v \in \mf g$ then $u_\chi(v) = \chi([u, v])$, and  thus $u_\chi \in {\rm T}_\chi(\mf g^*) = \mf g^*$ is identified with $ u \cdot \chi$.   
If $\mf u \subseteq \mf g $ let $\mf u_\chi = \{ u_\chi \ | \ u \in \mf u\}$.
Thus $\mf g_\chi \cong \mf g \cdot \chi \cong \mf g/ \mf g^\chi$ and $\rk B_\chi = \dim \mf g_\chi$.

\begin{proposition}\label{Prpoi}
Let $\mf g $ be a Lie algebra with $\dim \mf g < |\kk|$.
Let $X$ be an affine variety and let $\phi:  X \to \mf g^* $ be a morphism of varieties. 
Then $\phi$ is Poisson if and only if $\mf g_{\phi (x)}\subseteq \phi_x({\rm T}_xX)$ for all $x\in X$.
\end{proposition}
 Proposition~\ref{Prpoi} is a  direct consequence of the following lemma.
\begin{lemma}Let $\mf g, X, \phi$ be as in the statement of Proposition~\ref{Prpoi} and let $D$ be a derivation of $\Sa(\mf g)$. 
Then $\ker \phi^*$ is $D$-stable if and only if $D_{\phi(x)}\in \phi_x({\rm T}_xX)$ for all $x\in X$.\end{lemma}
\begin{proof}
Suppose that $D_{\phi(x)} \in \phi_x({\rm T}_xX)$ for all $x\in X$.
Let $f \in \ker \phi^*$.  
We show that $\phi^*(Df) = 0$, or, equivalently, that $Df \in \mathfrak{m}_{\phi(x)}$ for all $x \in X$.
Fix $x$, and let $\ell \in {\rm T}_x X$ be such that $D_{\phi(x)} = \phi_x(\ell)$.
Then $$(Df)(\phi(x)) = \ell(\phi^*f) = \ell(0) = 0,$$
as needed.

Conversely, let $x \in X$ and let $\mf n_x := \mathfrak{m}_x \cap \phi^*(\Sa(\mf g))$.
If $\ker \phi^*$ is $D$-stable, then $D$ induces a derivation $\overline{D}$ on $\phi^*(\Sa(\mf g))$, defined by 
\[ \overline D(\phi^*f) = \phi^*(Df),\]
and thus defines an element $\overline D_x \in (\mf n_x/\mf n_x^2)^*$.  
Let $\ell \in {\rm T}_xX$ be any extension of $\overline D_x$ to $\mathfrak{m}_x/\mathfrak{m}_x^2$.
Then for $f \in \mathfrak{m}_{\phi(x)}$ we have
\[ D_{\phi(x)}(f) = \ev_{\phi(x)}(Df) = {\phi}^*(Df)(x) = \overline{D}_x  \phi^*f = \ell  \phi^*f,\]
showing that $D_{\phi(x)} \in \phi_x({\rm T}_xX)$.
\end{proof}

\subsection{Some pseudo-orbits in $W^*$ and $W_{\geq-1}^*$}
\label{SSexample}

This subsection is effectively an extended example, where we use the methods of the previous subsection, particularly Lemma~\ref{Lprtop}, to compute the Poisson cores of some particular functions in $W^*$ and $W_{\geq-1}^*$.
We will see later (Proposition~\ref{PGK2}) that these give all of the prime Poisson ideals of $\Sa(W)$ and $\Sa(W_{\geq-1})$ of co-GK-dimension 2.

Throughout the subsection fix $x,\alpha, \gamma \in \kk$ with $x \neq 0$ and $\alpha, \gamma$ not both 0.
Let $\chi := \chi_{x;\alpha,\gamma} \in W^*$ be defined by
\[  f \del \mapsto \alpha f(x) + \gamma f'(x).\]
Further, given $g\in \kk[t,t^{-1}]$ let $W(g) :=  \kk[t,t^{-1}] g \del \subseteq W$.
\label{ind:Wg}
Both notations will be generalised and used more extensively in Section~\ref{Spridlf}. 

We first compute the isotropy subalgebra of $\chi$.

\begin{lemma}\label{lem:Wchi}
We have:
\[ W^\chi = \begin{cases} W((t-x)^2) & \gamma = 0 \\
  \{ g \del \ | \ g(x) = \alpha g'(x) + \gamma g''(x) = 0 \} & \gamma \neq 0.
  \end{cases} \]
\end{lemma}
\begin{proof}
Recall that $\chi$ defines a bilinear form $B_\chi$ on $W$ by $B_\chi(v, w) = \chi([v, w])$, 
and that $W^\chi = \ker B_\chi$.  For all $\chi$---that is, for all choices of $x, \alpha, \gamma$---therefore, $W^\chi \neq W$.

First assume that $\gamma = 0$.  
If $g\del \in W((t-x)^2)$ 
then $$\chi([g\del,f\del]) = \alpha(g(x) f'(x) -f(x) g'(x)) = 0,$$ so $W((t-x)^2) \subseteq W^\chi$.  But  $B_\chi$ defines a nondegenerate bilinear form on $W/W^\chi$, so $\dim W/W^\chi \geq 2$ and thus  $W((t-x)^2) = W^\chi$.

Now suppose $\gamma \neq 0$.  
Let
\[ V:= \{ g \in \kk[t,t^{-1}]\  |\  g(x) = \chi(g'\del) = 0 \},\]
which is a codimension 2 subspace of $\kk[t,t^{-1}]$, 
and let $g \in V$, $f\in \kk[t,t^{-1}]$.  
Then 
\[
\chi([f\del, g\del]) = \alpha (f(x) g'(x) - f'(x) g(x))+ \gamma(f(x) g''(x) - f''(x) g(x))   = f(x) \chi(g'\del) - g(x) \chi(f'\del),
\]
which is 0 by assumption on $g$.
Thus $W^\chi \supseteq V\del$, and as before  the two must be equal.
\end{proof}

Note that in all cases $W^\chi \supseteq W((t-x)^3)$,
and that  if $\lambda \neq 0$ then $W^{\chi} = W^{\chi_{x; \lambda \alpha , \lambda \gamma}}$. 

We now compute $\CoreP(\chi)$.
Let $B = \CC [t,t^{-1}, y]$, and define a Poisson bracket on $B$  induced from
defining $\{y, t\} = 1$.

Define
$p_\gamma:  \Sa(W) \to B$ 
as the algebra homomorphism induced by defining 
\beq\label{pgamma} p_\gamma( f \del) = f y + \gamma f'.
\eeq
We verify:
\[    \{ p_\gamma(f\del), p_\gamma(g\del)\} =
    \{fy+\gamma f', gy+ \gamma g' \} 
    = y(fg'-f'g) + \gamma(fg''-f''g) = p_\gamma((fg'-f'g)\del).
\]
Thus $p_\gamma$ 
respects Poisson brackets, so 
 $\ker p_\gamma$
 is a Poisson ideal of $\Sa(W)$.


\begin{lemma}\label{lem:Jgamma}
The Poisson core of $\chi$ is equal to $\ker p_\gamma$, and in particular depends only on $\gamma$ as long as $(\alpha , \gamma) \neq (0,0)$.
\end{lemma}

\begin{proof}
First, $\chi(f\del) = p_\gamma(f \del)|_{t=x,y=\alpha}$, and it follows immediately that if we extend $\chi$ to a homomorphism $\ev_\chi:  \Sa(W) \to \kk$, it factors through $p_\gamma$. 
As a result,  $\ev_\chi(\ker p_\gamma) = 0.$
Since in all cases $W^\chi$ has codimension $ 2 = \GK(B)$ by Lemma~\ref{lem:Wchi}, the result is a direct consequence of Lemma~\ref{Lprtop}.
\end{proof}

\begin{remark}\label{rem:restrictJgamma}
Let $x, \alpha, \gamma \in \kk$ with $\alpha, \gamma$ not both 0, and define $\nu  \in W_{\geq-1}^*$ analogously to $\chi$:  that is, $\nu (f\del) = \alpha f(x) + \gamma f'(x)$.
One may similarly prove that $\CoreP(\nu) =  \Sa(W_{\geq-1}) \cap \ker p_\gamma$, and in particular that
\[\CoreP(\chi|_{W_{\geq-1}}) = \CoreP(\chi) \cap \Sa(W_{\geq-1}).\]
We will see in Proposition~\ref{Plfpoi} that this is true for all elements of $ W^*$.
Likewise, $\CoreP(\chi|_{W_{\geq 1}}) = \CoreP(\chi) \cap \Sa(W_{\geq 1})$.
\end{remark}

\subsection{Pseudo-orbits versus orbits}
We wish to relate the pseudo-orbits from  subsection~\ref{SSfrpo} to  orbits of an algebraic group acting on an appropriate algebraic variety $X$. 
The next result gives us a general technique to do this.

\begin{proposition}\label{Practpoi} 
Let $\mf g$ be  a Lie algebra with $ \dim \mf g < | \kk |$.
Let $X$ be an irreducible affine algebraic variety acted on by a connected algebraic group $H$ with Lie algebra $\mf h$ and let $U\subseteq X$ be an open affine subset. 
Fix a morphism of varieties $\phi: U \to \mf g^*$. 
Assume that 
for every $x\in U$ we have $\mf g_{\phi(x)}\subseteq \phi_x(\mf h_x)$. 
\begin{itemize}
\item[(a)] For all $x\in U$ the pseudo-orbit of $\phi(x)$ is contained in $\overline{\phi(Hx\cap U)}$, where recall that the topology on $\mf g^*$ is the Zariski topology.
\item[(b)] For $x\in U$, let $I^U(Hx)$ be  the defining ideal in $\kk[U]$ of $\overline{Hx} \cap U$.
Let $x\in U$ be such that $\dim \mf g_{\phi(x)}=\dim \mf h_x$.
Then $\CoreP(\phi(x))$ is equal to the kernel of the induced homomorphism 
\beq \label{homom} \Sa(\mf g) \to \kk[U]/I^U(Hx) = \kk[\overline{Hx} \cap U],
\eeq
and 
\beq\label{EGK}
\dim \OO(\phi(x)) = \GK \Sa(\mf g)/\CoreP(\phi(x)) = \dim \mf g_{\phi(x)}.
\eeq

In particular, if $y \in U$ is such that $Hx = Hy $ and 
 $\dim \mf g_{\phi(y)}=\dim \mf h_y$ then $\CoreP(\phi(x))=\CoreP(\phi(y))$.
\item[(c)] Let $x, y\in U$ be such that $\dim \mf g_{\phi(x)}=\dim \mf h_x$ and $\dim \mf g_{\phi(y)}=\dim \mf h_y$. Then $\CoreP(\phi(x))=\CoreP(\phi(y))$ if and only if there are open subsets $U_x\subseteq Hx\cap U$ and $U_y\subseteq Hy\cap U$ such that $\phi(U_x)=\phi(U_y)$.
\end{itemize}
\end{proposition}
\begin{proof} 
(a). 
The kernel of \eqref{homom} is Poisson by Proposition~\ref{Prpoi} and is contained in $\mathfrak{m}_{\phi(x)}$ by definition.  
Thus it is contained in $\CoreP(\phi(x))$, which is what we need. 
Note also that this statement is completely trivial if $\mf g$ is the Lie algebra of a finite-dimensional algebraic group.

(b). 
Let $K$ be the kernel of \eqref{homom}.  
Then $K$ is Poisson by (a).  
We have
\[ \GK \Sa(\mf g)/K \leq \dim Hx = \dim \mf h_x = \rk B_{\phi(x)},\]
so $K = \CoreP(\phi(x))$ by Lemma~\ref{Lprtop}.  
Certainly 
\[\GK \Sa(\mf g)/\CoreP(\phi(x)) \leq \GK \kk[U]/I^U(Hx) = \dim \mf h_x = \dim \mf g_{\phi(x)},\]
and the two are equal by Lemma~\ref{Lprtop}.
The final statement follows from Corollary~\ref{Cpodscr}.

(c). 
For any dense subset $U_x \subseteq Hx \cap U$, by (b) $\CoreP(\phi(x))$ is equal to the kernel of the induced map $\Sa(\mf g) \to \kk[U_x]$.
This is determined by $\phi(U_x)$, so if $\phi(U_x) = \phi(U_y)$ then $\CoreP(\phi(x)) = \CoreP(\phi(y))$.

Suppose now that $\CoreP(\phi(x)) = \CoreP(\phi(y))$. 
Consider the induced maps $Hx\to \mf g^*, Hy\to \mf g^*$ and the respective fibre product $(Hx)\times_{\mf g^*}(Hy)$. 
Note that $(Hx)\times_{\mf g^*}(Hy)$ is a closed subset of $(Hx)\times (Hy)$ and hence the ideal defining $(Hx)\times_{\mf g^*}(Hy)$ is generated by a finite collection of elements. 
This implies that there is a finite-dimensional subspace $V\subset \mf g$ such that $$(Hx)\times_{\mf g^*}(Hy)=(Hx)\times_{V^*}(Hy).$$
As $\CoreP(\phi(x)) =\CoreP(\phi(y))$, the images of $Hx$  and $Hy$ in $V^*$ have the same closure in $V^*$; call it $Z$.
The image of $Hx$ in $V^*$  may not be open in $Z$, but it is constructible and thus contains an open subset $U(x)$ of $Z$.  
Likewise the image of $Hy$ in $V^*$ contains an open subset $U(y)$ of $Z$.
We pick $U_x $ to be the preimage of $U(x) \cap U(y)$ in $Hx$ and let $U_y$ be the preimage of $U(x) \cap U(y)$ in $Hy$; then they satisfy the desired property. 
\end{proof}

\section{Primitive ideals and local functions}\label{Spridlf}
We now specialise to let $\mf g$ be one of $W$,  $W_{\geq 1}$, $W_{\geq -1}$, or $\Vir$. 
The main goal of this section is to determine which functions $\chi \in \mf g^*$ have  nontrivial Poisson core $\CoreP(\chi)$.
We will see that these $\chi$ are precisely those $\chi$ which measure the local behavior of $f\del \in \mf g$ at a finite collection of points; we  call these functions {\em local}.  (For example, the functions $\chi \in W^*$ of Subsection~\ref{SSexample} are local.)
We  provide several equivalent characterisations of local functions, and  apply these to classify subalgebras of $\mf g$ of finite codimension.
As a consequence, we  show that $\Sa(\Vir)/(z-\zeta)$ is Poisson simple for $\zeta \neq 0$. 

More formally, consider $\mf g = W_{\geq -1}$.
Let $x, \alpha_0, \dots,\alpha_n \in \kk$  with $(\alpha_0, \dots, \alpha_n) \neq \vec{\hspace{2pt}0}$.  
Define a linear function $\chi_{x; \alpha_0, \dots, \alpha_n} \in W_{\geq -1}^*$ by
\begin{equation}\chi_{x; \alpha_0,\ldots, \alpha_n}: f\del\mapsto \alpha_0f(x)+\alpha_1 f'(x)+\ldots+\alpha_nf^{(n)}(x).  
\label{Elf}\end{equation}
The same formula defines elements of $W_{\ge1}^*$ and $W^*$, although in the last case we need to require that $x\ne0$.

We now formally define local functions.
\begin{definition}\label{def:local}
\begin{enumerate}
\item[(a)] A {\em local function on $W_{\ge -1}$ or $W_{\ge 1}$} 
is a  sum of finitely many  functions of the form \eqref{Elf} with (possibly) distinct $x$;
\item[(b)] A {\em local function on $W$} is a  sum of finitely many functions of the form \eqref{Elf} with (possibly) distinct $x\ne0$;
\item[(c)] A {\em local function on $\Vir$} is the pullback of a local function on $W$ via the canonical map  $\Vir \to W$.
\end{enumerate}
A local function of the form \eqref{Elf} is called a {\em one-point local function}.
Let $\chi = \chi_{x; \alpha_0,\ldots, \alpha_n}$ be a one-point local function.
We say that $\{x\}$ is the {\em support}
 of  $\chi$ and  that $x$ is the {\em base point} of $\chi$.  
If $\alpha_n  \neq 0$ we say that $n$ is the {\em order} of $\chi$.

Let $\chi$ be  an arbitrary local function. The  {\em support} 
\label{ind:support}
of $\chi$ is the union of the supports of the component one-point local functions. 
Further, the orders of the component one-point local functions give rise to a partition $\lambda(\chi)$. 
More explicitly, 
write $\chi = \chi_1 + \dots + \chi_{r}$, where the $\chi_i $ are one-point local functions based at distinct points.
Let $m_i$ be the order of $\chi_i$.  
By reordering the $\chi_i$ if necessary, we may assume that $m_1 \geq m_2 \geq \dots \geq m_r$.
The partition
\[ \lambda(\chi) := (m_1+1, \dots, m_r+1)\]
\label{ind:lambda}
is called the {\em order partition} of $\chi$.
(We add 1 here so that the partition $(0)$ corresponds to the zero function.)  
We call $m_1$ the  {\em order} of $\chi$.
\label{ind:order}
\end{definition}

It follows from the Chinese remainder theorem  that a local function is 0 if and only if it is 0 pointwise and thus any  local function $\chi$ on $W$, $W_{\geq -1}$, or $\Vir$ has a unique presentation as a sum of 
nonzero one-point local functions with distinct base points.
This also shows that the  partition $\lambda(\chi)$ and the order of $\chi$ are well-defined.
For $W_{\ge1}$ it is easy to see that $\chi_{0; 1}=\chi_{0;0, 1}=0$ and the presentation is unique under the assumption that the coefficients  of  $f(0)$ and  $f'(0)$ are 0.

\begin{remark}\label{rem:dag}
Let $\mf g$ be $\Vir$ or $W$ or $W_{\geq -1}$ or $W_{\geq 1}$.  
Then $\mf g^*$ and the subspace of local functions are both uncountable-dimensional vector spaces;
for local functions observe that any set of one-point local functions with distinct base points is linearly independent. 
On the other hand, clearly ``most'' elements of $\mf g^*$ are not local.
In fact, we will see in  Remark~\ref{rem:countableU} that local functions are parameterised by a countable union of algebraic varieties.  

For a specific example of a non-local function, let $\alpha_0, \alpha_1, \dots \in \kk$ be algebraically independent over $\QQ$, and define  $\varkappa \in W_{\geq -1}^*$ by 
$\varkappa (t^i\del) = \alpha_i$.  
\end{remark}
As local functions are defined similarly for $W$, $ W_{\ge 1}$,  $W_{\ge -1}$, and $\Vir$ we sometimes discuss all of them simultaneously. 
When we do so, we will assume without comment  whenever we talk about a one-point local function on $W$ or $\Vir$ that $x\ne0$. 

Pick $\chi\in W^*$ or $W_{\ge1}^*$ or $W_{\ge-1}^*$. 
In this section we will show that $\CoreP(\chi)$ is nonzero iff $\chi$ is local and prove a similar statement for $\Vir^*$. 
The starting point here is the following result, due in its strongest form to Iyudu and the second author. 
\begin{theorem}\label{Tgkf} Let $\mf g$ be $W$ or $W_{\ge1}$ or $W_{\ge-1}$ and let $I$ be a nonzero Poisson ideal of $\Sa(\mf g)$.
Then 
\[\GK (\Sa(\mf g))/I < \infty.\]
In particular, if $\chi\in\mf g^*$ is such that $(0) \neq \CoreP(\chi)$, then 
 $\GK (\Sa(\mf g)/\CoreP(\chi))<\infty$.

Further, if $\chi \in \Vir^*$ is such that $\CoreP(\chi) \neq (z - \chi(z))$, then 
\[ \GK (\Sa(\Vir)/\CoreP(\chi))< \infty.\]
\end{theorem}
\begin{proof}See~\cite[Theorem~1.4]{PS} for $W_{\geq 1}$ and~\cite[Theorem~1.3]{IS} for all other $\mf g$.\end{proof}

Theorem~\ref{Tgkf} has the following extremely useful consequence:
\begin{corollary}\label{cor:finiteorbit} Let $\mf g$ be as above and let $\chi \in \mf g^*$.
If $\mf g = W$ or $W_{\geq-1}$ or $W_{\geq 1}$ assume that $\CoreP(\chi) \neq (0)$, and if $\mf g = \Vir$ assume that $\CoreP(\chi) \neq (z - \chi(z))$.
Then  $\dim \mf g\cdot\chi<\infty$.
\end{corollary}
\begin{proof}
Combine Theorem~\ref{Tgkf} and Lemma~\ref{Lrk+}.
\end{proof}

We also recall:
\begin{proposition}\cite[Corollary~5.1]{LSS}
\label{prop9}
Let $\mf g $ be $\Vir$ or $W$ or $W_{\geq-1}$ or $W_{\geq 1}$.
Then $\Sa(\mf g)$ satisfies the ascending chain condition on radical Poisson ideals and every Poisson ideal has finitely many minimal primes above it, each of which is Poisson.
\end{proposition}

It is not known for any of these Lie algebras whether $\Sa(\mf g)$ satisfies the ascending chain condition on arbitrary Poisson ideals. 

Although we use similar notation for $W_{\geq -1}$, $W_{\geq 1}$, $W$, and $\Vir$, the details are slightly different, so we analyse local functions in each of these cases separately.

\subsection{Local functions on $W_{\ge-1}$ and $W_{\ge1}$} \label{SSloc}
In this subsection we set $\mf g=W_{\ge-1}=\CC[t]\del$. 
It will be useful to consider Lie subalgebras of $W_{\geq -1}$ of a particular form.  
For any $f\in\kk[x] \setminus \{0\}$ denote by $W_{\ge-1}(f)$ 
\label{ind:W-1f}
the space of vector fields of the form $$\{gf\del\}_{g\in\CC[x]}.$$
In other words, $W_{\geq -1}(f) = f W_{\geq -1}$ under the obvious notation.
It is clear that $W_{\ge-1}(f)$ is a Lie subalgebra of $W_{\ge-1}=W_{\ge-1}(1)$. 

We give five equivalent conditions for local functions.  Similar conditions will hold for the other Lie algebras we consider, see Theorems~\ref{Tloc} and~\ref{Tlocvir}.

\begin{theorem}\label{Tloc01} Let  $f\in\CC[t] \setminus \{0\}$ and $\chi\in W_{\ge-1}(f)^*$.  
Then the following conditions are equivalent:
\begin{itemize}

\item[(0)] $\chi$ is the restriction of a local function on $W_{\geq -1}$; 

\item[(1)] $\CoreP(\chi)\ne(0)$;

\item[(2)]  $\dim W_{\ge-1}(f)/W_{\ge-1}(f)^\chi=\dim W_{\geq -1}(f) \cdot\chi <\infty$;

\item[(3)] there exists $h\in\CC[t]$ such that $\chi|_{W_{\ge-1}(fh)}=0$;
\item[(4)]  $W_{\ge-1}(f)^\chi\ne (0)$.
\end{itemize}
\end{theorem}

\begin{remark}\label{Rlrec} Fix a basis $\{ ft^i\del \ | \  i\ge-1 \} $ of $W_{\ge-1}(f)$ and consider a local function $\chi\in W_{\ge-1}(f)^*$. 
Then $\chi$ can be identified with a sequence
\beq\label{Elr} \chi_0=\chi(f\del), \chi_1=\chi(ft\del), \chi_2=\chi(ft^2\del), \chi_3=\chi(ft^3\del),\cdots \in \kk. \eeq
Condition $(3)$ above
can be restated as follows:
$$a_n\chi_{m+n}+...+a_0\chi_{m}=0$$
for all $m\ge 0$, where $h=a_nt^n+\cdots+a_0$.
Therefore,  local functions on $W_{\ge-1}(f)$ can be identified with   sequences \eqref{Elr} obeying a linear recurrence relation.
This shows, in particular, that the function $\varkappa$ defined in Remark~\ref{rem:dag} is not local.  
\end{remark}

Part of the proof of Theorem~\ref{Tloc01} is a general technique that can allow us to show that Poisson cores of elements of $\mf g^*$ are nontrivial for any Lie algebra $\mf g$. 

\begin{lemma}\label{lem:detideal}
Let $\mf g$ be an arbitrary Lie algebra  and let $n \in \NN$.  
There is a Poisson ideal $I(n)$ with the property that
\beq\label{detideal} I(n) \subseteq \mf m_\chi \iff \dim \mf g \cdot \chi \leq n\eeq
for any $\chi \in \mf g^*$. 
\end{lemma}
\begin{proof}
Recall the determinant $D(u_1,  \ldots, u_n; v_1,  \ldots, v_n)$ from \eqref{Ednp1}, and note that $\dim \mf g \cdot \chi  \leq~n-1$
if and only if for all $u_1,\ldots, u_{n}, v_1, \ldots, v_{n}\in \mf g$, \eqref{Ednp1} evaluated at $\chi$
is degenerate, i.e. $${\rm ev}_\chi(D(u_1, u_2, \ldots, u_{n}; v_1, v_2, \ldots, v_{n}))=0.$$

Let $I(n)$ be the ideal generated by the $D(u_1, u_2, \ldots, u_{n+1}; v_1, v_2, \ldots, v_{n+1})$
for all possible tuples \[u_1, \ldots, u_{n+1}, v_1, \ldots, v_{n+1} \in \mf g.\] 
By the previous paragraph,
$$I(n) \subseteq \mf m_\chi\iff\dim \mf g \cdot \chi \leq~n.$$ 

Let $w \in \mf g$.  
It is easy to check that
\begin{multline*}\{D(u_1, u_2, \ldots; v_1, v_2, \ldots), w\}=\\
D([u_1, w], u_2, u_3, \ldots; v_1, v_2, \ldots)+D(u_1, [u_2, w], u_3, \ldots; v_1, v_2, \ldots)+\ldots\\+ D(u_1, u_2, \ldots; [v_1, w], v_2, v_3, \ldots)+D(u_1, u_2, \ldots; v_1, [v_2, w], v_3, \ldots)+\ldots \in I(n).
\end{multline*}
It follows that $I(n)$ is Poisson.
\end{proof}

\begin{proof}[Proof of Theorem~\ref{Tloc01}] That $(0)$ $\iff$ $(3)$ is a straightforward application of the Chinese remainder theorem. 
It is clear that $(2)$  $\Rightarrow$ $(4)$;  
Corollary~\ref{cor:finiteorbit} gives that $(1)$ $\Rightarrow$ $(2)$.
We will show that $(4)$ $\Rightarrow$ $(3)$ $\Rightarrow$ $(2)$ $\Rightarrow$ $(1)$.
This will complete the proof. 

We  first show that $(4)$ implies $(3)$. 
Let $hf\del \in W_{\ge-1}(f)^\chi\setminus\{0\}$ with $h\in\CC[t]$. 
Then 
\begin{equation}0=\chi([hf\del, hfr\del])=\chi(h^2f^2r'\del)\label{Erprime}\end{equation}
for all $r\in\CC[t]$. 
This is equivalent to $\chi|_{W_{\ge-1}(h^2f^2)}=0$ as needed.

Next, we will show that $(3)$ implies $(2)$.  
Let $h$ satisfy condition $(3)$. 
As $$[W_{\geq -1}(f^2h^2), W_{\geq -1}] \subseteq W_{\geq -1}(fh),$$ we have
 $W_{\ge-1}(h^2f^2)\subseteq  W_{\geq -1}(f)^\chi$ and thus 
 $\dim W(f)/W(f)^\chi < \infty$ 
as needed.

Finally we show that $(2)$ implies $(1)$. 
Suppose that  $\chi\in W_{\ge-1}(f)$ satisfies condition $(2)$. 
Let $n = \dim W_{\geq -1}(f)/W_{\geq -1}(f)^\chi = \dim W_{\geq -1}(f) \cdot \chi$ and let $I(n)$ be the ideal defined in Lemma~\ref{lem:detideal}.  
By that lemma, $I(n)$ is Poisson and $I(n) \subseteq \mf m_\chi$, so 
 $I(n)\subseteq  \CoreP(\chi)$. 

Therefore if $I(n)\ne0$ then $\CoreP(\chi)\ne0$. 
To show that $I(n)\ne0$ it suffices to find $\omega \in W_{\geq -1}(f)^*$  with $\dim W_{\geq -1}(f) \cdot\omega>n$. 
In fact, we will find $\omega \in W_{\geq-1}^*$ with $$\dim W_{\geq-1}(f) \cdot \omega =\infty.$$ 
Indeed, as $(2)$ $\Rightarrow$ $(3)$, if $\dim W_{\geq -1}(f)\cdot \omega < \infty$ then $\omega$ can be represented by a linearly recurrent sequence by Remark~\ref{Rlrec}.  On the other hand, the sequence  $1, \frac12, \frac13, \ldots$ is clearly not linearly recurrent.
\end{proof}

\begin{remark}\label{rem:majority}
Similarly to Remarks~\ref{rem:dag} and ~\ref{rem:countableU}, we should expect that ``most'' sequences are not linearly recurrent and that the linearly recurrent sequences are parameterised by a countable union of affine varieties, although we do not formalise these notions here.
\end{remark}

Note that $W_{\ge-1}(t^2)=W_{\ge1}$ and that $  W_{\geq -1}(t)$ is equal to the  {\em non-negative Witt algebra} $W_{\geq 0}$.  Therefore Theorem~\ref{Tloc01} gives a complete characterization of local functions on $W_{\ge1}$ and $W_{\geq 0}$.

\subsection{Local functions on $W$ and applications} 
In this subsection we set $\mf g=\CC[t, t^{-1}]\del$ and  define $W(f)$ similarly to $W_{\ge1}(f)$.  
\label{ind:Wg2}
A partial analogue of Theorem~\ref{Tloc01} holds for $W$. 
\begin{theorem}\label{Tloc} For any $f\in\CC[t] \setminus\{0\}$ and $\chi\in W(f)^*$ the following conditions are equivalent:
\begin{itemize}
   
\item[(0)] $\chi$ is the restriction of a local function on $W$;

\item[(1)] $\CoreP(\chi)\ne(0)$;

\item[(2)] $\dim W(f)/W(f)^\chi<\infty$;

\item[(3)] there exists $h\in\kk[x]$ such that $\chi|_{W(fh)}=0$.
\end{itemize}
\end{theorem}
\begin{remark}\label{Rres} The reason that  Theorem~\ref{Tloc} differs slightly from Theorem~\ref{Tloc01} is that the function $\Res_0(\cdot) \in W^*$ satisfies condition (4) of Theorem~\ref{Tloc01} but does not satisfy the other conditions (0), (1), (2), (3). \end{remark}

Before proving Theorem~\ref{Tloc} we give two lemmata on functions defined by residues.  
Denote by $\CC((t))$ the field of formal Laurent power series in $t$.
Fix $f\in\CC((t))$ and consider the map
\begin{equation} 
(a, b)\mapsto (a, b)_f:=\Res_0(f(ab'-a'b))\label{Esres}\end{equation}
which defines a  skew-symmetric bilinear form on $\kk((t))$.
\begin{lemma}\label{Lpsq} The kernel of $(a, b)_f$ is one-dimensional if $f$ is a perfect square in $\CC((t))$ and is trivial otherwise.
In the first case the kernel is generated by $\frac1{\sqrt{f}}$.\end{lemma}
\begin{proof}Let $a$ be in the kernel of $(\cdot, \cdot)_f$. 
Then 
\beq \label{Esker} \Res_0(fa^2r') = \Res_0(f(a (ar)' - a'(ar))) = (a, ar)_f = 0 \eeq
for all $r \in\kk((t))$.
This implies that $fa^2$ is constant. 
Thus $f$ is a perfect square in $\CC((t))$ and $a$ is proportional to $\frac1{\sqrt{f}}$. 
Equation~\eqref{Esker} also gives that if $f$ is a perfect square, then $\frac1{\sqrt f}$  belongs to the kernel of $(\cdot, \cdot)_f$.
\end{proof}
The second lemma proves part of (a more general version of) Remark~\ref{Rres}.
\begin{lemma}\label{lem:two}
Let $g \in \kk[t,t^{-1}] \setminus\{0\}$ and define $\omega \in W(g^2)^*$ by 
\[ \omega = \Res_0(\frac{\cdot}{g^2}).\]
Then $\dim W(g^2)\cdot \omega = \infty$.
\end{lemma}
\begin{proof}
For $a,b \in \kk[t, t^{-1}]$ we have 
\[ B_\omega(g^2 a \del, g^2 b\del) = \Res_0(g^2(ab'-a'b)).\]
Suppose that $a \not \in \kk \cdot \frac{1}{g}$.
Then by Lemma~\ref{Lpsq} there is a formal Laurent series $\hat{b}$ so that 
\[(a,b)_{g^2} = \Res_0(g^2(a \hat{b}'-a'\hat{b})) \neq 0.\]
However, the computation of $\Res_0(g^2(a \hat{b}'-a'\hat{b})) $ needs only finitely many terms in the Laurent expansion of $\hat{b}$ and so we may replace $\hat{b}$ by $b \in \kk[t, t^{-1}]$ so that
$\Res_0(g^2(ab'-a'b)) \neq 0$.
Thus $a g^2\del \not \in \ker B_\omega = W(g^2)^\omega$.
This means that $W(g^2)^\omega \subseteq \kk \cdot g \del$ and $\dim W(g^2)\cdot \omega = \infty$.
\end{proof}

\begin{proof}[Proof of Theorem~\ref{Tloc}] The proofs of $(0)$ $\iff$ $(3)$, $(1)$ $\Rightarrow$ $(2)$ and $(3)$ $\Rightarrow$ $(2)$ $\Rightarrow$ $(1)$ are very similar to the corresponding steps of the proof of Theorem~\ref{Tloc01}. 
The only part which is significantly different is (2)$\Rightarrow (3)$. 

Pick $\chi$ satisfying $(2)$ and $h \in W(f)^\chi \setminus\{0\}$. 
Then \eqref{Erprime} holds for all $r\in\CC[t, t^{-1}]$. 
Unfortunately, this is not enough to show that $\chi$ vanishes on $W(h^2f^2)$, as the  map $r\mapsto r'$ is not surjective on $\CC[t, t^{-1}]$. 

Consider $\chi|_{W(f^2h^2)}$:  we have
\[ \chi(f^2h^2p\del) = \chi(\frac{f^2h^2}{t}\del) \Res_0(p)\] 
for all $p \in \kk[t]$.  Suppose that $\chi(\frac{f^2h^2}{t}\del) \neq 0$.
Then $\dim W(f^2h^2)\cdot \chi|_{W(f^2h^2)} = \infty$  by Lemma~\ref{lem:two}.
Thus $\dim W(h^2f^2) \cdot \chi = \infty$, as $\dim W(f)/W(f^2h^2) < \infty$.
This contradicts our assumption that $\chi$ satisfies $(2)$ and so $\chi(\frac{f^2h^2}{t}\del ) = 0$, i.e. $\chi|_{W(f^2h^2)} = 0$.
\end{proof}

To end the subsection, we apply Theorem~\ref{Tloc} to obtain a structure result on finite codimension subalgebras of $W$.  
For a polynomial $f\in\CC[t, t^{-1}]$ set 
$${\rm rad}(f):=\prod \{ (t-x) | x\in\CC^\times, f(x) = 0\}.$$
\label{ind:rad}

\begin{proposition}\label{prop:4.14} Let $\mf k$ be a subalgebra of $W$ of finite codimension. Then
\begin{itemize}
\item[(a)] there exists $f\in \CC[t, t^{-1}]$ so that $W({\rm rad}(f)) \supseteq \mf k \supseteq W(f)$;

\item[(b)] we can choose $f$ satisfying $\mathrm{(a)}$ so that $f\in \CC[t]$, $f$ is monic and $f(0) \neq 0$; 

\item[(c)]  if we assume that $f$ is of minimal degree then such a choice of $f$ is unique.
\end{itemize}
\end{proposition}
\begin{proof} 
The inclusion $\mf k \subseteq W$ induces the dual map $W^*\to\mf k^*$. 
We identify $(W/\mf k)^*$ with the kernel of this map; that is, with elements of $W^*$ which vanish on $\mf k$, so $\mf k$ is the set of common zeroes of $(W/\mf k)^* \subseteq W^*$.
Let $\chi_1, \dots, \chi_s$ be a basis of $(W/\mf k)^*$.
 Fix $i\in \{1, \dots, s\}$; by definition we have $B_{\chi_i}(\mf k, \mf k)=0$,  so $\mf k$ is an isotropic subspace of $W$ with respect to $B_{\chi_i}$. 
Hence the rank of $B_{\chi_i}$ is at most $2\dim (W/\mf k)$ and thus is finite. 
By Theorem~\ref{Tloc},   $\chi_i$ is  local. 

Theorem~\ref{Tloc} implies that for all $i$ there is $h_i\in\CC[t, t^{-1}]\setminus \{0\}$ with $\chi_i(W(h_i))=0$. 
Therefore $W(h_1\cdots h_s)$ is annihilated by all $\chi_i$  and therefore $W(h_1\cdots h_s)\subseteq\mf k$ as desired.

Let $f \in \kk[t,t^{-1}] \setminus \{0\}$ with $\mf k \supseteq W(f)$; we may assume without loss of generality that $f \in \CC[t]$, $f$ is monic  and $f(0) \neq 0$ as $W(f)$ corresponds to an ideal of $\CC[t, t^{-1}]$.  
Suppose in addition that  $f$ has minimal degree among all such polynomials with $\mf k \supseteq W(f)$.  
Thus if $\mf k \supseteq W(h)$, then by the Euclidean algorithm $f | h$. 
This justifies uniqueness of $f$. 

Write $f = \prod_{i=1}^n (t-x_i)^{a_i}$ with the $x_i\ne0$ distinct and $a_i > 0$; 
set $h:={\rm rad}(f)=(t-x_1)\cdots(t-x_n).$ It is clear that $h\mid f\mid h^{\max(a_i)}$.
Let $k\del \in \mf k$.  
We wish to show that $h| k$. 
Indeed, for all $r \in \CC[t,t^{-1}] $  the element
$[k\del, fr\del] = (k(fr'+ f'r)-frk')\del$ is in $\mf k$.  
As $\mf k \supseteq W(f)$, thus $kf'r \in \mf k$ for all $r\in\CC[t, t^{-1}]$.  
Thus $\mf k \supseteq W(kf')$ and so $f | f'k$.  
This forces $k$ to vanish at all roots of $f$, which is equivalent to $h | k$. 
\end{proof}

\subsection{Local functions on $\Vir$}\label{SSlvr}
In this subsection we set $\mf g=\Vir$. 

The natural map $\Vir\to W$ extends to the morphism $\Sa(\Vir)\to \Sa(W)$; the kernel is the Poisson ideal $(z)$ of $\Sa(\Vir)$.
The main goal of the subsection is to prove the following  analogue of Theorem~\ref{Tloc}  for $\Vir$. 
\begin{theorem}\label{Tlocvir} For 
$\chi\in \Vir^*$ we have 
\[\CoreP(\chi)\ne (z-\chi(z)) \quad \iff \quad \chi \mbox{ is local } \quad \iff \quad \dim \Vir/\Vir^\chi < \infty.\]
In particular, if $\CoreP(\chi) \neq (z-\chi(z))$ then $\chi(z) = 0$.
\end{theorem}

\begin{remark}\label{rem:locvir}
It follows from Theorems~\ref{Tloc} and \ref{Tlocvir} that $\chi \in \Vir^*$ is local if and only if there is some $f \in \kk[t, t^{-1}]$ such that $\chi$ vanishes on $W(f) + \kk \cdot z \subseteq \Vir$.
\end{remark}

Before proving Theorem~\ref{Tlocvir}, we consider arbitrary subalgebras of $\Vir$ of finite codimension, and show they are strongly constrained. 

\begin{proposition}\label{prop:4.15}
If $\mf k$ is a subalgebra of $\Vir$ of finite codimension, then $z \in [\mf k, \mf k]$.  
Thus $\dim \Vir/\mf k < \infty$ if and only if $\mf k$ contains some $W(f) + \kk \cdot z$, where $f \in \kk[t,t^{-1}] \setminus \{0\}$.
\end{proposition}

This result generalises \cite[Proposition~2.3]{OW}, which considered subalgebras of $\Vir$ of codimension 1. We also note that \cite{OW} refers to  subalgebras of $\Vir$ of the form $W(f)+\kk \cdot z$ as {\em polynomial subalgebras}.

\begin{proof}[Proof of Proposition~\ref{prop:4.15}] Let $\bar{\mf k}$ be the image of $\mf k$ in $W$.  By Proposition~\ref{prop:4.14} there is some $ f \in \CC[t, t^{-1}] \setminus \{0\}$ so that $\bar{\mf k} \supseteq W(f)$.  Thus for all $p \in \ZZ$, there is $\zeta_p \in \CC$ so that the element
\[ v_p := f t^p \del + \zeta_p z\]
is in $\mf k$.
Therefore, $[\mf k, \mf k]$ contains the elements  
\begin{equation}\frac1{q-p}[v_p, v_q]=f^2t^{p+q-1}\del +  \Res_0\Bigl(t^{p+q-3}(2t^2(f')^2+ff't(p+q-1)-t^2ff''+f^2pq)\Bigr)z
\label{Evir01}\end{equation}
for all $p,q \in \ZZ$.
Fix $d=p+q$ and consider $p = d-q$ as a function of $q$. 
The only  part of \eqref{Evir01} that varies with $q$ is 
$q(d-q) \Res_0(t^{d-3} f^2)$.

If 
$\Res_0(t^{d-3} f^2)$ 
is not zero then
$$\frac1{d-2q_1}[v_{d-q_1}, v_{q_1}]-\frac1{d-2q_2}[v_{d-q_2}, v_{q_2}] = (q_1-q_2)(d- q_1-q_2) \Res_0(t^{d-3} f^2) z$$
is a nonzero scalar multiple of $z$ for almost all  $q_1, q_2$. 
If $z\not \in [ \mf k, \mf k]$, we therefore have
$\Res_0(t^d f^2)=0$ for all $d\in\mathbb Z$. 
This implies that $f^2=0$, contradicting our assumption on $f$.

The final sentence is an immediate consequence of Proposition~\ref{prop:4.14}.
\end{proof}
\begin{proof}[Proof of Theorem~\ref{Tlocvir}] 
Let $\chi \in \Vir^*$.
If $\chi$ is local then by definition $\chi$ descends to a local function $\overline{\chi}$ on $\Vir/(z) \cong W$.
By Theorem~\ref{Tloc},
 $\CoreP(\chi) \supsetneqq (z)$ and $\dim \Vir/\Vir^\chi = \dim W/W^{\overline\chi} < \infty$.
 
 If $\dim \Vir/\Vir^\chi < \infty$ then by Lemma~\ref{prop:4.15}, $z \in [\Vir^\chi, \Vir^\chi]$ and so $\chi(z) = 0$, as $\chi$ vanishes, by definition, on $[\Vir^\chi, \Vir^\chi]$.
We may thus factor out $z$ and apply Theorem~\ref{Tloc} again to conclude that $\chi$ is local.

Finally, suppose that $\CoreP(\chi) \neq (z-\chi(z))$.
Then by Corollary~\ref{cor:finiteorbit} $\dim \Vir \cdot \chi = \dim \Vir/\Vir^\chi < \infty$.
\end{proof}

As an immediate corollary of Theorem~\ref{Tlocvir}, we obtain a powerful result on Poisson ideals of $\Sa(\Vir)$.
\begin{corollary}\label{cor:Psimple}
If $\zeta \in \kk^\times$ then $\Sa(\Vir)/(z-\zeta)$ is Poisson simple, i.e. contains no nontrivial Poisson ideals.
\end{corollary}
\begin{proof}
Let $\zeta \in \kk$.  
If $\Sa(\Vir)/(z-\zeta)$ is not Poisson simple then $(z-\zeta)$ is strictly contained in some proper Poisson ideal $J$ of $\Sa(\Vir)$.
By the Nullstellensatz there is some $\chi \in \Vir^*$ with $J \subseteq \mathfrak{m}_\chi$; thus $\chi(z) = \zeta$ as $z -\zeta \in \mf m_\chi$.
Further $\CoreP(\chi) \supseteq J \supsetneqq (z-\zeta)$ and so by Theorem~\ref{Tlocvir} we have $\zeta = 0$.
\end{proof}

We thus show that almost all prime Poisson ideals of $\Sa(\Vir)$ contain $z$.
\begin{corollary}\label{cor:primeVir}
Let $Q$ be a prime Poisson ideal of $\Sa(\Vir)$.
Then either:
\begin{itemize}
    \item $Q = (0)$;
    \item $Q = (z-\zeta)$ for some $\zeta \in \kk^\times$;
    \item $Q \supseteq (z)$.
\end{itemize}
\end{corollary}
\begin{proof}
By Corollary~\ref{cor:Psimple} and primeness of $Q$, it suffices to prove that if $Q \neq (0)$ then $Q$ contains a nonzero element of $\kk[z]$.
Let $h \in Q \setminus \{0\}$; using primeness of $Q$ we may assume that $h$ is not a multiple of any element of $ \kk[z] \setminus \kk$.
Let $\chi \in \Vir^*$ so that $Q \subseteq \mf m_\chi$.  
As $h \in \CoreP(\chi)$, we see that $\CoreP(\chi) \neq (z-\chi(z))$.
By Theorem~\ref{Tlocvir} $\CoreP(\chi) \supseteq (z)$.
Thus, applying Lemma~\ref{lem:basics}(a), 
\[Q = \bigcap \{ \CoreP(\chi) \ | \ \ev_\chi(Q) = 0 \} \supseteq (z).\]
\end{proof}

Given Corollary~\ref{cor:Psimple} it is natural to conjecture:  
\begin{conjecture}\label{conj:Usimple}
If $\zeta \neq 0$ then $\Ua(\Vir)/(z-\zeta)$ is simple.
\end{conjecture}

However, we as yet have no proof of Conjecture~\ref{conj:Usimple}.  
Note that the obvious strategy of proof by taking the associated graded of an ideal $ (z-\zeta) \subsetneqq J \triangleleft \Ua(\Vir)$ does not work; for in this case $\gr J \ni z$ and Corollary~\ref{cor:Psimple} is not directly relevant.

\section{Pseudo-orbits and Poisson primitive ideals for the algebras of interest}
\label{Spo}
Let $\mf g$ be one of $\Vir, W, W_{\ge-1}$, or  $W_{\ge1}$.  
In this section we will describe the pseudo-orbits for $\mf g$, using the
results on local functions from the previous section and the general strategy of Proposition~\ref{Practpoi}, and derive some consequences for the Poisson primitive spectrum of $\Sa(\mf g)$. 
 
 We begin by describing the pseudo-orbits of non-local functions, where the results of Section~\ref{Spridlf} quickly give the  answer.
 
 \begin{proposition}\label{prop:nonlocal}
 For $\mf g =  W, W_{\ge-1}$, or $W_{\ge1}$, the non-local functions in $\mf g^*$ form a  pseudo-orbit.  If $\mf g = \Vir$, then for any $\zeta \in \kk$ the non-local functions $\chi$ with $\chi(z) = \zeta$ form a pseudo-orbit.
 \end{proposition}
 \begin{proof}
 This is immediate from Theorems~\ref{Tloc01} ($\mf g = W_{\geq -1}$ or $W_{\geq 1}$), \ref{Tloc}  ($\mf g= W$), or \ref{Tlocvir} ($\mf g = \Vir$).  For by those results, if $\mf g = W, W_{\geq -1},$ or $W_{\geq 1}$ then $\chi \in \mf g^*$ is not local if and only if $\CoreP(\chi) = 0$; and if $\mf g = \Vir$, then $\chi$ is not local if and only if $\CoreP(\chi) = (z - \chi(z))$.
 \end{proof}
 We may thus restrict to considering pseudo-orbits of local functions.  By Proposition~\ref{prop:nonlocal} if $\chi  \in \mf g^*$ is local and $\omega \in \mb O(\chi)$ then $\omega$ is also local. 
 Since by definition local functions on $\Vir$ vanish on $z$, the pseudo-orbits for $W$ directly determine those for $\Vir$.
 
 Thus for the rest of the section we  let $\mf g$ be $ W, W_{\ge-1}$, or $W_{\ge1}$.
In Subsection~\ref{SSgra} we will introduce a  finite-dimensional action which determines the pseudo-orbits of one-point local functions on $\mf g$, in   Subsection~\ref{SSexd} we describe the orbits of this action explicitly, and in Subsection~\ref{SSimd} we use this action to  describe pseudo-orbits of arbitrary local functions.

\subsection{An algebraic group acting on local functions} \label{SSgra} 
Set $\mf g$ to be $W$, $W_{\ge-1}$, or $W_{\ge1}$.  We fix notation for the subsection.
\begin{definition} 
For $x \in \kk$, $n \in \ZZ_{\geq 0}$, let $\Loc_x^{\le n}$ 
\label{ind:Loc1}
denote the subspace of $\mf g^*$ consisting of  one-point local functions based at $x$ and of order $\leq n$.
Let  $\Loc_x:=\bigcup_{n\ge 0}\Loc_x^{\le n}.$
\label{ind:Loc2}
Define $\Loc^{\le n}_{\mf g}=\bigcup_{x\in\CC^\times}\Loc_x^{\le n}$ if $\mf g=W$, and  $\Loc^{\le n}_{\mf g}=\bigcup_{x\in\CC}\Loc_x^{\le n}$ if $\mf g=W_{\ge-1}$ or $W_{\ge1}$.
\label{ind:Loc3}
Let
\[ \widetilde{\Loc}_{\mf g}^{\leq n}  = 
\begin{cases} 
\kk^\times \times \kk^{n+1} & \mf g = W \\
\kk \times \kk^{n+1} & \mf g= W_{\geq -1} \\
\kk \times \kk^{n+1}  & \mf g = W_{\geq 1}. 
\end{cases}
\]
\label{ind:Loc4}
For all $\mf g$ there is a canonical 
map $\pi_{\mf g}^{\leq n}: \widetilde{\Loc}_{\mf g}^{\leq n} \to \Loc_{\mf g}^{\leq n}$
\label{ind:pig}
 (this is in fact birational, cf. Proposition~\ref{prop:iso}).  
If $\mf g=W, W_{\ge-1}$ then  $\pi_{\mf g}^{\leq n} $ is an isomorphism away from $\chi = 0$.
If  $\mf g=W_{\ge1}$ it is an isomorphism away from $\chi = 0$ and $x=0$, 
due to the fact that $\chi_{0; \beta_0, \beta_1}=0$ for all $\beta_0, \beta_1\in\kk$ in this case.  
Formally, we let  $\phi_{\mf g }^{\leq n}:  \widetilde{\Loc}_{\mf g}^{\leq n} \to \mf g^*$ 
\label{ind:phig}
be the composition of $\pi_{\mf g}^{\leq n}$ with the inclusion $\Loc_{\mf g}^{\leq n} \subseteq \mf g^*$.
There is an induced pullback morphism   $(\phi^{\le n}_{\mf g})^*:\Sa(\mf g)\to \kk[\widetilde{\Loc}_{\mf g}^{\leq n}]$.
\end{definition}

If $\chi$ is a local function on $\mf g$ and $v \in \mf g$, then $v\cdot \chi$ is local, and in fact the coadjoint action of $\mf g$ preserves $\Loc_x$ for all $x$.
We now study the action of $\mf g$ on $\Loc_x$, and relate it to a finite-dimensional action using Proposition~\ref{Practpoi}.

We begin by defining a group action.
Fix  $x\in\CC$ (if $\mf g = W$ we assume that $x \neq 0$).  Let $\tilde t:=t-x$. 
Clearly, \eqref{Elf} makes sense for every formal power series $f\del\in \CC[[\tilde t]]\del$, and so $\Loc_{x}^{\leq n}$ also gives elements of  $(\kk[[\tilde t]]\del)^*$.  
For every $s\in \CC[[\tilde t]]$ with $s(x)=x$ 
we introduce a local change of coordinates endomorphism $\End_{t\to s}(\cdot)$
\label{ind:End}
 of $\CC[[\tilde t]]$ through the formula $t\to s$. 
Note that $\End_{t \to s}$ is  invertible if $s'(x)\ne0$. 
Let $\DLoc_x$,
\label{ind:DLoc}
 the group of {\em formal local diffeomorphisms at $x$}, denote the group of all $\End_{t \to s}$ with $s(x) = x$ and $s'(x) \neq 0$.
The group $\DLoc_x$ has a subgroup of transformations of the form $t \to \zeta t + (1-\zeta) x$, for $\zeta \in \CC^\times$; we also write this transformation as  $\tilde t\to\zeta \tilde t$ and let $\Dil_x$
\label{ind:Dil}
 denote the  group of such transformations, which we term {\em dilations at $x$}.

Pick $s\in\CC[[\tilde t]]$ with $s(x)=x$ and $s'(x) \neq 0$. 
We extend $\End_{t\to s}$ to an automorphism of $\CC[[\tilde t]]\del$ via the formulae
\beq \label{derived action}
t\mapsto s,\hspace{10pt} \del\mapsto \frac1{s'}\del~(=\partial_s).
\eeq
This gives  actions of $\DLoc_x$ on $\CC[[\tilde t]]\del$ and on $(\kk[[\tilde t]]\del)^*$. 
We may consider  $\Loc_x^{\leq n}$ as a subset of  $(\kk [[\tilde t]]\del)^*$, and this subset is preserved by the $\DLoc_x$-action. 

Denote by $\DLoc_x^{\le n}$
\label{ind:DLocn}
 the image of $\DLoc_x$ in the group $\Aut(\Loc_x^{\le n})$ of linear automorphisms of $\Loc_x^{\le n}$.
Although $\DLoc_x$ is  infinite-dimensional, its image in $\Aut(\Loc_x^{\le n})$ is a finite-dimensional solvable algebraic group. 
Let us consider the action of the corresponding (finite-dimensional) Lie algebra, which we denote by $\lie_x^{\le n}$. 
\label{ind:lien}

\begin{lemma}\label{Lwgr} Identify $\lie_x^{\le n}$ with the tangent space to  $\DLoc_x^{\le n}$ at the identity. 
Let $s \in \kk[[t]]$ with 
 $s(x)=0$  and denote by $\xi_s$ the tangent direction at the identity defined by the line $h\to \End_{t\to t+hs}(\cdot)$ for $ h\in\CC$. 
The action of $\DLoc_x^{\le n}$ derives to an action of $\xi_s$ on $\Loc_x^{\le n}$. Then

\begin{itemize}
\item[(a)] $\lie_x^{\le n}$ consists of vectors of the form $\xi_s$ with $s(x)=0$;

\item[(b)]  $\xi_s\cdot\chi=s\del\cdot \chi$ for all $\chi\in \Loc_x$. 
\end{itemize}
\end{lemma}
\begin{proof} Part (a) is clear from the definition of $\DLoc_x^{\le n}$. 
For (b), as the action of $\xi_s$ on $\Loc_x$ is induced by the action \eqref{derived action} on $\kk[[\tilde t]] \del$, it is enough to check that $\xi_s \cdot f \del = [s \del, f\del]$ for all $f \del \in \mf g$. 
But 
\begin{align*}
    \xi_s \cdot f\del & = (\xi_s \cdot f) \del + f (\xi_s \cdot \del) \\
    & = (\del_h f(t+hs))|_{h = 0}\del + f (\del_h(\End_{t \to t+hs}\del))|_{h=0} \\
    & = (f'(t+hs) s)|_{h = 0} \del + f (\del_h(\frac{1}{1+hs'}\del))|_{h = 0} \\
    & = f's \del - f s'\del = [s\del, f \del],
\end{align*}
as needed.
\end{proof}
We now let $\mf g=W_{\ge-1}$. 
For all $x, n$,  the coadjoint action of $W_{\geq -1}(t-x)$ 
preserves $\Loc_x^{\leq n}$, and Lemma~\ref{Lwgr} shows that 
\begin{equation}\label{Etsp}W_{\geq -1}(t-x)\cdot\chi=\lie_x^{\leq n}\cdot\chi\end{equation}
for $\chi\in \Loc_x^{\leq n}$. 
For $z \in \kk$ define
$$\Shift_z(\chi_{x; \alpha_0, \alpha_1, \ldots}):=\chi_{x+z; \alpha_0, \alpha_1, \ldots}.$$
The set $\{\Shift_z |  z\in\CC\}$ forms a one-dimensional algebraic group, which we denote by $\Shifts$; 
\label{ind:Shifts}
clearly $\Shifts \cong \kk^+$. 
Deriving the action of  $\Shifts$  on $\Loc_{W_{\geq -1}}$ and on $\Loc_{W_{\geq -1}}^{\leq n}$, we have
\begin{multline}
    \del_h \Shift_h(\chi_{x; \alpha_0, \cdots, \alpha_n}(f\del))|_{h=0} \\ 
    = \del_h(\alpha_0 f(x+h) + \alpha_1 f'(x+h) + \cdots + \alpha_n f^{(n)}(x+h)) |_{h=0}\\
    = \alpha_0 f'(x) + \cdots + \alpha_n f^{(n+1)}(x) = (\del \cdot \chi_{x; \alpha_0, \cdots, \alpha_n})(f\del)\label{Eshift}
\end{multline}
and thus we may without ambiguity identify the Lie algebra of $\Shifts$  with $\kk \cdot \del$.

For all $x, z\in \kk$,  $\Shift_z$ gives a   continuous homomorphism $\CC[[t-x]]\to\CC[[t-x-z]]$ and hence induces an isomorphism $$\Shift_z: \DLoc^{\le n}_x\to \DLoc^{\le n}_{x+z}.$$ 
In particular, we can take $x = 1$, and then for $\mf g = W_{\geq -1}$ and for all $n$ the 
 action map $ \Shifts\times \Loc^{\le n}_1\to \Loc_{W_{\ge-1}}^{\le n}$ is clearly bijective.
This allows us to introduce the action of $\widehat{\DLoc}^{\le n}:=\Shifts \times \DLoc_1^{\le n}$ 
and $\widehat{\DLoc} := \Shifts \times \DLoc_1$ 
\label{ind:whDloc}
on $$\Loc_{W_{\ge-1}}^{\le n}\cong \CC \times \Loc_1^{\le n}$$ componentwise.  
Note that $\widehat{\DLoc}$ and $\widehat{\DLoc}^{\le n}$ also act on $\widetilde{\Loc}_{W_{\geq -1}}^{\leq n}$ in such a way that $\pi_{W_{\geq -1}}^{\leq n}$ is equivariant. 

Let $\widehat{\lie}^{\leq n}$
\label{ind:whlie}
 denote the Lie algebra of $\widehat{\DLoc}^{\leq n}$.
From Lemma~\ref{Lwgr} we have:
\begin{lemma}\label{Llact} Let $\mf g = W, W_{\geq -1}$ or $W_{\ge1}$ and 
fix $x$, $n$, and $\chi \in \Loc_x^{\leq n}$ (recall that $x\ne0$ if $\mf g=W$).  
\begin{itemize}
\item[(a)] 
If $\mf g=W_{\ge-1}$ then
\[\mf g \cdot\chi=\lie_x^{\le n}\cdot\chi + \kk \del_x\cdot\chi.\] 

\item[(b)]
If $\mf g=W_{\ge-1}$ then $\mf g\cdot\chi=\widehat{\rm lie}^{\le n}\cdot\chi$ for all $\chi\in \Loc_{\mf g}^{\le n}$.

\item[(c)] If $x\ne0$ then we can identify $\Loc_x^{\le n}$ for $W, W_{\ge-1},$ and $W_{\ge1}$. 
If $\mf g = W$ or $\mf g = W_{\geq 1}$, under  this identification we have $\mf g\cdot\chi=W_{\ge-1}\cdot\chi$. 

\item[(d)] If $\mf g=W_{\ge1}$ and $x=0$ then $\mf g\cdot\chi=\lie_x^{\le n}\cdot\chi$. 
\end{itemize}
\end{lemma}
\begin{proof}Parts (a),~(b), and (d) are straightforward from the above discussion. 
For part (c) note that the image of $f \del$ in $\Loc_x^{\leq n}$ depends only on the Taylor series expansion of $f$ around $x$ up to degree $n+1$, which is not affected by the behaviour of $f$ at $0$. 
\end{proof}

It is useful to have a more detailed description of $\DLoc_x$ and $\DLoc_x^{\le n}$. 
The main point here is that $\DLoc_x^{\le n}$ is a connected solvable algebraic group which has a filtration by normal subgroups with one-dimensional quotients. 

Pick $k\ge 2$. Let $x \in \kk$ (or $x \in \kk^\times$ if $\mf g = W$), and recall  that $\tilde t$ denotes $t-x$. 
Elements of the form $$\End_{\tilde t\to \tilde t + {\tilde t}^kh},\hspace{10pt} h\in\CC[[\tilde t]],$$ constitute a subgroup of $\DLoc_x$ and we denote this subgroup by $\DLoc_x^{k+}$.  
\label{ind:DLock}
We use the notation $\DLoc_x^{\le n, k+}$ for the image of $\DLoc_x^{k+}$ in $\DLoc_x^{\le n}$. 
The following lemma is straightforward. 
\begin{lemma}\label{Lfiltd} Let $x\in\kk$, and let  $\mf g=W, W_{\ge1}, W_{\ge-1}$ with $x\ne0$ for $\mf g=W$.

\begin{itemize}
\item[(a)] $\DLoc_x^{2+, \le n}$ is the unipotent radical of $\DLoc_x^{\le n}$ and $\Dil_x$ is a maximal reductive subgroup of $\DLoc_x^{\le n}$. 
In particular, the natural map $$\Dil_x\to \DLoc_x^{\le n}/\DLoc_x^{2+, \le n}$$ is an isomorphism.

\item[(b)] $\DLoc_x^{k+, \le n}/\DLoc_x^{k+1+, \le n}$ is either isomorphic to $ \kk$ or $\{0\}$ if $2\le k\le n$. Moreover, the natural map $$\kk \to  \DLoc_x^{k+, \le n}/\DLoc_x^{k+1+, \le n}$$ 
induced by sending $h \mapsto \End_{\tilde t\to \tilde t+ h\tilde t^n}(\cdot)$ is surjective.

\item[(c)] $\DLoc_x^{k+, \le n}$ is trivial for $k\ge n$. \qed
\end{itemize} 
\end{lemma}

\subsection{Explicit description of $\DLoc_x^{\le n}$-orbits} \label{SSexd}

We now compute the $\DLoc_x$-orbits of one-point local functions; in the next subsection we will see that this allows us to compute the pseudo-orbit of an arbitrary local function.  
If $\mf g = W_{\geq -1}$ then the action of $\DLoc_x$ is clearly homogeneous in $x$ for all $x\in\CC$, and similarly  for $W$ for $x \neq 0$.
If $x=0$ and $\mf g=W_{\ge1}$ the story is a bit more delicate.

Let $\mf g $ be $W_{\geq -1}$ or $W$ or $W_{\geq 1}$, and let $x \in \kk$.  If $\mf g= W$ or $W_{\geq-1}$ we additionally assume that $x \neq 0$.  
For $i \in \ZZ_{\geq 0}$ we set $e_i(x)=(t-x)^{i+1}\del$ and define $e_i(x)^*$
\label{ind:ei}
 by the formula $$f\del\mapsto\frac{f^{i+1}(x)}{(i+1)!}, $$
so that $e_i(x)^*(e_j(x))=\delta_{ij}$.
(We view the $e_i(x)^*$ as elements either of $(\CC[[t-x]]\del)^*$ or $\mf g^*$, depending on context.) 
The main goal of this subsection is to prove the following theorem, and to consider its consequences.
\begin{theorem}\label{Texpw} Assume $\mf g=W_{\ge-1}$ or $W$ or $W_{\geq 1}$, and let $x \in \kk$.  If $\mf g= W$ or $W_{\geq 1}$ we additionally assume that $x \neq 0$. 
Fix $n, \beta_0, \ldots, \beta_n$ with $\beta_n\ne0$ and set $\chi=\chi_{x; \beta_0,..., \beta_n}$. 
Let $e_i := e_i(x)$ and $e_i^*:= e_i(x)^*$ for all $i \in \ZZ_{\geq 0}$.

\begin{itemize}
\item[(a)] If $n$ is even then $\DLoc_x^{\le n}\chi=\DLoc_x^{\le n}e_{n-1}^*$, and $\dim \DLoc_x^{\le n}\chi=n+1$. 

\item[(b)] If $n$ is odd and $n>1$ then there is $\beta$ such that $\DLoc_x^{\le n}\chi=\DLoc_x^{\le n} (e_{n-1}^*+\beta e_k^*)$ where $k=\frac{n-1}2$. We have 
$\dim \DLoc_x^{\le n}\chi=n$.

\item[(b$'$)] If $n=1$ then 
$\DLoc_0^{\le n}\chi=\DLoc_x^{\le n} (\beta_1 e_0^*)$, and $\dim \DLoc_x^{\le n}\chi=1$. 

\item[(c)] Pick $\beta_1, \beta_2\in\CC$ and $k\in\mathbb Z_{\ge1}$. 
Then $\DLoc_x^{\le n}(e_{2k}^*+\beta_1 e_k^*)= \DLoc_x^{\le n}(e_{2k}^*+\beta_2 e_k^*)$ if and only if $\beta_1=\pm \beta_2$. 

\item[(c$'$)] Pick $\beta_1, \beta_2\in\CC$. 
Then $\DLoc_x^{\le 1}(\beta_1e_{0}^*)= \DLoc_x^{\le 1}(\beta_2e_{0}^*)$ if and only if $\beta_1=\beta_2$.
\end{itemize}
\end{theorem}

\begin{proof} 
The proof is the same for any base point; to reduce notation we give the proof for $\mf g= W_{\geq -1}$ and $x =0$.

First we compute the dimensions of the corresponding orbits.
Note that  
\[ \dim \DLoc_0^{\le n}\chi=\dim \lie_0^{\le n}\cdot \chi=\dim W_{\ge0}\cdot\chi, \]
see~\eqref{Etsp}. 
Moreover,~\eqref{Eshift} implies that $W_{\ge-1}\cdot\chi\not\subseteq \Loc_0^{\le n}$. 
Thus $\dim W_{\geq -1}\cdot \chi=\dim W_{\ge0}\cdot\chi+1$.  
Further $\dim W_{\geq -1}\cdot \chi=\rk B_{\chi}$. 
The rank of this form can be evaluated explicitly. 
Indeed,
$$\chi_{0; \beta_0, \ldots, \beta_n}([e_i, e_j])=
\begin{cases}0&{\rm~if~}i+j>n-1\\
(j-i)(i+j+1)!\beta_{i+j+1}&{\rm~if~}i+j\le n-1\end{cases}.$$
From this formula it is clear that $B_\chi(e_i, e_j)=0$ if $i\ge n+1$ or $j\ge n+1$. 
Thus the rank of this form can be evaluated on the first $(n+2)\times (n+2)$ entries corresponding to $-1\le i, j\le n$. 
This block is skew-upper-triangular.
If $n$ is even then all values on the  skew-diagonal line are nonzero (because $\beta_n\ne0$); hence the rank is $n+2$ in this case. 
If $n$ is odd then this skew-diagonal line contains zero only in the position corresponding to $i=j = \frac{n-1}2$; hence the rank is $n+1$ in this case. 
This provides the desired dimensions. 

The idea of the rest of the proof is to use the subnormal series in Lemma~\ref{Lfiltd} to reduce the number of coefficients which we have to consider step by step. 

Let $i \in \ZZ_{\geq 0}$.  For $j \geq 1$ we have 
\[
\End_{t \to t + \alpha t^{j+1}}(e_i) = 
    \frac{(t + \alpha t^{j+1})^{i+1}}{1 + \alpha(j+1) t^j} \del 
     = 
    (t^{i+1} + \alpha(i-j) t^{i+j+1} + \mbox{higher}) \del  = e_i + \alpha(i-j) e_{i+j} + \mbox{higher}.
\]
Thus 
\beq \label{star1}
\End_{t \to t+\alpha t^{j+1}}(e_i^*) = e_i^* + \alpha(i-2j) e_{i-j}^* + \mbox{a linear combination of $e_{< i-j}^*$}. 
\eeq
Likewise
\beq \label{star2}
\End_{t \to \zeta t}(e_i^*) = \zeta^i e_i^*.
\eeq

We apply these transformations to 
\beq \label{star3}
\chi_{0; \beta_0, \dots, \beta_n}= \beta_0 e_{-1}^* + \dots + \beta_n (n!) e_{n-1}^*,
\eeq
noting that the coefficient  of $e_{n-1}^*$ in \eqref{star3} is nonzero.

By applying $\End_{t \to t+ \alpha_1 t^2}$ with appropriate $\alpha_1$ we may cancel the coefficient of $e_{n-2}^*$, using \eqref{star1}.  This does not affect the coefficient of $e_{n-1}^*$.
Then by applying the appropriate $\End_{t \to t+\alpha_2 t^3}$ we can cancel the coefficient of $e_{n-3}^*$ without changing the coefficient of $e_{n-2}^*$ or $e_{n-1}^*$.  
Repeating, we may cancel the coefficients of all $e_{n-1-k}^*$ with $1 \leq k \leq n$, unless 
$k =\frac{ n-1}{2}$.

Thus in case (a) or (b$'$) we obtain that $\omega = \beta_n e_{n-1}^* \in \DLoc_0^{\leq n}(\chi)$.
In case (b) we obtain some $$\omega = \alpha e_{\frac{n-1}{2}}^* + \beta_n e_{n-1}^* \in \DLoc_0^{\leq n}(\chi).$$

Now applying \eqref{star2} to $\omega$ we may rescale the coefficient of $e_{n-1}^*$ by any $\zeta^{n-1}$, so if $n \neq 1$ we may set the coefficient of $e_{n-1}^*$ to be $1$.  This proves rest  of (a), (b), (b$'$). 

We now prove (c$'$). It follows from \eqref{star1} that   applying any nontrivial element of the unipotent radical of $\DLoc_0^{\leq 1}$ to $e_0^*$ will give a nonzero $e_{-1}^*$ term.  Thus we must simply consider acting by $\Dil_0$, and rescaling as in \eqref{star2} does not affect the coefficient of $e_0^*$. 

Finally, we prove (c). 
Formula 
\eqref{star2} implies
that if $\beta_1=\pm\beta_2$ then $$\DLoc_0^{\le n}(e_{2k}^*+\beta_1e_k^*)=\DLoc_0^{\le n}(e_{2k}^*+\beta_2e_k^*).$$ 
Pick $\beta_1, \beta_2\in\CC$ and set $\chi_1=e_{2k}^*+\beta_1e_{k}^*,  \chi_2=e_{2k}^*+\beta_2e_{k}^*$. 
Assume that $\DLoc_0^{\le n}(\chi_1)=\DLoc_0^{\le n}(\chi_2)$, so there exists $s\in\CC[[t]]$ with 
 $\End_{t\to s}(\chi_1)=~\!\!\chi_2$,  $s(0)=0$, and $ s'(0)\ne0$. 

If 
\beq\label{easy} s=s'(0)t \eeq then the statement of (c) is straightforward. 
Assume to the contrary that $s\ne s'(0)t$. Then $$s(t)\ne s'(0)t\mod t^d$$ for some $d$, and we choose $d$ to be minimal. 

If $d\ge 2k+2$ then linear operators $\End_{t\to s}$ and $\End_{t\to s'(0)t}$ coincide after the restriction to $\Loc_0^{\le n}$, i.e. we can replace $s$ by $s'(0)t$, which is case \eqref{easy}.

If $d\le 2k+1$ then $$s = s'(0)t+\frac{s^{(d)}(0)}{d!}t^n\mod t^{d+1}$$ with $s^{(d)}(0)\ne0$; set $\gamma=s'(0), \tau=\frac{s^{(d)}(0)}{d!}$. 
Further
\eqref{star1}, \eqref{star2} imply
$$\End_{t\to s}(\chi_1)=\gamma^{-2k}e_{2k}+\gamma^{i-2j-1}\tau (2k-2d) e_{2k-d} \mod \Loc_0^{\le 2k-d}.$$
This cannot give $\chi_2$ unless $d =k$.

So suppose $d=k$, and  note that 
$$\End_{t\to t+\alpha t^{k+1}+(\alpha^2(k+1)-\alpha \frac{\beta}2)t^{2k+1}}(e_{2k}^*+\beta e_k^*)=(e_{2k}^*+\beta e_k^*)$$ for every $\alpha\in\CC$. 
Thus we can replace $s$ by 
\begin{equation}s+\alpha s^{k+1}+(\alpha^2(k+1)-\alpha \frac{\beta}2)s^{2k+1}\label{Estab}\end{equation}
for an arbitrary $\alpha$. 
Pick $\alpha=-\frac{s^{(d+1)}(0)}{(d+1)!(s'(0))^{d+1}}$. 
Then \eqref{Estab} is equal to $s'(0)t \mod t^{d+1}$ and we have reduced this case to the previous one.
\end{proof}

\begin{remark}
If $\chi \in W^*$ is a one-point local function, the  beginning of the proof of Theorem~\ref{Texpw} provides an explicit description of $W^\chi$. 
Indeed, if $\chi=\chi_{x; \beta_0, \beta_1, \ldots, \beta_n}$ with $\beta_n\ne0$ then $W^\chi=W((t-x)^{n+1})$ if $n$ is even, and $W^\chi$ is the sum of $W((t-x)^{n+1})$ with a one-dimensional subspace if $n$ is odd.\end{remark}

A  modification of Theorem~\ref{Texpw} holds for $\mf g = W_{\ge1}$ and $x=0$ with a very similar proof, which we leave to the reader.
\begin{theorem}\label{Texpw+} Assume $\mf g=W_{\ge1}$.  
Fix $n \geq 2$ and $ \beta_2, \ldots, \beta_n$ with $\beta_n\ne0$ and set $\chi=\chi_{0; 0, 0, \beta_2,..., \beta_n}$. 

\begin{itemize}
\item[(a)] If $n$ is even then $\DLoc_0^{\le n}\chi=\DLoc_0^{\le n}e_{n-1}^*$. We have $\dim \DLoc_0^{\le n}\chi=n-1.$ 

\item[(b)] If $n$ is odd and $n>1$ then there is $\beta$ such that $\DLoc_0^{\le n}\chi=\DLoc_0^{\le n} (e_{n-1}^*+\beta e_k^*)$ where $k=\frac{n-1}2$. We have $\dim \DLoc_0^{\le n} \chi=n-2$. 

\item[(c)] Pick $\beta_1, \beta_2\in\CC$ and $k\in\mathbb Z_{\ge1}$. 
Then $\DLoc_0^{\le n}(e_{2k}^*+\beta_1 e_k^*)= \DLoc_0^{\le n}(e_{2k}^*+\beta_2 e_k^*)$ if and only if $\beta_1=\pm\beta_2$.
\end{itemize}
\qed
\end{theorem}

We next describe the ``Bruhat order'' on $\DLoc_x^{\leq n}$ orbits (that is, the inclusions of orbit closures):  it turns out that it can almost  be computed just from the dimension of the orbit.  
\begin{corollary}\label{Cdim} Let $\mf g, x$ be as in the statement of Theorem~\ref{Texpw}.
Let $\chi^-, \chi^+$ be local functions on $\mf g$ based at $x$ and of order $\leq n$. 

Then the following conditions are equivalent:
\begin{itemize}
\item[(1)] $\DLoc_x^{\le n}\cdot \chi^-$ is contained in  the closure of $\DLoc_x^{\le n}\cdot \chi^+$ and $\DLoc_x^{\le n}\cdot \chi^-\ne \DLoc_x^{\le n}\cdot \chi^+$; 
\item[(2)]  $\CoreP(\chi^+)\subsetneqq\CoreP(\chi^-)$;
\item[(3)] \label{Edimo}$\dim \DLoc_x^{\le n}\cdot \chi^-<\dim \DLoc_x^{\le n}\cdot \chi^+$ and $\chi^+$ is not of order 1. 
\end{itemize}
\end{corollary}
\begin{proof}
We again give the proof for $\mf g = W_{\geq -1}$ and $x=0$. 
We will show that $(2)\Rightarrow (3) \Leftrightarrow (1)  \Rightarrow (2)$. 
Next, without loss of generality we assume that $n$ is the order of $\chi^+$.

We first show that  $(3)$ implies $(1)$.
As in Theorem~\ref{Texpw}, pick the related presentation $\chi_{0; \beta_0^+, \beta_1^+, \ldots, \beta_n^+}$ for $\chi^+$, with $\beta_n^+ \neq 0$. 
If $n$ is even then the closure of $\DLoc_0^{\le n}\cdot \chi^+$ equals $\Loc_0^{\le n}$. 
One the other hand all orbits of strictly smaller dimension belong to $\Loc_0^{\le n}$, so this completes the proof in this case. 

If $n=2k+1$ is odd then $k\geq 1$ by the assumptions of (3).
The closure of $\DLoc_0^{\le n}\cdot \chi^+$ is an irreducible subvariety of $\Loc_0^{\le n}$ of codimension 1, i.e. a hypersurface defined by some function $F_+$. 
We claim that this function is semiinvariant with respect to $\DLoc_0^{\le n}$.

To prove that $F_+$ is semiinvariant, we describe it in more detail.
Let $R=\kk[\beta_0, \beta_1,\ldots, \beta_{n-1}, \beta_n^{\pm1}]$, where $\beta_0, \beta_1, \ldots, \beta_n$ are  free variables. Denote by $(\cdot)_{\kk\to R}$ the base change from $\kk$ to $R$. 
The arguments of Theorem~\ref{Texpw}(b) imply that there exists a group element $g\in (\DLoc_0^{\le n})_{\kk\to R}$ and $h\in R$ such that 
$$g\cdot \chi_{\mbox{$0$}; \underbrace{0,\ldots, 0}_{k\text{~times}}, \mbox{$h$}, \underbrace{0, \ldots, 0}_{k\text{~times}}, \mbox{$\beta_n$}}=\chi_{0; \beta_0, \beta_1,\ldots, \beta_n}.$$
Recall that $k = \frac{n-1}{2}\geq 1$.
By replacing each $\beta_i$ by $\frac{e_{i-1}}{i!}$ (cf. \eqref{star3}) we may identify $R$ with $\kk[\Loc_0^{\leq n}][e_{2k}^{-1}]$.
(Here we regard the $e_i$ as functions on $\Loc_0^{\leq n}$ in the obvious way.)
Then $h=f/e_{2k}^\ell$ for some $f\in\kk[\Loc_0^{\le n}]$ and $\ell \in\mathbb Z_{\ge0}$.

Let $\chi, \chi' \in  \Loc_0^{\leq n}$ with $e_{2k}(\chi), e_{2k}(\chi') \neq 0$.
As in the proof of Theorem~\ref{Texpw}(c), $\DLoc_0^{\leq n} \cdot \chi = \DLoc_0^{\leq n} \cdot \chi'$ if and only if 
\[ e_{2k}(\chi)^{-1}\frac{f^2(\chi)}{e_{2k}(\chi)^{2\ell}}=e_{2k}(\chi')^{-1}\frac{f^2(\chi')}{e_{2k}(\chi')^{2\ell}}. \]
The rational function 
\beq\label{F} F:= f^2/e_{2k}^{2\ell+1}\eeq
is thus $\DLoc_0^{\leq n}$-invariant and separates orbits.
Therefore, $\DLoc_0^{\leq n} \cdot \chi^+$ is the hypersurface in $\Loc_0^{\leq n} \setminus \Loc_0^{\leq n-1}$ defined by
\[ F_+ := e_{2k}(\chi^+)^{2 \ell+1} f^2 - f^2(\chi^+) e_{2k}^{2\ell+1}.\]
Note that $F_+$ is semiinvariant, as claimed.

All orbits of dimension $<n$ belong to $\Loc_0^{\le n-2}$.  
We will  check that the  closure of $\DLoc_0^{\le n}\cdot\chi^+$ contains $\Loc_0^{\le n-2}$.   
The orbit closure $\overline{\DLoc_0^{\leq n} \chi^+}$ is  defined in $\Loc_0^{\leq n}$ by $F_+$, so we must show that  $F_+|_{\Loc_0^{\le n-2}}=0.$
The restriction of $F_+$ to $\Loc^{\le n-1}_{0}$ is also a semiinvariant function; in particular, this restriction is invariant with respect to $\DLoc_0^{1+, \le n}$, see Lemma~\ref{Lfiltd}. 
The arguments of Theorem~\ref{Texpw} imply that all $n$-dimensional $\DLoc_0^{1+, \le n}$-invariant hypersurfaces of $\Loc_0^{\le n-1}$ are defined by the equality $e_{n-1}=\beta_{n-1}$ for some $\beta_{n-1}\in\kk$. 
This implies that the restriction of $F_+$ is a polynomial in $e_{n-1}$. 
The fact that $F_+$ is $\DLoc_0^{\le n}$-semiinvariant implies that 
$F_+|_{\Loc^{\le n-1}_{0}}=c (e_{n-1})^m$ for some $m \ge 0$ and $c\in\kk$. 
This implies the desired condition that $F_+|_{\Loc^{\le n-2}_{0}}=0$.

We now show that $(1)$ implies $(3)$, so suppose that $(1)$ holds.  
The dimension inequality in $(3)$ is clear. 
Thus we are left to deal with the case when $\chi^+$ has order~1; we want to show this cannot happen.

Suppose now that  $\chi^+= \chi_{x;\alpha, \beta}$ has order 1, so $\beta \neq 0$.  
By the dimension inequality, $\chi^- = 0$.
Now, Theorem~\ref{Texpw}(b$'$) shows that
$\DLoc_0^{\leq 1} \cdot \chi^+ = \DLoc_0^{\leq 1} \cdot (\beta e_0^*)$ consists of all $\chi_{0; \alpha', \beta}$ for $\alpha' \in \kk$, and so 
is defined in $\Loc_0^{\leq 1} \setminus \Loc_0^{\leq 0} = \Spec \kk[\beta_0, \beta_1, \beta_1^{-1}]$ by $\beta_1 = \beta$.  
Thus $\overline{\DLoc_0^{\leq 1} \cdot \chi^+}$ is defined in $\Loc_0^{\leq 1} \cong \Spec \kk[\beta_0, \beta_1]$ by $\beta_1= \beta$ and does not contain $0 = \chi_{0; 0, 0}$, a contradiction.

We now show that $(1) \Rightarrow (2)$.
This is an application of Proposition~\ref{Practpoi}. 
Let $X = \widetilde{\Loc}^{\leq n}_{\mf g}$ (recall that $\mf g = W_{\geq -1}$) and let $H = \widehat{\DLoc}^{\leq n}$.
Let $\mf h = \widehat{\lie}^{\leq n} = \operatorname{Lie}(H)$.
Let $\phi = \phi^{\leq n}_{\mf g}:  X \to \mf g^*$.  
Choose $x^+ \in \phi^{-1}(\chi^+)$ and $x^- \in \phi^{-1}(\chi^-)$; note here that $\phi$ is bijective onto $\Loc^{\leq n}_{\mf g} \setminus \{0\}$ so $x^{\pm}$ may be uniquely determined.
By Lemma~\ref{Llact}, $\dim \mf h_{x^\pm} = \dim \mf g_{\chi^\pm}$.
Thus by Proposition~\ref{Practpoi}, $\CoreP(\chi^+)$ is the kernel of the map $\phi^+:\Sa(\mf g) \to \kk[\overline{H \cdot x^+}]$ induced from $\phi^*: \Sa(\mf g) \to \kk[X]$ and $\CoreP(\chi^-)$ is the kernel of $\phi^-:\Sa(\mf g) \to \kk[\overline{H \cdot x^-}]$.
However, by assumption $H \cdot x^- \subseteq \overline{H \cdot x^+} \subseteq X$, giving a commutative diagram
\[ \xymatrix{
\Sa(\mf g) \ar[r]^{\phi^+} \ar[rd]_{\phi^-} & \kk[\overline{H\cdot x^+}] \ar[d] \\
& \kk[\overline{H \cdot x^-}].
}
\]
Thus $\CoreP(\chi^-) = \ker \phi^- \supseteq \ker \phi^+ = \CoreP(\chi^+)$.
The two are clearly distinct.

We still must show  (2) implies (3). 
If $(2)$ holds then $\dim \OO(\chi^-)<\dim \OO(\chi^+)$.
By the discussion above and \eqref{EGK}, 
\[\dim \mb O(\chi^\pm) = \dim  \mf g_{\chi^\pm} = \dim \DLoc^{\leq n}_x \cdot \chi^\pm +1.\]

Finally, note that if $\chi^+ = \chi_{x;\alpha,\gamma}$ has order $0$ or $1$,
by the dimension inequality we must have $\chi^- = 0$.
We have seen in Lemma~\ref{lem:Jgamma} that $\CoreP(\chi^+)$ is the kernel of 
\[ p_\gamma:  \Sa(\mf g) \to \kk[t, t^{-1}, y], \quad f\del \mapsto fy + \gamma f'.\]
We compute:
\[
p_\gamma([(t-x)\del]^2 - \del [(t-x)^2 \del]) = 
((t-x)y + \gamma)^2 - y ((t-x)^2 y + 2\gamma (t-x)) = \gamma^2\]
and so $[(t-x)\del]^2 - \del [(t-x)^2 \del]-\gamma^2 \in \ker p_\gamma$.
%
This element is contained in $\mf m_0 = \CoreP(0)$ if and only if $\gamma=0$, so $\chi^+$ cannot have order 1.
\end{proof}

For future reference, we record that Corollary~\ref{Cdim} also gives information about  $\widehat{\DLoc}^{\leq n}$ orbits.
Let 
\[ \Loc_{W_{\geq -1}}^n := \Loc_{W_{\geq -1}}^{\leq n} \setminus \Loc_{W_{\geq -1}}^{\leq n-1}\]
and similarly define $\Loc_W^n$.
We identify  $\kk[\Loc_{W_{\geq -1}}^n]$ with $\kk[x, \beta_0, \dots, \beta_n, \beta_n^{-1}]$. 
Note that each $\Loc_{W_{\geq -1}}^n$ is $\widehat{\DLoc}^{\leq n}$-stable.

The action of $\Shifts$, and thus the action of $\widehat{DLoc}$, on  $\Loc_W$ is only partially defined, as we cannot shift the base point to $0$. We will refer to {\em orbits} of this partial action with the obvious meaning:  the intersection of an orbit in $\Loc_{W_{\geq -1}}$ with $W^*$. 

\begin{corollary}\label{cor:orbitspace}
Let $\mf g = W$ or $W_{\geq -1}$.
\begin{itemize}
\item[(a)] If $n$ is even, then $\Loc_{\mf g}^n$ is a single $\widehat{\DLoc}^{\leq n}$-orbit.
\item[(b)] If $n\geq 3$ is odd, there is   $F_n \in \kk[x, x^{-1}, \beta_0, \dots, \beta_n, \beta_n^{-1}]$ (if $\mf g = W$) or in $ \kk[x, \beta_0, \dots, \beta_n, \beta_n^{-1}]$ (if $\mf g = W_{\geq -1})$ so that the $\widehat{\DLoc}^{\leq n}$-orbits are precisely the fibres of the morphism $F_n: \Loc_{\mf g}^n \to \AA^1.$
\item[(c)] Define $F_1: \Loc_{\mf g}^1 \to \AA^1\setminus \{0\}$ by $\chi_{x; \alpha, \gamma}\mapsto \gamma$. 
Then the fibres of $F_1$ are exactly the orbits of $\widehat{\DLoc}^{\leq 1}$ on $\Loc_{\mf g}^1$. 
\end{itemize}
\end{corollary}

\begin{proof}
This  result is a consequence of Theorem~\ref{Texpw} and the proof of Corollary~\ref{Cdim}.  
If $n\geq3$ is odd, note that by definition, applying an element of $\Shifts$ changes $x$ but not any of the $\beta_i$.  
Thus we may take $F_n$ to be the polynomial $F$ defined in  \eqref{F}, regarded as an element of $\kk[x, \beta_0, \dots, \beta_n, \beta_n^{-1}]$. 
The proof for  $n=1$  is similar. 
\end{proof}

It is instructive to compute that for  $n=3$, the polynomial $F_3$ given by Corollary~\ref{cor:orbitspace} is a scalar multiple of $\beta_2^2/\beta_3$.  In particular, $x$ does not occur.

\begin{remark}\label{rem:dim}
 The dimension of a $\widehat{DLoc}^{\leq n}$-orbit in Corollary~\ref{cor:orbitspace} is 
\[ 2 \lfloor \frac{n+2}2 \rfloor = \begin{cases} n+2 & \mbox{ $n$ even} \\ n+1 & \mbox{ $n$ odd} \end{cases}.\]
\end{remark}

\subsection{Implicit description of pseudo-orbits}\label{SSimd} 
Let $\mf g = W$, $W_{\geq -1}$, or $W_{\geq 1}$.  We now consider arbitrary local functions on $\mf g$.
We will provide every pair of local functions $\chi, \nu\in \mf g^*$ with an affine variety $X$ acted on by an algebraic group $H$, an open subset $U\subseteq X$, a Poisson map $\phi:U \to \mf g^*$, and a pair of points $x, y\in X$ which will satisfy the conditions of Proposition~\ref{Practpoi}. 
Thus we will reduce the mysterious infinite-dimensional case to the more easily comprehensible action of an algebraic group on an affine variety.   We will use this to determine the pseudo-orbits in $\mf g^*$.

\begin{theorem}\label{Tlfpoi} Let $ \mf g = \Vir$ or $\mf g=W$ or $W_{\ge-1}$. Let $\chi^{\MakeUppercase{\romannumeral 1}}$ and $\chi^{\MakeUppercase{\romannumeral 2}}$ be local functions in $\mf g^*$ represented by sums
$$\chi^{\MakeUppercase{\romannumeral 1}}=\sum_{i=1}^{\ell}\chi_i^{\MakeUppercase{\romannumeral 1}}, \hspace{10pt}\chi^{\MakeUppercase{\romannumeral 2}}=\sum_{i=1}^{\ell}\chi_i^{\MakeUppercase{\romannumeral 2}},\hspace{10pt}\chi_i^{\MakeUppercase{\romannumeral 1}}\in \Loc_{x_i^{\MakeUppercase{\romannumeral 1}}}, \chi_i^{\MakeUppercase{\romannumeral 2}}\in \Loc_{x_i^{\MakeUppercase{\romannumeral 2}}},$$
such that $x^{\MakeUppercase{\romannumeral 1}}_i\ne x^{\MakeUppercase{\romannumeral 1}}_j$ and $x^{\MakeUppercase{\romannumeral 2}}_i\ne x^{\MakeUppercase{\romannumeral 2}}_j$ for $i\ne j$. 
Then $\chi^{\MakeUppercase{\romannumeral 1}}$ and $\chi^{\MakeUppercase{\romannumeral 2}}$ are in the same pseudo-orbit if and only if it is possible to reorder $x^{\MakeUppercase{\romannumeral 1}}_i$ and $x^{\MakeUppercase{\romannumeral 2}}_i$ in such a way that 
\begin{equation}\label{Edloc}\DLoc_{1}( \Shift_{1-x_i^{\MakeUppercase{\romannumeral 1}}}\chi^{\MakeUppercase{\romannumeral 1}}_i)=\DLoc_{1}( \Shift_{1-x_i^{\MakeUppercase{\romannumeral 2}}}\chi^{\MakeUppercase{\romannumeral 2}}_i) \mbox{ for all $i$;}
\end{equation}
that is, 
\[
\widehat{DLoc}(\chi^{\MakeUppercase{\romannumeral 1}}_i)=\widehat{\DLoc}(\chi^{\MakeUppercase{\romannumeral 2}}_i) \mbox{ for all $i$.}
\]

In particular, if $\OO(\chi^{\MakeUppercase{\romannumeral 1}})=\OO(\chi^{\MakeUppercase{\romannumeral 2}})$ then $\lambda(\chi^{\MakeUppercase{\romannumeral 1}})=\lambda(\chi^{\MakeUppercase{\romannumeral 2}})$ and hence the supports of $\chi^{\MakeUppercase{\romannumeral 1}}$ and $\chi^{\MakeUppercase{\romannumeral 2}}$ have the same cardinality.
\end{theorem}

\begin{remark}\label{rem:dagdag}
Let $\mf g = W$ or $W_{\geq -1}$.  Fix a partition $\lambda = (m_1, \dots, m_r)$ and consider the set $V^\lambda_{\mf g}$ 
\label{ind:Vlambda}
of all local functions $\chi \in {\mf g}^*$ with $\lambda(\chi) = \lambda$.
Let $\chi = \chi_1 + \dots + \chi_r, \nu = \nu_1 + \dots + \nu_r \in V^\lambda_{\mf g}$, where each $\chi_i$ (respectively, $\nu_i$) is a one-point local function of order $m_i-1$. 
By Theorem~\ref{Tlfpoi}, $\mb O(\chi) = \mb O(\nu)$ $\iff$ $\mb O(\chi_i) = \mb O(\nu_i)$ for $1 \leq i \leq r$ (up to permutation of indices). 
Combining Corollary~\ref{cor:orbitspace} and Theorem~\ref{Tlfpoi} applied to one-point local functions we have
\begin{itemize}
\item[(a)]  if $m_i$ is odd, then $\Loc^{m_i-1}_{\mf g}$ is a single pseudo-orbit;

\item[(b)] if $m_i$ is even and $m_i\ne2$ then the pseudo-orbits in $\Loc_{\mf g}^{m_i-1}$ are parameterised by $\AA^1$; 

\item[(c)] the pseudo-orbits in $\Loc_{\mf g}^{1}$ are parameterised by $\AA^1\setminus\{0\}$.

\end{itemize}
Thus pseudo-orbits in $V^\lambda_{\mf g}$ (and the corresponding Poisson primitive ideals of $\Sa({\mf g})$) are parameterised by a symmetric product $(\AA^1)^{\times(k-k_2)}\times(\AA^1\setminus\{0\})^{\times k_2}$, where $k$ is the total number of even parts of $\lambda$ and $k_2:=|\{i\mid m_i=2\}|.$ 

More specifically, let $$
\AA(\lambda) := \begin{cases}\AA^k & \text{if}~k_2=0\\
\AA^{k-1} \times (\AA^1\setminus\{0\})& \text{if}~k_2\ne0\end{cases}.$$
\label{ind:Alambda}
Let $\mf S_{\lambda} = \{ \sigma \in \mf S_k \ | \ m_i = m_{\sigma(i)} \mbox{ for all } i\}$.
\label{ind:Slambda}
The
pseudo-orbits in $V^\lambda_{\mf g}$ are parameterised by $$
[(\AA^1)^{\times(k-k_2)}\times(\AA^1\setminus\{0\})^{\times k_2}] / \mf S_\lambda \cong \AA(\lambda).$$

In Subsection~\ref{SSpPp} we will refine this parameterisation to apply to arbitrary prime Poisson ideals of $\Sa(\mf g)$.
\end{remark}

To prove Theorem~\ref{Tlfpoi} we need two preparatory results.
\begin{lemma} \label{Lwcrt} Let $\chi \in W^*$ be a local function with $\chi =\chi_1+\ldots+\chi_s$ and $\chi_i\in \Loc_{x_i}$  such that $x_i\ne x_j$ if $i\ne j$. Then $W\cdot\chi=\bigoplus_i W\cdot \chi_i$.
Likewise, if $\chi \in W_{\geq-1}^*$, then $W_{\geq -1} \cdot \chi = \bigoplus_i W_{\geq -1}\cdot \chi_i$.
\end{lemma}
\begin{proof}This statement is basically implied by the Chinese remainder theorem, but we give the details for $W$.
It is easy to verify that $W\cdot\chi=\{B_\chi(u, \cdot)\mid u\in W\}$. 
Since $\chi$ is local, there are  $d_1, \ldots, d_s \in \ZZ_{\geq 1}$ so  that $\chi|_{W(\prod_i(t-x_i)^{d_i})}=0$.  Let $f:=\prod_i(t-x_i)^{d_i+1}$. Then $B_\chi|_{W(f)}=0$. 
By the Chinese remainder theorem, 
\begin{equation}\label{Ecrt}\CC[t, t^{-1}]/(\prod(t-x_i)^{d_i+1})\cong\bigoplus_i \CC[t]/(t-x_i)^{d_i+1}.\end{equation}
This leads to the decomposition
$$W/W(f)\cong\bigoplus_i[W(f_i)/W(f)],\mbox{ where } f_i:=\frac{f}{(t-x_i)^{d_i+1}}.$$
It is straightforward to verify that the partial restriction $B_{\chi_i}(W(f_j), \cdot)$ is nonzero only if $i=j$. 
Hence
$$W\cdot\chi=\{B_\chi(u, \cdot)\mid u\in W\}=\bigoplus_i \{B_{\chi}(u_i, \cdot)\mid u_i\in W(f_i)\} .$$
This implies the desired result.
\end{proof}

We next apply the general results from Section~\ref{BACKGROUND} and the group action from Subsection~\ref{SSgra}  to the Poisson cores of local functions and deduce some consequences.
\begin{proposition}\label{Plfpoi}
\begin{itemize}
\item[(a)] Let $\chi \in W^*$. 
Then 
\[\CoreP(\chi) \cap \Sa(W_{\geq -1}) = \CoreP(\chi|_{W_{\geq -1}}).\]
\item[(b)] Let $\mf g $ be $W$ or $W_{\geq -1}$ or $W_{\geq 1}$ and pick $\chi \in \mf g^*$. 
Then 
\[ \dim \OO(\chi) = \GK \Sa(\mf g)/\CoreP(\chi) = \dim \mf g \cdot \chi.\]
\end{itemize}
\end{proposition}
We alert the reader that the proof of Proposition~\ref{Plfpoi} establishes notation that is also  used in the proof of Theorem~\ref{Tlfpoi}.  (This notation generalises that used in the proof of the last part of Corollary~\ref{Cdim}.)
\begin{proof}
(a).  
If $\chi$ is not local then by Theorems~\ref{Tloc} and~\ref{Tloc01}  $\CoreP(\chi)$  and $\CoreP(\chi|_{W_{\geq-1}}) $ are the zero ideals, respectively, of $\Sa(W)$ and $\Sa(W_{\geq -1})$.

Thus we may assume that $\chi$ is local.
Write $\chi = \sum_{i=1}^\ell \chi_i$ with $\chi_i \in \Loc_{x_i}$ and $x_i \neq x_j$ if $i\neq j$; further assume that no $\chi_i =0$.  Let $n$ be the order of $\chi$.

Let $X := \Bigl(\widetilde{\Loc}^{\leq n}_{W_{\geq -1}}\Bigr)^{\times \ell}$, let
$Y :=\Bigl(\widetilde{\Loc}^{\leq n}_{W}\Bigr)^{\times \ell} $,
and let $U \subseteq Y$ be the complement of all the diagonals $\Delta_{ij}$ in $Y$, where for $i \neq j$ we define
\[ \Delta_{ij} = \{ (x^1, \underline{\alpha}^1, \dots, x^\ell, \underline{\alpha}^\ell) \in Y \ | \ x^i = x^j \},\]
where $\underline{\alpha}^i$ stands for the sequence $\underline{\alpha}^i_0, \underline{\alpha}^i_1, \ldots$. 
That is, $U$ gives representations of local functions as sums of one-point local functions based at distinct points.
Note that $U, Y$ are open subsets of $X$, and $H:= \Bigl(\widehat{\DLoc}^{\leq n}\Bigr)^{\times \ell}$ acts on $X$.
Let $\Sigma_W:  (W^*)^{\times \ell} \to W^*$ be the summation map, and likewise define $\Sigma_{W_{\geq -1}}$.
Let $\phi_W:  U \to W^*$ be $\Sigma_W \circ (\phi_W^{\leq n})^{\times \ell}$ and let 
$\phi_{W_{\geq-1}}:  X \to W_{\geq -1}^*$ be $\Sigma_{W_{\geq -1}} \circ (\phi_{W_{\geq -1}}^{\leq n})^{\times \ell}$.
There is a commutative diagram
\beq \label{E1}
\xymatrix{
X \ar[rr]^{\phi_{W_{\geq-1}}} && W_{\geq -1}^* \\
U \ar[rr]_{\phi_W} \ar[u]^{\subseteq} && W^* \ar[u]
}
\eeq
where the right vertical arrow is  restriction.
This induces maps
\beq\label{E2}
\xymatrix{
\kk[X] \ar[d]_{\subseteq} && \Sa(W_{\geq-1})  \ar[ll]_{\phi_{W_{\geq-1}}^*} \ar[d]^{\subseteq} \\
\kk[U] && \Sa(W). \ar[ll]^{\phi_W^*} 
}
\eeq
By Lemma~\ref{Lwcrt}, $\dim W \cdot \chi = \oplus_i \dim W \cdot \chi_i$, which is $\dim \Bigl( (\widehat{\lie}^{\leq n})^{\times \ell}\Bigr) \cdot \chi$  by Lemma~\ref{Llact}.
Let $y \in \phi_W^{-1}(\chi)$.  
By Proposition~\ref{Practpoi}, $\CoreP(\chi)$ is the full preimage under $\phi_W^*$ of $I^U(Hy)$, and likewise, $\CoreP(\chi|_{W_{\geq -1}})$ is the full preimage under $\phi_{W_{\geq-1}}^*$ of $I^U(Hy)$.  
By commutativity of \eqref{E2}, this is $\CoreP(\chi) \cap \Sa(W_{\geq -1})$.

(b) If $\dim \mf g\cdot\chi=\infty$ then the result follows from Lemma~\ref{Lrk+}. 
Thus we can assume $\mf g\cdot\chi<\infty$. 
For $\mf g = W$ or $W_{\geq -1}$ this follows from the discussion above and Proposition~\ref{Practpoi}; see in particular \eqref{EGK}.  
The proof for $\mf g = W_{\geq 1}$ is similar.  
\end{proof}

We now give the proof of Theorem~\ref{Tlfpoi}.

\begin{proof}[Proof of Theorem~\ref{Tlfpoi}] We give the proof for the case $\mf g=W$.

By definition, $\OO(\chi^{\MakeUppercase{\romannumeral 1}}) = \OO(\chi^{\MakeUppercase{\romannumeral 2}})$ if and only if $\CoreP(\chi^{\MakeUppercase{\romannumeral 1}}) = \CoreP(\chi^{\MakeUppercase{\romannumeral 2}})$.  So we compute the Poisson cores of $\chi^{\MakeUppercase{\romannumeral 1}}$ and~$\chi^{\MakeUppercase{\romannumeral 2}}$.

First, by changing $\ell$ and switching $\chi^{\MakeUppercase{\romannumeral 1}}$ and $\chi^{\MakeUppercase{\romannumeral 2}}$ if necessary, we may assume that none of the $\chi_i^{\MakeUppercase{\romannumeral 1}}$ are zero.
Let $n$ be the maximum of the orders of  $\chi^{\MakeUppercase{\romannumeral 1}}$ and of  $\chi^{\MakeUppercase{\romannumeral 2}}$; thus both are in the image of $(\Loc_W^{\leq n})^{\times \ell}$.

Define $U, V, X, H$ as in the proof of Proposition~\ref{Plfpoi}, and set up  maps as in \eqref{E1} and \eqref{E2}.  Let $\phi = \phi_W$.  
Then $\chi^{\MakeUppercase{\romannumeral 1}} = \phi(y^{\MakeUppercase{\romannumeral 1}})$, $\chi^{\MakeUppercase{\romannumeral 2}} = \phi(y^{\MakeUppercase{\romannumeral 2}})$ for some 
$y^{\MakeUppercase{\romannumeral 1}}, y^{\MakeUppercase{\romannumeral 2}}\in U$. 
As in the proof of Proposition~\ref{Plfpoi}, 
$\CoreP(\chi^{\MakeUppercase{\romannumeral 1}})$ is the full preimage under ${\phi}^*$ of the defining ideal of $I^U(Hy^{\MakeUppercase{\romannumeral 1}})$. 
Likewise,  $\CoreP(\chi^{\MakeUppercase{\romannumeral 2}})$ is the full preimage under ${\phi}^*$ of the defining ideal of $I^U(Hy^{\MakeUppercase{\romannumeral 2}})$.

Note that $\phi$ factors through the dominant map $U \to U/{\mf S}_{\ell}$, where $\mf S_\ell$ is the symmetric group.
It is clear that if we can reorder the $\chi_i^{\MakeUppercase{\romannumeral 1}}$ and $\chi_i^{\MakeUppercase{\romannumeral 2}}$ as described then $\CoreP(\chi^{\MakeUppercase{\romannumeral 1}}) = \CoreP(\chi^{\MakeUppercase{\romannumeral 2}})$.

Suppose now that $\CoreP(\chi^{\MakeUppercase{\romannumeral 1}}) = \CoreP(\chi^{\MakeUppercase{\romannumeral 2}})$.  
Then by Proposition~\ref{Practpoi}(c), 
there are open sets $U^{\MakeUppercase{\romannumeral 1}}\subseteq Hy^{\MakeUppercase{\romannumeral 1}}\cap U$ and $U^{\MakeUppercase{\romannumeral 2}}\subseteq Hy^{\MakeUppercase{\romannumeral 2}}\cap U$ such that their images in $W^*$ are the same. 
In particular, there are $h_1^{\MakeUppercase{\romannumeral 1}}, h_1^{\MakeUppercase{\romannumeral 2}}, \dots, h_\ell^{\MakeUppercase{\romannumeral 1}}, h_\ell^{\MakeUppercase{\romannumeral 2}} \in \widehat{\DLoc}^{\le n}$,
with $(h_1^{\MakeUppercase{\romannumeral 1}}\chi_1^{\MakeUppercase{\romannumeral 1}}, \ldots, h_\ell^{\MakeUppercase{\romannumeral 1}}\chi_\ell^{\MakeUppercase{\romannumeral 1}}) \in U$, 
and so that $\sum h_i^{\MakeUppercase{\romannumeral 1}} \chi_i^{\MakeUppercase{\romannumeral 1}} - \sum h_i^{\MakeUppercase{\romannumeral 2}} \chi_i^{\MakeUppercase{\romannumeral 2}} = 0$.
Now,  a  set of one-point local functions with distinct base points is linearly independent.  
Since by definition of $U$ the base points of the one-point local functions $h_i^{\MakeUppercase{\romannumeral 1}} \chi_i^{\MakeUppercase{\romannumeral 1}}$ are distinct, the only option is that the supports of $\sum h_i^{\MakeUppercase{\romannumeral 1}} \chi_i^{\MakeUppercase{\romannumeral 1}}$ and $\sum h_i^{\MakeUppercase{\romannumeral 2}} \chi_i^{\MakeUppercase{\romannumeral 2}}$ are equal, and we can reorder
$\chi_1^{\MakeUppercase{\romannumeral 2}}, \dots, \chi_\ell^{\MakeUppercase{\romannumeral 2}}$ so that $h_i^{\MakeUppercase{\romannumeral 1}} \chi_i^{\MakeUppercase{\romannumeral 1}} = h_i^{\MakeUppercase{\romannumeral 2}} \chi_i^{\MakeUppercase{\romannumeral 2}}$ for all $i$.
This is clearly equivalent to \eqref{Edloc}.
\end{proof}

\begin{remark}\label{rem:6}
Let $\mf g = \Vir$ or $W$ or $W_{\geq -1}$ and let $\chi \in \mf g^*$.
If $\chi$ is not local then $\CoreP(\chi) = (0)$ (if $\mf g = W$ or $W_{\geq-1}$) or $\CoreP(\chi) = (z- \chi(z))$ (if $\mf g = \Vir$).  If $\chi$ is local, then by Lemma~\ref{lem:basics}(a) we may construct
$\CoreP(\chi) = \bigcap \{ \mf m_\nu \ | \ \nu \in \mb O(\chi)\}$ from Theorem~\ref{Tlfpoi}.

Further combining Corollary~\ref{cor:orbitspace} with Theorem~\ref{Tlfpoi} one may in principle compute all the Poisson primitive ideals 
of $\Sa(W)$ or $\Sa(W_{\geq-1})$; the function $F_{2k+1}$ whose fibres by Corollary~\ref{cor:orbitspace} give the pseudo-orbits of a one-point local function of order $2k+1$ may be worked out  using Lemma~\ref{Lfiltd} for any $k$.
There is an example of this  Section~\ref{SSexample}, although we have not given a fully general formula. 

Note also that we have not studied generators (in any sense) for the Poisson prime ideals $\CoreP(\chi)$, and this might be an interesting subject of research. 
\end{remark}

A version of Theorem~\ref{Tlfpoi} holds for $W_{\ge1}$ and $W_{\ge0}$ with almost the same proof. We state it as follows:

\begin{theorem}\label{Tlfpoi0} Let $\mf g=W_{\ge0}$ or $W_{\ge1}$. 
Let $\chi^{\MakeUppercase{\romannumeral 1}}, \chi^{\MakeUppercase{\romannumeral 2}}\in \mf g^*$ be local functions. Then 
 there are unique local functions $\chi_W^{\MakeUppercase{\romannumeral 1}}, \chi_W^{\MakeUppercase{\romannumeral 2}}\in W^*$ and $\chi^{\MakeUppercase{\romannumeral 1}}_{[0]}, \chi^{\MakeUppercase{\romannumeral 2}}_{[0]}\in \Loc_0$ such that 
$$\chi^{\MakeUppercase{\romannumeral 1}}=\chi^{\MakeUppercase{\romannumeral 1}}_{W}|_{\mf g}+\chi^{\MakeUppercase{\romannumeral 1}}_{[0]},\hspace{10pt}\chi^{\MakeUppercase{\romannumeral 2}}=\chi^{\MakeUppercase{\romannumeral 2}}_{W}|_{\mf g}+\chi^{\MakeUppercase{\romannumeral 2}}_{[0]}.$$
Moreover $\CoreP(\chi^{\MakeUppercase{\romannumeral 1}})=\CoreP(\chi^{\MakeUppercase{\romannumeral 2}}
)$ if and only if $\CoreP(\chi^{\MakeUppercase{\romannumeral 1}}_{W})=\CoreP(\chi^{\MakeUppercase{\romannumeral 2}}_{W}
)$ and $\DLoc_0(\chi^{\MakeUppercase{\romannumeral 1}}_{[0]})=\DLoc_0(\chi^{\MakeUppercase{\romannumeral 2}}_{[0]}
)$.
\qed
\end{theorem}

The description of pseudo-orbits in Theorem~\ref{Tlfpoi} has several immediate applications.  
We start by comparing pseudo-orbits in $W^*$ and $W^*_{\geq -1}$.

\begin{corollary}\label{Csameprimspec}Let $\chi^{\MakeUppercase{\romannumeral 1}}, \chi^{\MakeUppercase{\romannumeral 2}}\in W^*$ be local functions. Then 
\[\OO(\chi^{\MakeUppercase{\romannumeral 1}})=\OO(\chi^{\MakeUppercase{\romannumeral 2}}) \
\quad \iff \quad \OO(\chi^{\MakeUppercase{\romannumeral 1}}|_{W_{\ge-1}})=\OO(\chi^{\MakeUppercase{\romannumeral 2}}|_{W_{\ge-1}}).\]
Further, restriction induces a bijection 
\beq\label{resbij} \PSpec_{\rm prim}\Sa(W) \to \PSpec_{\rm prim} \Sa(W_{\geq -1}).\eeq
\end{corollary}


 In Corollary~\ref{Csameprimespec} we will show that restriction also provides a bijection between Poisson prime ideals of $\Sa(W)$ and $\Sa(W_{\geq-1})$.

\begin{proof}
The first statement is immediate from Theorem~\ref{Tlfpoi}.
By that result, if $\chi \in W_{\geq -1}^*$ is a  local function then there is $\nu \in \mb O(\chi)$ so that $0$ is not in the support of $\nu$.  Thus restriction of local functions induces a bijection on pseudo-orbits.

By Lemma~\ref{lem:basics} there are bijections
\begin{multline*} \mbox{nonzero Poisson primitive ideals of } \Sa(W)  \leftrightarrow \mbox{ pseudo-orbits in } W^* \\
\leftrightarrow
\mbox{ pseudo-orbits in } W_{\geq -1}^* \leftrightarrow \mbox{ nonzero Poisson primitive ideals of } \Sa(W_{\geq -1}),
\end{multline*}
where the middle bijection comes from restriction of local functions as in the first paragraph of the proof.
By Proposition~\ref{Plfpoi}(a) $\CoreP(\chi) \cap \Sa(W_{\geq -1}) = \CoreP(\chi|_{W_{\geq -1}})$.
Thus the bijection from left to right above agrees with restriction of ideals from $\Sa(W)$ to $\Sa(W_{\geq-1})$.
\end{proof}

 \begin{remark}\label{rem:extension}
 Let $P = \CoreP(\chi)$ be a  Poisson primitive ideal of $\Sa(W_{\geq -1})$, where $\chi \in (W_{\geq-1})^*$ is local.  As  we may apply elements of $\Shifts$ to $\chi$ without affecting $\CoreP(\chi)$, we may assume that $\chi$ does not involve any element of $\Loc_0$.  
Then $\chi$ extends uniquely to a local function $\widetilde{\chi}$ on $W$, and we have seen that $\CoreP(\chi) = \CoreP(\widetilde{\chi}) \cap \Sa(W_{\geq -1})$. 
\end{remark}

Let $\gamma \in \kk$ and recall the definition of the map $p_\gamma$ of \eqref{pgamma} and the fact, proved in  Lemma~\ref{lem:Jgamma} that $\ker p_\gamma= \CoreP(\chi_{x;\alpha, \gamma})$.

\begin{proposition}\label{PGK2}
The ideals $\ker p_\gamma$ are all of the Poisson prime ideals of $\Sa(W)$ of co-GK~2.
\end{proposition}

\begin{proof}
Let $J(\gamma) = \ker p_\gamma$ for all $\gamma \in \kk$.
Let $J$ be a prime Poisson ideal of $\Sa(W)$ with $\GK \Sa(W)/J = 2$.
By Lemma~\ref{lem:basics} $$J = \bigcap \{ \CoreP(\chi) \ | \ \chi\in W^*, \ev_\chi(J)=0\}.$$
Fix such $\chi$.
By Lemma~\ref{Lrk+} $\rk B_\chi$, which must be even as $B_\chi$ is an alternating form, is $0$ or $2$.  The only $\chi \in W^*$ with $\rk B_\chi = 0$ is $\chi = 0$; as $J \neq \CoreP(0) = \mf m_0$ there is some $\nu \in W^*$ with $\ev_\nu(J) = 0$ and $\rk B_\nu = 2$.

If $J \subsetneqq \CoreP(\nu) $ then $2 = \GK \Sa(W)/J > \GK\Sa(W)/\CoreP(\nu)$, but this last is $\dim W \cdot \nu$ by Proposition~\ref{Plfpoi}, which must be an even positive integer.
Thus $J = \CoreP(\nu) $ and $ \dim W\cdot \nu = 2$.

Lemma~\ref{Lwcrt} implies that $\nu$ is a one-point local function, and by Theorem~\ref{Texpw} we must have $\nu = \chi_{x; \alpha, \gamma}$ for some $x \in \kk^\times, \alpha, \gamma \in \kk$.  So $J = \CoreP(\nu) = J(\gamma)$.
\end{proof}

\begin{remark}\label{rem:GK3} 
The codomain of the maps $p_\gamma$ of \eqref{pgamma} is the localised Poisson-Weyl algebra
 $B = \kk[t, t^{-1}, y]$ with $\{ y, t \}=  1$.
Define a Poisson bracket on $B[s]$ by setting $s$ to be Poisson central, and define
\[ \Phi:  \Sa(W) \to B[s], \quad \quad f\del \mapsto fy + sf'.\]
As in Subsection~\ref{SSexample}, it is easy to check that $\Phi$ is a Poisson map.
One can show using the methods of the proof of Proposition~\ref{PGK2} that $\ker \Phi$ is the unique prime Poisson ideal of $\Sa(W)$ of co-GK 3.  
This is proved in Remark~\ref{rem:GK3two} with a different argument, so we do not give a proof here.
\end{remark}

Let $\mf g$ be a Lie algebra and let $\chi \in \mf g^*$.
The closure of $\OO(\chi)$ is, by definition, $V(\CoreP(\chi))$.   
We call $V(\CoreP(\chi)) =\overline{\OO(\chi)}$ the {\em orbit closure} of $\chi$.
\label{ind:orbitclosure}
(Technically, this term should probably be ``pseudo-orbit closure''; we have used wording which we find more pleasant at the cost of a slight abuse of terminology.)
Note that $\mu$ is in the orbit closure of $\chi$ if and only if $\CoreP(\mu)\supseteq \CoreP(\chi)$.

The orbit closure relations for one-point local functions are essentially given by Corollary~\ref{Cdim}, but for arbitrary local functions they are quite complex.  
Consider the following example, which for simplicity we give for $W_{\geq -1}$ only:
\begin{example}\label{ex:derivative}
Let  $\kk = \mb C$, let $x \neq y \in \mb C$, and let $a,b,c,d \in \mb C$ with $(a,b)$ not both 0 and likewise for $(c,d)$.
Let $\chi = \chi_{x;a,b} + \chi_{y; c, d}$. 
We claim that $Loc_{W_{\geq -1}}^{\leq 1} \subseteq \overline{\mb O(\chi)}$, and in particular $\overline{\mb O(\chi)}$ contains the functions $0$, $\chi_{x;a,b}$, and $\chi_{y;c,d}$.

To see this, let $\gamma,s \in \mb C^\times$.  
 Consider the local function
\[ \chi(\gamma,s) := \chi_{s;\gamma/s,b} + \chi_{0; -\gamma/s,d} \in W_{\geq -1}^*.\]
By Theorems~\ref{Tlfpoi} and \ref{Texpw}, the  $\chi(\gamma,s)$ are all in $\OO(\chi)$.
For $f \in \kk[t]$, we have
\[ \lim_{s \to 0} \chi(\gamma, s)(f\del) = \lim_{s\to0} \gamma\frac{f(s)-f(0)}{s} + b f'(s) + d f'(0) = (\gamma+b+d)f'(0).\]
As Zariski closed sets are closed in the complex analytic topology, 
$$\lim_{s \to 0} \chi(\gamma,s) = \chi_{0; 0, \gamma+b+d} \in \overline{\OO(\chi)}$$ for any $\gamma$.
Applying Corollary~\ref{Cdim}, $\overline{\OO(\chi)}$ contains all of $\Loc_{W_{\geq-1}}^{\leq 1}$.

The same statement holds for arbitrary $\kk$ by the Lefschetz principle.
\end{example}

Although arbitrary orbit closures are complicated, we are able to give a partial description of orbit closures of general local functions in the next corollary.

\begin{corollary}\label{C:sumoflfs}
Let $\mf g = \Vir$  or $W$ or $W_{\geq -1}$ or $W_{\geq 1}$.  \begin{itemize}
\item[(a)]
Let $\mu, \chi \in \mf g^*$ be local functions with disjoint support, and let $\nu$ be in the orbit closure of $\mu$.
Then  $\nu + \chi$ is in the orbit closure of $\mu + \chi$.  
\item[(b)]
Let $\chi$ be a local function.  Then $0$ is not in the orbit closure of $\chi$ if and only if $\lambda(\chi) = (2)$, i.e. $\chi = \chi_{x;\alpha, \beta}$ for some $\beta \in \kk^\times$.
\end{itemize}
\end{corollary}


\begin{proof}
(a).
By Theorem~\ref{Tlfpoi}, the pseudo-orbit of $\mu+\chi$ consists of all local functions of the form $\mu'+\chi'$ where $\mu' \in \OO(\mu)$,  $\chi' \in \OO(\chi)$, and the supports of $\mu'$ and $\chi'$ are disjoint.
A  modification of Corollary~\ref{Cdim} for the multipoint case implies that the orbit closure $\overline{\OO(\mu+\chi)}$ contains $\nu+\chi$.

For (b), Corollary~\ref{Cdim} gives that if $\lambda(\chi) = (2)$ then $0 \not \in \overline{\OO(\chi)}$.

Assume now that $0 \not \in \overline{\OO(\chi)}$.  
Write
\[ \chi = \mu_1 + \dots +\mu_\ell + \nu_1 + \dots + \nu_r,\]
where the $\mu_i, \nu_j$ are 1-point local functions supported at distinct points, the $\mu_i$ have order $0$ or $1$, and the $\nu_j$ have order $n_j > 1$. 
By (a), $0$ must not be in the orbit closure of at least one of the component 1-point local functions of $\chi$, so $\ell \geq 1$ by Corollary~\ref{Cdim}.

Write $\mu_i = \chi_{x_i;\beta_i, \gamma_i}$ and 
$\nu_j  = \chi_{y_j; \alpha^j_0, \dots, \alpha^j_{n_j}}$, where $\alpha^j_{n_j} \neq 0$.
If $r \geq 1$ then let $\mu_i' = \chi_{y_1;\beta_i,\gamma_i}$.
As we may move points within an orbit, 
\[ \mu_1' + \dots + \mu_\ell' + \nu_1,\]
which is a 1-point local function of order $n_1 > 1$, is in the orbit closure of $\mu_1 + \dots +\mu_\ell+ \nu_1$.
By part (a) and the previous paragraph, then, 0 is in the orbit closure of $\chi$, a contradiction.
Thus $r=0$.

If $\ell \geq 2$, then, applying Example~\ref{ex:derivative} and part (a) repeatedly,  we have $0 \in \overline{\mb O(\chi)}$.  Thus $\ell=1$, and by Corollary~\ref{Cdim} $\chi = \mu_1$ must have order 1.
\end{proof}



We now apply our results to classify maximal Poisson ideals of $\Sa(\Vir)$.

\begin{corollary}\label{cor:augmentation}
The maximal Poisson ideals of $\Sa(\Vir)$ are the ideals $(z-\zeta)$ for $\zeta \in \kk^\times$, the ideals $\CoreP(\chi_{1;0, \gamma})$ for $\gamma \in \kk^\times$, and the augmentation ideal $
\CoreP(0)$.
\end{corollary}
\begin{proof}
That $(z-\zeta)$ is a maximal Poisson ideal if $\zeta \in \kk^\times$ is Corollary~\ref{cor:Psimple}.  
If $Q$ is any other maximal Poisson ideal then we must have $z \in Q$; it follows from the Nullstellensatz and Theorem~\ref{Tloc} that there is a local function $\chi$ with $Q = \CoreP(\chi)$.
If $\lambda(\chi) \neq (2)$ then by Corollary~\ref{C:sumoflfs} we have $\CoreP(\chi) \subseteq \CoreP(0)$ so we must have  $\chi = 0$, and $Q$ must be the augmentation ideal $\CoreP(0)$.

If $\lambda(\chi) = (2)$ then $\CoreP(\chi)$ is maximal among Poisson primitive ideals by Corollaries~\ref{Cdim} and~\ref{C:sumoflfs}, and thus maximal among Poisson ideals.
\end{proof}

\section{Prime Poisson ideals}\label{PDME}

In this section we apply our results on Poisson primitive ideals to the structure of arbitrary Poisson prime ideals of $\Sa(\mf g)$, where $\mf g$ is one of our Lie algebras of interest.
After a preliminary subsection 
on birational maps in our context, 
we  show that the partition data associated to local functions may be used to parameterise Poisson primes of $\Sa(W)$ and $\Sa(W_{\geq 1})$.
This will allow us to show that the bijection $\PSpec_{\rm prim}\Sa(W) \to \PSpec_{\rm prim}\Sa(W_{\geq-1})$ of Corollary~\ref{Csameprimspec} extends to prime Poisson ideals.
We then investigate when the Poisson analogue of the Dixmier-Moeglin equivalence holds for $\Sa(\mf g)$ and, in particular, when Poisson primitive ideals are locally closed in $\PSpec \Sa(\mf g)$; we will see that this is almost always, but not always, the case.
Finally, we answer a question of Le\'on S\'anchez and the second author \cite{LSS} on heights of Poisson prime ideals, and in the process show that $\PSpec \Sa(W)$ has the somewhat counterintuitive property that every proper radical Poisson ideal contains a proper Poisson primitive ideal.

 \subsection{Birational maps}\label{SSChev}


Chevalley's fundamental result that images of algebraic maps are constructible holds for morphisms of finite presentation \cite[Theorem~10.29.10, Tag 00FE]{Stacks} and thus applies to any homomorphism from $\Sa(\mf g) $ to an affine algebra.  
The main result of this subsection is related to Chevalley's theorem, and allows us to conclude that two domains are birational, even in our non-noetherian context.  

\begin{proposition} \label{prop:iso} 
Let $A$, $B$ be commutative $\kk$-algebras which are domains, and let $\phi:  A\to B$ be an injective homomorphism.
Assume also that $B$ is affine, (we make no additional assumption on $A$) 
and that the map $\MSpec B\to \MSpec A$ is also injective. 
Then there is $f\in A \setminus \{ 0\}$ so that the natural map $A[f^{-1}]\to B[f^{-1}]$ is an isomorphism.
\end{proposition}

Before proving Proposition~\ref{prop:iso} we provide a preliminary lemma.
\begin{lemma}\label{lem:chevalley}
Let $A$, $B$ be commutative $\kk$-algebras which are domains, and let $\phi:  A\to B$ be an injective homomorphism.
Assume also that $B$ is affine (we make no additional assumption on $A$).
Let $b \in B \setminus \{ 0\}$.  
There is $a\in A \setminus \{ 0\}$ so that if $\chi:  A\to \kk$ is a homomorphism with $\chi(a) \neq 0$, then there is some $\chi':B \to \kk$ with $\chi' \phi = \chi$ and $\chi'(b) \neq 0$.
\end{lemma}
\begin{proof}
Without loss of generality $A \subseteq B$ and $\phi$ is the inclusion. 
Denote by $B[b^{-1}]$ the localisation of $B$ by $b$.
Since $B$ is affine, $B[b^{-1}]$ is finitely generated over $A$. 
By Chevalley's theorem  the image of $\MSpec B[b^{-1}]$ contains an open subset $\{\mf m\in \MSpec A\mid a\notin \mf m\}$ of $\MSpec A$ defined by some element $a\in A\setminus \{0\}$. 
This implies that $a\in A$ satisfies the desired property. 
%
\end{proof}

\begin{proof}[Proof of Proposition~\ref{prop:iso}]
Since $\dim_\kk A \leq \dim_\kk B \leq \aleph_0 < |\kk|$, $A$ is a Jacobson ring and the Nullstellensatz applies.  In particular, if $\mf m \in \MSpec A$ then the natural map $\kk \to A/\mf m$ is an isomorphism.

Without loss of generality $A \subseteq B$ and $\phi$ is the inclusion.
Let $Q(A)$ and $Q(B)$ be the quotient fields of $A$ and $B$ respectively. Assume that $Q(B)$ is not algebraic over $Q(A)$. 
Then there exists $x\in B$ so that $x$ is not algebraic over $Q(A)$, i.e. the natural map $A[t]\to B, t\mapsto x$  is injective; we denote the image of this map by $A[x]$. 

Choose $p\in A[x]\setminus \{0\}$ to be the element defined by Lemma~\ref{lem:chevalley} for $A=A[x], B=B, b=1$. 
Let $p_0\in A$ be a nonzero  coefficient of $p$.  
Then, for every $\mf m\in\MSpec A$ with $p_0\not\in \mf m$ there are at most finitely many $x_0\in\kk$ such that $p(x_0)\in \mf  m$. Every other $x_0$ provides a maximal ideal $(\mf m, x-x_0)$ of $A[x]$ which extends to a maximal ideal of $B$ thanks to Lemma~\ref{lem:chevalley}. Thus every such $\mf m$ has uncountably many preimages in $\MSpec B$, contradicting our assumption. 

Thus we can assume that $Q(B)$ is algebraic over $Q(A)$. 
Let $B$ be generated as a $\kk$-algebra by $b_1, \dots, b_k$.
If $Q(A)=Q(B)$ then every $b_i$ can be written $\frac{a^+_i}{a^-_i}$ with $a_i^\pm\in A\setminus\{0\}$; further we have that $A[f^{-1}]= B[f^{-1}]$ for $f=a_1^-\dots a_k^-$. 
Thus we can assume that $Q(A)\ne Q(B)$ and hence there exists $x\in B\setminus Q(A)$. 

As $x$ is algebraic over $Q(A)$ there are $a_0, a_1, \ldots, a_n\in A$ with $a_0\ne0, a_n\ne0$ such that 
\[F(x) := a_0 x^0+a_1 x^1+a_2 x^2+\ldots+a_nx^n=0\]
with $n\ge 2$. 
Let $p=\frac{p^+}{a_n^d}$ with $p^+\in A\setminus\{0\}, d\in\NN$ be the resultant of $F$. 
Consider $\mf m\in\MSpec A$ with $a_n, p^+\not\in \mf m$. 
Then it is well-known that there are exactly $n$ maximal ideals  $\mf m'$ of $ A[x]$ with $\mf m=\mf m' \cap A$. 
Pick $y\in A[x]$ to be the element defined by Lemma~\ref{lem:chevalley} for $A=A[x], B=B, b=1$. 
This element $y$ is algebraic over $A$ and thus there are $a_0', a_1', \ldots\in A$ with $a_0'\ne0$ such that 
\[a_0' y^0+a_1' y^1+a_2' y^2+\ldots = 0 .\]
If $\mf m'$ is a maximal ideal of $A[x]$ such that $a_0'\not\in \mf m=\mf m'\cap A$ then we have $y\notin \mf m'$. 

Pick $\mf m\in\MSpec A$ such that $a_np^+a_0'\notin \mf m$.
Then there are at least $n$ maximal ideals $\mf m''$ of $B$ such that $\mf m''\cap A=\mf m$. 
This contradicts the injectivity of $\MSpec B\to \MSpec A$.
\end{proof}



\subsection{Parameterising Poisson primes}
\label{SSpPp}

Throughout this subsection, let  $\mf g = W$ or $W_{\geq -1}$.  
We will refine Remark~\ref{rem:dagdag} to  a parameterisation of Poisson primes of $\Sa(\mf g)$; we will use this to show in Corollary~\ref{Csameprimespec} that the bijection of Corollary~\ref{Csameprimspec} extends to prime Poisson ideals.

Given a local function $\chi$, recall the definition of the partition $\lambda(\chi)$ from Section~\ref{Spridlf}.
Write $\lambda(\chi) = (m_i)$ and 
 define
\[ D(\lambda(\chi)) := 2 \sum \lfloor \frac{m_i+1}2 \rfloor. \]
\label{ind:Dlambda}
By Proposition~\ref{Plfpoi} and Remark~\ref{rem:dim}, 
\beq\label{dimlambda}
\dim \OO(\chi) = D(\lambda(\chi)), 
\eeq 
and in particular $\dim \OO(\chi)$ depends only on $\lambda(\chi)$.

Fix a partition $\lambda = (m_1 \geq m_2 \geq \cdots \geq m_r)$, where $m_i \in \ZZ_{\geq 1}$, and let 
\[ V^\lambda_{\mf g} := \{ \chi \in \mf g^* \ | \ \lambda(\chi) = \lambda\}.\]
\label{ind:Vlambda2}
Let 
\[\Loc^\lambda_{\mf g}:= \prod_{i=1}^r \Loc^{m_i-1}_{\mf g},\]
\label{ind:Loclambda}
where recall that $\Loc^{m}_{\mf g} = \{ \chi_{x; \alpha_0, \dots, \alpha_m} \in \mf g^* \ | \ \alpha_{m} \neq 0\}.$
There is a natural summation map $\Sigma^\lambda_{\mf g}: \Loc^\lambda_{\mf g} \to \mf g^*$
\label{ind:Sigma}
 defined by $\Sigma^\lambda_{\mf g}(\chi_1, \dots, \chi_r) = \sum \chi_i$.  
As in the proof of Proposition~\ref{Plfpoi}, let 
\[ U^\lambda_{\mf g} := \{ (\chi_{x_1;\underline{\alpha}^1}, \dots, \chi_{x_r;\underline{\alpha}^r}) \in \Loc^\lambda_{\mf g} \ | \ \mbox{the $x_i$ are all distinct }\}\]
(here $\underline{\alpha}^i$ stands for a sequence $\underline{\alpha}^i_0, \underline{\alpha}^i_1, \ldots$). 
Note that $\Sigma^\lambda_{\mf g}(U^\lambda_{\mf g}) $ is precisely $V^\lambda_{\mf g}$, whereas $\Sigma^\lambda_{\mf g}(\Loc^\lambda_{\mf g})$ usually includes other local functions. 

Let the partition
\[ \lambda^{\rm e} := (m_i \ | \ \mbox{$m_i$ is even} )\]
consist of the even parts occurring in $\lambda$ (with the same multiplicity); set $k:= \len(\lambda^{\rm e})$, $k_2:=|\{i\mid m_i=2\}|$.  
For any $m_i \in \lambda^{\rm e}$ let 
\[ \pi_i:  \Loc^{m_i-1}_{\mf g} \to \begin{cases} \AA^1 & \mbox{ if $m_i > 2$} \\ \AA^1 \setminus \{0\} & \mbox{ if $m_i=2$} \end{cases}
\]
be the morphism $F_{m_i-1}$ defined in Corollary~\ref{cor:orbitspace}.
The fibres of $\pi_i$ are pseudo-orbits (that is, $\widehat{\DLoc}^{\leq m_i-1}$-orbits) in $\Loc^{m_i-1}_{\mf g}$.  

Let
\[ \pi^\lambda_{\mf g}:  \Loc^\lambda_{\mf g} \to \AA^{k-k_2}\times(\AA^1\setminus \{0\})^{k_2}\]
be the composition
\[ \xymatrix{
\Loc^\lambda_{\mf g} \ar[r]^{\rm pr} & \Loc^{\lambda^{\rm e}}_{\mf g}\ar[r]^(0.2){\prod \pi_i} & [\AA^1 \times \dots ]\times[(\AA^1\setminus \{0\}) \times \dots] = \AA^{k-k_2}\times (\AA^1\setminus \{0\})^{\times{}k_2},
}
\]
where ${\rm pr}$ denotes projection.

Now, $\widehat{\DLoc}^{\leq m_i-1}$-orbits in $\Loc^{m_i-1}_{\mf g}$ are nonsingular and all are of dimension $2 \lfloor \frac{m_i+1}2 \rfloor$.
Further, $\pi_i$ is flat:  any torsion-free module over $\kk[x]$ or $\kk[x, x^{-1}]$ is flat. 
Thus, by \cite[Theorem~III.10.2]{Hartshorne}, $\pi_i$ is smooth, and so $\pi^\lambda_{\mf g}$ is also smooth.

Given a partition $\lambda = (m_i)$ as above, we also write
\[ \lambda = (\underbrace{n_1, \ldots,n_1}_{j_1},\underbrace{n_2, \ldots,n_2}_{j_2}, \dots, \underbrace{n_\ell, \ldots,n_\ell}_{j_\ell}),\]
where $n_1 > n_2 > \dots > n_\ell$.
Note that $m_1= \dots = m_{j_1} = n_1$, $m_{j_1+1} = m_{j_1+2} = \dots = m_{j_1+j_2} = n_2$, etc., and $r =\len(\lambda)= \sum_{i=1}^\ell j_i  $.
The multisymmetric group $\mf S_\lambda := \mf S_{j_1} \times \cdots \times \mf S_{j_\ell}$ 
acts on $\Loc^\lambda_{\mf g}$ and on $U^\lambda_{\mf g}$ by permuting the factors.
Further, $\mf S_\lambda$ also acts on $\AA^{k-k_2}\times (\AA^1\setminus 0)^{\times{}k_2}$ by permutations (if $n_i$ is odd, the factor $\mf S_{j_i}$ acts trivially).
We will name the action map
\beq\label{muaction} \mu_\lambda:  \mf S_\lambda \times \AA^{k-k_2}\times (\AA^1\setminus 0)^{\times{}k_2} \to \AA^{k-k_2}\times (\AA^1\setminus 0)^{\times{}k_2} \eeq
since we will need to refer to it later.
As in Remark~\ref{rem:dagdag}, denote  the  quotient  $[(\AA^1)^{\times(k-k_2)}\times(\AA^1\setminus\{0\})^{\times k_2}] / \mf S_\lambda$ by $\AA(\lambda)$.
\label{ind:Alambda2}

Note that $\pi^\lambda_{\mf g}$ is $\mf S_\lambda$-equivariant.
There is thus an induced map
\[ \overline{\pi^\lambda_{\mf g}} : U^\lambda_{\mf g}/ \mf S_\lambda \to \AA(\lambda).\]
We remark that $U^\lambda_{\mf g}/\mf S_\lambda$ and $\AA(\lambda)$ are affine. 
Further, $\AA(\lambda)$ is nonsingular --- in fact, isomorphic to $\AA^k$ or $\AA^{k-1}\times(\AA^1\setminus\{0\})$ --- since $\mf S_\lambda$ is a reflection group, see also Remark~\ref{rem:dagdag}.  
The fibres of $\overline{\pi^\lambda_{\mf g} } $  are isomorphic to fibres of the smooth morphism $\pi^\lambda_{\mf g}$ and are thus nonsingular.  Therefore $U^\lambda_{\mf g}/ \mf S_\lambda $  is also nonsingular.

The summation map $\Sigma^\lambda_\mf g$ factors through $\Loc^\lambda_\mf g/\mf S_\lambda$.
This induces a morphism of varieties
\[ \psi^\lambda_{\mf g}:  U^\lambda_{\mf g}/\mf S_\lambda \to \mf g^*,\]
which is  injective, as any collection of one-point functions with distinct supports is linearly independent.
The image of $\psi^\lambda_{\mf g}$ is precisely the set $V^\lambda_{\mf g}$ of local functions with order partition $\lambda$. 
Further note that $\psi_\mf g^\lambda$ maps the tangent space to a $\widehat{\DLoc}$-orbit of $\chi\in\mf g^*$ at $\chi^*$ to $\mf g_{\psi^\lambda_\mf g(\chi)}$ (see Lemma~\ref{Llact} and~\ref{Lwcrt}), i.e. it satisfies the conditions of Propositions~\ref{Prpoi} and~\ref{Practpoi}.

\begin{remark}\label{rem:countableU}
We may identify $V^\lambda_{\mf g}$ with $U^\lambda_{\mf g}/\mf S_{\lambda}$ as a set.
This construction allows us to write the set of local functions on $\mf g$ as a countable union of the affine varieties $U^\lambda_{\mf g}/\mf S_{\lambda}$, further justifying the statement in Remark~\ref{rem:dag} that most elements of $\mf g^*$ are not local.
\end{remark}

We abuse notation slightly, and identify $\Loc^\lambda_W$, via restriction, with  the subset of $\Loc^\lambda_{W_{\geq-1}}$ consisting of  functions whose support does not contain $0$.
 The diagram
 \beq \label{diagram} \xymatrix@C=3em{
 W^* \ar[rr]^{\rm res} && W_{\geq-1}^* \\
 U^\lambda_W/\mf S_\lambda \ar[u]^{\psi^\lambda_W}
 \ar[rr]^{\subset} \ar[rd]_{\overline{\pi^\lambda_W}} && 
 U^\lambda_{W_{\geq -1}}/\mf S_\lambda \ar[u]_{\psi^\lambda_{W_{\geq-1}}} \ar[ld]^{\overline{\pi^\lambda_{W_{\geq-1}}}} \\
 & \AA(\lambda) &
 }
 \eeq
 commutes, and both  maps to $\AA(\lambda)$ are surjective.

Let $\widehat{\DLoc}^\lambda := \widehat{\DLoc}^{\leq m_1-1} \times \cdots \times\widehat{\DLoc}^{\leq m_r-1}$.
\label{ind:whDLoclambda}
  The group $\widehat{\DLoc}^\lambda$ acts on $ \Loc^\lambda_{W_{\geq-1}}$, and there are partial $\widehat{\DLoc}^\lambda$-actions on $\Loc^\lambda_W$, on $U^\lambda_W$, and on $U^\lambda_{W_{\geq-1}}$; as usual, we speak of $\widehat{\DLoc}^\lambda$-orbits on these varieties to mean the intersections with full orbits in $ \Loc^\lambda_{W_{\geq-1}}$.
This (partial) action does not commute with the $\mf S_\lambda$-action. 
Nevertheless, it is compatible with the natural action of $\mf S_\lambda$ on $\widehat{\DLoc}^\lambda$ and together these provide an action of the semidirect product $\widehat{\DLoc}^\lambda\rtimes \mf S_\lambda$. 
Further by Corollary~\ref{cor:orbitspace} the fibres of $\overline{\pi^\lambda_{\mf g}}$ are precisely the images in $U^\lambda_{\mf g}/\mf S_\lambda$ of $\widehat{\DLoc}^\lambda\rtimes \mf S_\lambda$-orbits in $U^\lambda_{\mf g}$.
(In particular, they all have dimension $D(\lambda)$.)  
It follows from Theorem~\ref{Tlfpoi} that these fibres correspond under $\psi^\lambda_{\mf g}$ precisely with pseudo-orbits of local functions with partition $\lambda$.
We will refer to the image in $U^\lambda/\mf S_\lambda$ of a $\widehat{\DLoc}^\lambda\rtimes \mf S_\lambda$-orbit in $U^\lambda_{\mf g}$ as simply a {\em $\widehat{\DLoc}^\lambda$-orbit quotient}.  
\label{ind:orbitquotient}
We abuse terminology and say a subset or subvariety of $U^\lambda/\mf S_\lambda$  is {\em $\widehat{\DLoc}^\lambda$-invariant} if it is a union of $\widehat{\DLoc}^\lambda$-orbit quotients. 

We can use the maps $\psi^\lambda_{\mf g}$ defined above to strengthen Proposition~\ref{Prpoi} in our setting.

\begin{proposition}\label{prop:poisson}
Let $\mf g = W$ or $\mf g = W_{\geq-1}$.
Let $Q$ be  a proper radical ideal of $\Sa(\mf g)$ and let $Z := V(Q)$.  
Suppose also that for some $\lambda$ the set $V^\lambda_{\mf g} \cap Z$ is (Zariski) dense  in $Z$.
Then $Q$ is Poisson if and only if $(\psi^\lambda_{\mf g})^{-1}(Z \cap V^\lambda_{\mf g})$ is  $\widehat{\DLoc}^\lambda$-invariant.
\end{proposition}
\begin{proof}
Let $\psi := \psi^\lambda_{\mf g}$, let $U := U^\lambda_{\mf g}/\mf S_\lambda$, let $X := \overline{\psi^{-1}(Z)} \subseteq U $, and let $\widetilde X$ be the preimage of $X$ in $U^\lambda_W$.
As $Z \cap V^\lambda$ is dense in $Z$, thus $Q$ is the kernel of the composition
\[ \xymatrix{
\Sa(\mf g) \ar[r]^{\psi^*} & \kk[U] \ar[r] & \kk[X] \ar@{^{(}->}[r] & \kk[\widetilde X]
}\]
and by Proposition~\ref{Prpoi}, $Q$ is Poisson if and only if $\mf g_{\psi(x)} \subseteq \psi_x(T_xX)$ for all $x \in X$. 

Denote by $A$ the image of $\Sa(\mf g)$ in $\CC[X]$. 
The map $U\to\mf g^*$ is injective and therefore the map $X\to\MSpec A$ is also injective. Let $f\in A$ be the element given by Proposition~\ref{prop:iso}. 
Then the open subsets $\MSpec(A[f^{-1}])$ and $\MSpec(\CC[X][f^{-1}])$ of $\MSpec(A)$ and $X$, respectively, are isomorphic to each other. 
This implies that $\psi_x$ identifies the tangent spaces $T_xX$ and $T_{\psi(x)}Z$ for every $x \in \MSpec(\CC[X][f^{-1}])$. 

Moreover, $\widetilde{X}\to X$ is finite and so by generic smoothness  there are dense open affine subsets $X' \subseteq X$, $\widetilde{X'} \subseteq \widetilde{X}$ so that $\widetilde{X'} \to X'$ is \'etale and so identifies tangent spaces.

Let $X'' = X' \cap \MSpec(\kk[X][f^{-1}])$ and let $\widetilde{X''} \subseteq \widetilde X$ be the preimage of $X''$, which is also affine.
For every $x\in X''$ and every preimage $\widetilde x \in \widetilde{X''}$ of $x$, the map $\widetilde X\to\mf g^*$ gives an isomorphism of the tangent spaces $T_{\widetilde x} \widetilde X$ and $T_{\psi(x)}Z$.

Let $\mf h$ be the Lie algebra of $\widehat{\DLoc}^\lambda$.  
The proofs of Lemma~\ref{Llact} and Lemma~\ref{Lwcrt} show that $\psi_x( \mf h_x) =  \mf g_{\psi(x)}$ for any $x\in U$. 
By Proposition~\ref{Prpoi} applied to the morphism $X'' \to \mf g^*$ and the above discussion, the ideal $Q$ is Poisson if and only if 
$\mf h_x \subseteq T_x\widetilde X$ for all $x\in \widetilde{X''}$.
To finish the proof we note that the following conditions are equivalent:
\begin{itemize}
\item $\mf h_x \subseteq T_x\widetilde{X}$ for all $x\in \widetilde{X''}$ (that is, tangent spaces to $\widehat{\DLoc}^\lambda$-orbits of points of $\widetilde{X''}$ are tangent to $\widetilde{X}$);
\item $\widetilde X=\overline{\widetilde{X''}}$ is $\widehat{\DLoc}^\lambda$-invariant: that is, a union of $\widehat{DLoc}^\lambda$-orbits;
\item the defining ideal of $X = \overline{X''}$ is $\widehat{\DLoc}^\lambda$-invariant.
\end{itemize}
\end{proof}

Let $J^\lambda_{\mf g} := I(V^\lambda_{\mf g})$.
\label{ind:Jlambda}
We immediately obtain:
\begin{corollary}\label{cor:J}
The ideal $J^\lambda_{\mf g}$ is a prime Poisson ideal. 
\end{corollary}
\begin{proof}
That $J^\lambda_{\mf g}$ is Poisson is an immediate application of Proposition~\ref{prop:poisson}. To see that it is prime, note that $U^\lambda_{\mf g}/\mf S_\lambda$ is irreducible and that $J^\lambda_{\mf g}$ is the kernel of $(\psi^\lambda_{\mf g})^*:  \Sa(\mf g) \to \kk[U^\lambda_{\mf g}/\mf S_\lambda]$.
\end{proof}

\begin{remark}\label{rem:opensubset}
Recall that $V^\lambda_{\mf g} \subset \mf g^*$ is the image of the affine variety $U^\lambda_{\mf g}/\mf S_\lambda$.  
Thus by Chevalley's Theorem $V^\lambda_{\mf g} $  contains a nonempty open subset of $\overline{V^\lambda_{\mf g}}$.
Let $U$ be the union of all open subsets of $\overline{V^\lambda_{\mf g}}$ contained in $V^\lambda_{\mf g}$.  It is surely true that $U$ is Poisson, in the sense that the defining ideal of $\overline{V^\lambda_{\mf g}} \setminus U$ is a Poisson ideal. However, we cannot prove this at this point. 
\end{remark}

We next show how to parameterise Poisson primes of $\Sa(\mf g)$.

\begin{theorem}\label{thm:paramprimes}
Let $\mf g = W$ or $\mf g = W_{\geq -1}$.
\begin{itemize}\item[(a)]
Let $Q$ be a nonzero proper prime Poisson ideal of $\Sa(\mf g)$.  
There is a unique $\lambda:= \lambda(Q)$ so that $V(Q) \cap V^\lambda_{\mf g}$ contains a nonempty open subset of $V(Q)$.  
Further, $Y(Q) := \overline{\pi^\lambda_{\mf g}}(\psi^\lambda_{\mf g})^{-1}(V(Q))$ is a (nonempty) irreducible closed subvariety of $\AA(\lambda)$, and
\[ \GK \Sa(\mf g)/Q = \dim Y(Q) + D(\lambda(Q)).\]
\item[(b)] The function $Q \mapsto (\lambda(Q), Y(Q))$ defines a bijection  
\[
\Psi_{\mf g}:  
\PSpec \Sa(\mf g) \to \{ (\lambda, Y) \ | \ Y \mbox{ is an irreducible subvariety of } Var_{\lambda(Q)} \} .
\]
\item[(c)]  If $Q, Q'$ are Poisson primes of $\Sa(\mf g)$ with $\lambda(Q) = \lambda(Q')$, then $Q\subseteq Q'$ iff $Y(Q) \supseteq Y(Q')$.  
\end{itemize}
\end{theorem}
We call the partition $\lambda(Q)$ defined in Theorem~\ref{thm:paramprimes} the {\em generic order partition of $Q$},  in other words the partition of a generic orbit in the associated variety. 
It is easy to check that the construction of Theorem~\ref{thm:paramprimes}(a) guarantees that a Poisson prime $Q$ is the kernel of the map
\begin{equation}\label{Eqy}\Sa(\mf g)\to \CC[\overline{(\pi_\mf g^\lambda)}^{-1}(Y(Q))],\end{equation}
where $\lambda = \lambda(Q)$.
This provides the explicit inverse to the map $\Psi_\mf g$ of Theorem~\ref{thm:paramprimes}(b). 

\begin{proof}
Let $Q$ be a nonzero proper prime Poisson ideal of $\Sa(\mf g)$ and let $Z:= V(Q)$.  As $Q$ is prime, there can be at most one $\lambda$ so that $Z \cap V^\lambda_{\mf g}$ contains a nonempty open subset of $Z$; it remains to establish existence of some such $\lambda$.

By Theorem~\ref{Tgkf} $d:=\GK(\Sa(\mf g)/Q) < \infty$.
By \eqref{dimlambda} there are thus  only finitely many partitions $\lambda$ such that $\dim \OO(\chi) \leq d$ for any $\chi$ with order partition $ \lambda$, and so for some $\lambda^1, \dots, \lambda^n$ we have
\[ Z \subseteq V^{\lambda^1}_{\mf g}\cup \dots \cup V^{\lambda^n}_{\mf g}.\]

Fix $i$ and let $\lambda = \lambda^i$.
Let $Q'$ be a minimal prime of $Q + J^\lambda_{\mf g}$, which exists by Proposition~\ref{prop9}, and let $Z'= V(Q')$, which is an irreducible component of $Z \cap \overline{V^\lambda_{\mf g}}$.  Let $X' = (\psi^\lambda_{\mf g})^{-1}(Z')$.
By Chevalley's Theorem 
$ V^{\lambda}_{\mf g} \cap Z$ contains a dense subset $U_i$ which is locally closed in $Z$. 
As $\overline{ \bigcup U_i } = Z$, by primeness of $Q$ some $U_j$ has $\overline{U_j} =Z$: in other words, $U_j$ is open in its closure $Z$.  Thus $\lambda(Q) = \lambda_j$.

 Let $\lambda:= \lambda(Q) = \lambda_j$, and let $X := (\psi^{\lambda}_{\mf g})^{-1}(Z)$.  
 This is an irreducible closed subvariety of $U^\lambda_{\mf g}/\mf S_\lambda$, and by Proposition~\ref{prop:poisson} it is a union of $\widehat{\DLoc}^\lambda$-orbits.  It therefore follows that $Y(Q):= \overline{\pi^\lambda_{\mf g}}(X)$ is closed in $\AA(\lambda)$.  This proves all but the last statement of (a). 
 
 We now prove (b). 
 Fix a partition $\lambda$ and let $k := \len(\lambda^{\rm e})$. 
 Let $\mf S_\lambda$ acts as in \eqref{muaction}.  
 Let $Y$ be an irreducible
 closed subvariety of $\AA(\lambda)$.  
 Let $X:= (\overline{\pi^\lambda_{\mf g}})^{-1}(Y)$.  
 As the fibres of $\overline{\pi^\lambda_{\mf g}}$ are $\widehat{\DLoc}^\lambda$-orbit quotients and are therefore irreducible (note that $\widehat{\DLoc}^\lambda$ is a connected algebraic group), $X$ is an irreducible $\widehat{\DLoc}^\lambda$-invariant closed subvariety of $U^\lambda_{\mf g}/\mf S_\lambda$ with $Y = \pi^\lambda_{\mf g}(X)$.  
 Let $Q$ be the kernel of $\Sa(\mf g) \to \kk[X]$; clearly $Q$ is prime, and by Proposition~\ref{prop:poisson} $Q$ is a Poisson ideal of $\Sa(\mf g)$.  
Let $Z:= V(Q)$. 
 By applying Chevalley's Theorem to the morphism $\psi^\lambda_{\mf g}: X \to Z$ we see that $\psi^\lambda_{\mf g}(X)$ contains a nonempty open subset of $Z$.  
 This shows in particular that $\lambda = \lambda(Q)$, and it is clear that $Y = Y(Q)$.  
 
 Conversely, given $Q$ apply the procedure in the previous paragraph to $\lambda(Q)$ and $Y(Q)$; one recovers $Q$, see also~\eqref{Eqy}.  This shows that $\Psi_{\mf g}$ is bijective, completing the proof of (b).  
 
Given a nonzero proper Poisson prime $Q$ of $\Sa(\mf g)$ let $\lambda := \lambda(Q)$ and $Y := Y(Q)$.  
 Let $X:= (\overline{\pi^\lambda_{\mf g}})^{-1}(Y)$.  
 We have seen that the fibres of $\overline{\pi^\lambda_{\mf g}}$ have dimension $D(\lambda)$ and so $\dim X = \dim Y + D(\lambda)$.
 It is shown in the proof of Proposition~\ref{prop:poisson} that $\Sa(\mf g)$ and $\kk[X]$ are birational, so $\dim X = \GK \Sa(\mf g)/Q$. 
 This completes (a).
 
 For part (c), let $\lambda = \lambda(Q)=\lambda(Q')$, and let $\psi = \psi^\lambda_{\mf g}$ and $\pi = \overline{\pi^\lambda_{\mf g}}$.  Now let $Z = V(Q)$, $X = \psi^{-1}(Z) = \psi^{-1}(Z \cap V^\lambda_{\mf g})$, and $Y = Y(Q) = \pi(X)$.
 Likewise define $Z'=V(Q')$, $X' = \psi^{-1}(Z')$, and $Y' = Y(Q')$.
 
 By the Nullstellensatz, $$Q \subseteq Q' \iff Z \supseteq Z'.$$  
 Thus if $Q \subseteq Q'$ we have $Z \cap V^\lambda_{\mf g} \supseteq Z' \cap V^\lambda_{\mf g}$ and we see immediately that $Y\supseteq Y'$.
 Conversely, if $Y \supseteq Y'$ then $\psi(X) = Z \cap V^\lambda_{\mf g} \supseteq \psi(X') = Z' \cap V^\lambda_{\mf g}$.
The proof of (b) showed that $Z = \overline{Z \cap V^\lambda_{\mf g}}$, which clearly contains $\overline{Z' \cap V^\lambda_{\mf g}} = Z'$.
\end{proof}

It follows from  Theorem~\ref{thm:paramprimes} that $\PSpec(W)$ and $\PSpec(W_{\geq -1})$ are partitioned into countably many affine strata, corresponding to partitions.
Given a partition $\lambda$, the corresponding stratum consists of prime Poisson ideals $Q$ with $\lambda(Q)=\lambda$, and is homeomorphic to the affine variety $\AA(\lambda)$, an open subset of
a finite-dimensional affine space. 
However, we do not know how to tell in terms of the parameterisation $Q \leftrightarrow (\lambda, Y)$ when $Q \subseteq Q'$ for arbitrary $Q, Q'$.  In fact, we cannot answer this question completely even when $Q, Q'$ are Poisson primitive; see Corollary~\ref{C:sumoflfs} and Example~\ref{ex:derivative} for some discussion of the complexities.

\begin{corollary}[{cf. \cite[Lemma~2.9]{PS}}]
Let $\mf g = W$ or $\mf g = W_{\geq -1}$ and let $Q$ be a prime Poisson ideal of $\Sa(\mf g)$. 
Then there exists $f\in\Sa(\mf g)\setminus Q$ such that $(\Sa(\mf g)/Q)[f^{-1}]$ is a finitely generated Poisson algebra.\end{corollary}
\begin{proof} It is clear that $(\Sa(\mf g)/Q)[f^{-1}]$ is a Poisson algebra for every such $f$ and thus we need to choose $f$ in such a way that $(\Sa(\mf g)/Q)[f^{-1}]$ is finitely generated. 
  Set $A=\Sa(\mf g)/Q$,  $\lambda= \lambda(Q)$, and $B=\CC[\overline{(\pi_\mf g^{\lambda})}^{-1}(Y(Q))]$; note that $B$ is finitely generated as it is easy to check that $\overline{(\pi_\mf g^\lambda)}^{-1}(Y(Q))$ is affine. 
By \eqref{Eqy} we have an injective map $A\to B$ and the induced map $\MSpec B\to\MSpec A$ is injective as well because it is a restriction of an injective map $U^\lambda_\mf g/\mf S_\lambda \to\mf g^*$. 
This together with Proposition~\ref{prop:iso} guarantees the existence of 
$f$.\end{proof}

\begin{remark}\label{rem:GK3two} 
Let $Q$ be a prime Poisson ideal of $\Sa(W)$ of co-GK 3.  
By Theorem~\ref{thm:paramprimes} we have
\[ 3 = \dim Y(Q) + D(\lambda(Q)),\]
and, recalling that $D(\lambda(Q))$ is even by definition, we must have $\dim Y(Q) = 1$ and $D(\lambda(Q)) = 2$.  
(The case $\dim Y(Q) = 3$, $D(\lambda(Q)) = 0$ cannot occur.)
Thus $\lambda(Q) = (2)$ and $\AA(\lambda) = \AA^1 \setminus \{0\} = Y(Q)$.

Now, if $\chi \in W^*$  then $\lambda(\chi) = (2)$ if and only if $\chi = \chi_{x; \alpha, \gamma}$ for some $x, \gamma \in \kk^*$, $\alpha \in \kk$,
and so
\[ Q= \bigcap_{\substack{x, \gamma \in \kk^* \\ \alpha \in \kk}} \CoreP (\chi_{x;\alpha, \gamma}) = \bigcap_\gamma \ker p_\gamma,\]
using Lemma~\ref{lem:Jgamma}.
But from the definition of $p_\gamma$ in \eqref{pgamma} we see that $\bigcap \ker p_\gamma$ is precisely the kernel of the map $\Phi$ from Remark~\ref{rem:GK3}.  
Thus $\ker \Phi$ is the only prime Poisson ideal of $\Sa(W)$ of co-GK 3.

We thank the anonymous referee for drawing our attention to this fact.
\end{remark}

\begin{remark}\label{rem:J}
Fix a partition $\lambda$.  Taking $Y = \AA(\lambda)$ in Theorem~\ref{thm:paramprimes} we see that the pair $(\lambda, Y)$ corresponds to the ideal $J^\lambda_{\mf g}$.  This is another way to see that $J^\lambda_{\mf g}$ is prime.
\end{remark}

\begin{remark}\label{rem:paramprimesVir}
By Corollary~\ref{cor:primeVir} any non-centrally generated prime Poisson ideal of $\Sa(\Vir)$ strictly contains $(z)$ and thus corresponds to a nontrivial prime Poisson ideal of $\Sa(W)$ as given in Theorem~\ref{thm:paramprimes}.  
On the other hand, centrally generated prime Poisson ideals of course correspond to prime ideals of $\kk[z]$.
Thus $\PSpec \Sa(\Vir)$ is also partitioned into countably many strata, each homeomorphic either to a finite-dimensional affine space or to some $\AA^k \times (\AA^1\setminus \{0\})$. 
\end{remark}

We now prove that restriction from $\Sa(W)$ to $\Sa(W_{\geq-1})$ induces a bijection on arbitrary prime Poisson  ideals, not just Poisson primitive ideals.
We conjecture that this bijection is in fact a homeomorphism between ${\rm PSpec_{prim}}(\Sa(W))$ and ${\rm PSpec_{prim}}(\Sa(W_{\geq -1}))$, but at this point we cannot prove it.
 
\begin{corollary}\label{Csameprimespec} 
Restriction induces a bijection between prime Poisson $\Sa(W)$  and prime Poisson ideals of $\Sa(W_{\geq -1})$, and  a bijection between  irreducible closed subsets of  ${\rm PSpec_{prim}}\Sa(W_{\geq-1})$ and of ${\rm PSpec_{prim}}\Sa(W)$.
\end{corollary}
\begin{proof}
We must show that  restriction gives a bijection 
\[  \operatorname{res}:  \{ \mbox{ prime Poisson ideals of $\Sa(W)$ }\} \to \{ \mbox{ prime Poisson ideals of $\Sa(W_{\geq-1})$} \}.\]

Certainly, if $J$ is a prime Poisson ideal of $\Sa(W)$ then $J$ is closed under all $\{v, -\}$ for $v \in W_{\geq -1}$ and so $J \cap \Sa(W_{\geq -1})$ is a prime Poisson ideal.
Thus $\operatorname{res}$ is well-defined.

Let $P \in \PSpec \Sa(W)$.  Let $(\lambda, Y) := \Psi_W(P)$ and let $Q := \Psi_{W_{\geq -1}}^{-1}(\lambda, Y)$.
We claim that $Q = \operatorname{res}(P)$.

For each $y \in Y$ choose a representative $\chi_y \in W^*$ lying in the pseudo-orbit corresponding via $\overline{\pi^\lambda_W}$ to $y$.
Then Theorem~\ref{thm:paramprimes} gives that 
\[ P = \Psi_W^{-1}(\lambda, Y) = \bigcap\{ \mf m_\chi \ | \ \chi \in W^*, \lambda(\chi) = \lambda,  \overline{\pi^\lambda_W}(\chi) \in Y\}.\]
By Lemma~\ref{lem:basics}, this is $\bigcap_{y\in Y} \CoreP(\chi_y)$.
Likewise, $Q$ is equal to 
\[ \Psi_{W_{\geq-1}}^{-1}(\lambda, Y) = \bigcap\{ \mf m_\chi \ | \ \chi \in W_{\geq -1}^*, \lambda(\chi) = \lambda,  \overline{\pi^\lambda_{W_{\geq -1}}}(\chi) \in Y\} = \bigcap_{y \in Y} \CoreP(\chi_y|_{W_{\geq-1}}).\]
By Proposition~\ref{Plfpoi}(a), this is 
\[ \bigcap_{y \in Y} \CoreP(\chi_y) \cap \Sa(W_{\geq -1}) = P \cap \Sa(W_{\geq -1}) = \operatorname{res}(P).\]

As $\operatorname{res}$ acts on nonzero Poisson prime ideals as the composition of the bijections $\Psi_W$ and $\Psi_{W_{\geq-1}}^{-1}$, it is a bijection.
%
%
%
\end{proof}

\subsection{The Poisson Dixmier-Moeglin equivalence}\label{SSPDME}
We next consider which Poisson prime ideals satisfy the {\em Poisson Dixmier-Moeglin equivalence}.  This is the Poisson version of the equivalent conditions for primitive ideals in enveloping algebras of finite-dimensional Lie algebras, which are due to Dixmier and Moeglin.  We describe the conditions here.
\label{ind:PDME2}

Let $Q$ be a Poisson prime ideal in a Poisson algebra $A$, which we assume to be a domain.  
Then $Q$ is {\em Poisson locally closed} if it is locally closed in the Zariski topology on $\PSpec(A)$.
We say $Q$ is {\em Poisson rational} if the Poisson centre of the field of fractions of $A/Q$ is algebraic over $\kk$.  
We say the {\em Poisson Dixmier-Moeglin equivalence (PDME) holds for $A$} if for any prime Poisson ideal $Q$ of $A$, the conditions
\[
\begin{array}{c}
\mbox{$Q$ is Poisson locally closed,}\\
\mbox{$Q$ is Poisson primitive,}\\
\mbox{$Q$ is Poisson rational}
\end{array}
\]
are equivalent. 
If $\mf g$ is a finite-dimensional Lie algebra, then the PDME holds for $\Sa(\mf g)$ \cite[Theorem~2]{LLS}.

On the other hand, the next result shows that the PDME fails for $\Sa(\Vir)$, $\Sa(W)$, and $\Sa(W_{\geq -1})$.  However, it {\em almost} holds:  for all these algebras, there is only one Poisson prime for which the PDME fails.
\begin{theorem}\label{thm:PDME}
The Poisson Dixmier-Moeglin equivalence holds for all prime Poisson  ideals of $\Sa(\Vir)$ except for $(z)$:  that is, if $Q \neq (z)$ is a prime Poisson  ideal of $\Sa(\Vir)$ then $Q$ is locally closed in the Poisson spectrum if and only if $Q$ is Poisson primitive, if and only if $Q$ is Poisson rational.  However, $(z)$ is Poisson primitive and Poisson rational but not Poisson locally closed.
Thus the PDME for $\Sa(W)$ fails but holds for all prime Poisson ideals  except for $(0)$.
Likewise, the PDME for $\Sa(W_{\geq-1})$ fails but holds for all prime Poisson ideals  except for $(0)$.
\end{theorem}

\begin{remark}\label{rem:pathology}
By Corollary~\ref{cor:Psimple}, the Poisson spectrum of $\Sa(\Vir)[z^{-1}]$, which is in natural bijection with the set of prime Poisson ideals of $\Sa(\Vir)$ that do not contain $z$, consists of $(0)$ and the ideals $(z-\lambda)$ for $\lambda \in \kk^*$:  in other words, $\PSpec(\Sa(\Vir)[z^{-1}]) \cong \AA^1 \ssm \{0\}$.
Therefore,  the most interesting structure of $\Sa(\Vir)$ is concentrated above the ideal $(z)$, which Theorem~\ref{thm:PDME} shows to be pathological in some sense. 
\end{remark}

To prove Theorem~\ref{thm:PDME}, we will need notation for the Zariski topology on $\PSpec \Sa(\mf g)$.  If $A$ is a Poisson algebra and  $N$ is a Poisson ideal of $A$, we denote the corresponding closed subset of $\PSpec A$ 
by
\[ \VP(N) := \{ P \in \PSpec A \ |\  N \subseteq P\}.\]
\label{ind:VP}
We caution the reader that although $V(N) \subseteq \MSpec A$ consists of maximal ideals,  $\VP(N) \subseteq \PSpec A$ consists of prime (Poisson) ideals.  

Before proving Theorem~\ref{thm:PDME}, we note that the standard equivalent condition for a Poisson prime to be locally closed also holds in this infinite-dimensional setting. 

\begin{lemma}\label{lem:locally closed}
Let $A$ be a (possibly non-noetherian) Poisson algebra, and let $Q$ be a prime Poisson ideal of $A$.  Then $Q$ is locally closed in $\PSpec A$ if and only if there is some $f \in A \setminus Q$ so that $(A/Q)[f^{-1}]$ is Poisson simple.
\qed
\end{lemma}
We leave the proof to the reader.

%
\begin{proof}[Proof of Theorem~\ref{thm:PDME}]
Let $\mf g$ be any countable-dimensional Lie algebra.  
By \cite[Theorem~6.3]{LSS}, for prime Poisson ideals of $\Sa(\mf g)$, Poisson locally closed implies Poisson primitive and Poisson primitive is equivalent to Poisson rational.  Thus to prove the theorem for $\Sa(\Vir)$ and $\Sa(W)$, it suffices to prove:
\begin{itemize}\item[(a)]Let $Q$ be a prime Poisson  ideal of $\Sa(\Vir)$ with $Q \neq (z)$.   If $Q$ is Poisson primitive, then  $Q$ is  Poisson locally closed.
\item[(b)] The ideal $(z)$ is Poisson primitive but not Poisson locally closed.
\end{itemize}
We first prove (b).   
Let $\nu \in \Vir^*$ be any non-local function with $\nu(z) = 0$ (for example, we can take $\nu$ to induce the non-local function $\varkappa$ on $W$ given in Remark~\ref{rem:dag}).  By Theorem~\ref{Tlocvir}, $\CoreP(\nu) = (z)$ and so $(z)$ is Poisson primitive.

To prove that $(z)$ is not Poisson locally closed, by Lemma~\ref{lem:locally closed} it is enough to prove that $\Sa(W)[f^{-1}]$ is not Poisson simple for any $f\in\Sa(W) \setminus \{0\}$.
Suppose that $\Sa(W)[f^{-1}]$ is Poisson simple. Then $f$ is contained in all proper Poisson ideals of $\Sa(W)$.  Let $\chi \in W^*$ be a local function.  The Poisson core $\CoreP(\chi)$ of $\chi$ is nontrivial, so $f \in \CoreP(\chi) \subseteq \mf m_\chi$.  In other words, for any local function $\chi \in W^*$, we have $\ev_\chi(f) = 0$, which is ridiculous. 

We now prove (a).   Let $\chi \in \Vir^*$ with $\CoreP(\chi) \neq (z)$.  We show that $\CoreP(\chi)$ is Poisson locally closed.

The ideals $(z-\zeta)$ with $\zeta \neq 0$ are maximal in $\PSpec \Sa(\Vir)$ by Corollary~\ref{cor:Psimple} and are thus closed points of $\PSpec \Sa(\Vir)$.  
So we may assume that 
 $\CoreP(\chi)  \neq (z-\chi(z))$.  
By Theorem~\ref{Tlocvir} $\chi$ is therefore local and $d:= \dim \Vir/\Vir^\chi < \infty$.  

Let $I(d-1)$ be the Poisson ideal of $\Sa(\Vir)$ defined in Lemma~\ref{lem:detideal}.  
We claim that
\[ \VP(\CoreP(\chi)) \setminus \VP(I(d-1))  = \{ \CoreP(\chi)\}\]
so $\CoreP(\chi)$ is Poisson locally closed.

First, by \eqref{detideal} $I(d-1) \not\subseteq \mf m_\chi$ so $\CoreP(\chi) \not \in \VP(I(d-1))$.
By Proposition~\ref{Plfpoi} $d = \GK \Sa(\Vir)/ \CoreP(\chi)$.
Thus if $\nu \in \overline{\OO(\chi)} \setminus \OO(\chi)$, we have
\[ d-1 \geq \GK \Sa(\Vir)/\CoreP(\nu) = \dim \Vir \cdot \nu,\]
where we have used Proposition~\ref{Plfpoi} again for the last equality.
Thus $I(d-1) \subseteq \mf m_\nu$ by \eqref{detideal}, so $I(d-1) \subseteq \CoreP(\nu)$.
It follows that if $P$ is a Poisson prime ideal of $\Sa(\Vir)$ with $P \supsetneqq \CoreP(\chi)$ then, writing
\[ P =\bigcap \{ \CoreP(\nu) \ | \ \mf m_\nu \supseteq P\},\]
we have $P \supseteq I(d-1)$, establishing the claim.  

The proof for $\Sa(W_{\geq-1})$ is almost identical. 
\end{proof}

We remark that we have used relatively few pieces of structure theory of $\Vir$, $W$, and $W_{\geq-1}$ in the proof of Theorem~\ref{thm:PDME} (the inputs are essentially that $\dim \OO(\chi)=\dim\mf g\cdot\chi$ and that $\dim\mf g\cdot\chi<\infty$ for all $\chi$ in which we are interested). 
Thus similar results may hold for a wider class of Lie algebras.

\subsection{Radical Poisson ideals contain Poisson primitive ideals}\label{SSPradprim}

Let $\mf g = W$ or $W_{\geq-1}$.  
%
From the bijection in Theorem~\ref{thm:paramprimes} it is tempting to think of Poisson primitive ideals of $\Sa(\mf g)$ as analagous to closed points:  by that result, a prime Poisson ideal $Q$ is Poisson primitive if and only if $Y(Q)$ is a single point.  
Note, however, that Theorem~\ref{thm:PDME} also tells how to distinguish Poisson primitive ideals in $\PSpec \Sa(W)$, and we saw there that in general they are only locally closed.
The next result shows how far most of the Poisson primitive ideals are from being maximal in $\PSpec \Sa(\mf g)$; see also Corollary~\ref{cor:augmentation}.

\begin{proposition}\label{Pradprim}
Let  $\mf g=W$ or  $W_{\ge-1}$ or $W_{\ge1}$ and let $I$ be a nonzero radical Poisson ideal of $\Sa(\mf g)$. 
Then there is some local function $\nu\in\mf g^*$ with $\CoreP(\nu)\subseteq I$.  
\end{proposition}

If $I$ is prime and $\mf g = W$ or $W_{\geq -1}$ this may be deduced from the bijection $\Psi$ of Subsection~\ref{SSpPp}.  
We provide a direct proof, however.

\begin{proof} 
Thanks to Theorem~\ref{Tgkf} we have that $\GK(\Sa(\mf g)/I)<\infty$; pick $d\in\ZZ_{\ge1}$ with $2d>\GK(\Sa(\mf g)/I)$. 
By Lemma~\ref{lem:basics} $I$ is an intersection of a family of primitive ideals $\CoreP(\mu)$, and
$$\GK(\Sa(\mf g)/\CoreP(\mu))\le \GK(\Sa(\mf g)/I)$$
for each $\mu$ with $\mf m_\mu \supseteq I$. 
We will show that there exists $\nu = \nu_{d} \in\mf g^*$
such that if $\mu \in \mf g^*$ with $$\GK(\Sa(\mf g)/\CoreP(\mu))<2d,$$ then $\CoreP(\nu)\subseteq \CoreP(\mu)$. 

Pick $\mu\in\mf g^*$ with $\dim  \mf g\cdot\mu=\GK(\Sa(\mf g)/\CoreP(\mu))<2d$, see Proposition~\ref{Plfpoi}(b). 
Then $\mu$ is local and thus is a sum of several (say $\ell$) nonzero one-point 
local functions $\mu_i\in\mf g^*$ with distinct supports $x_i$. 
Recall that $\dim \mf g\cdot\mu=\sum_{i=1}^\ell\dim\mf g\cdot \mu_i$ by Lemma~\ref{Lwcrt} and therefore $\dim \mf g\cdot \mu_i<2d$ together with $\ell< 2d$. 
Pick $n\in\NN$ such that $n\ge 2d$ and $\mu_i\in \Loc_{x_i}^{\le n}$ for all $i$. 
Thanks to Lemmata~\ref{Lwgr},~\ref{Llact} we have $\dim \DLoc_{x_i}^{\le n}\mu_i<2d$ for all $i$.

Let $\widetilde\chi_i:=\chi_{\mbox{$x_i$}; \small{\underbrace{0, 0, \ldots,0,}_{2d\text{~times}}} \mbox{1}}$. 
Thanks to Theorem~\ref{Texpw} we have $\dim \DLoc_{x_i}^{\le n}\widetilde\chi_i=2d+1>\dim \DLoc_{x_i}^{\le n}\mu_i$. 
This together with Corollary~\ref{Cdim} implies that $\CoreP(\widetilde \chi_i)\subseteq \CoreP(\chi_i)$.

By Corollary~\ref{C:sumoflfs}, $\CoreP(\sum_{i=1}^{\ell}\widetilde{\chi_i})\subseteq \CoreP(\mu)$. 

Set $\widetilde{\chi_i}':=\chi_{\mbox{\it{i}}; \underbrace{\scriptsize{0, 0, \ldots,0,}}_{2d\text{~times}} \mbox{1}}$.
By Theorem~\ref{Tlfpoi},  $\CoreP(\sum_{i=1}^{\ell}\widetilde {\chi_i})=\CoreP(\sum_{i=1}^\ell\widetilde{\chi_i}')$ and hence, applying Corollary~\ref{C:sumoflfs}, $\nu=\sum_{i=1}^{2d}\widetilde{\chi_i}'$ satisfies the desired properties.

By Lemma~\ref{lem:basics}, $\CoreP(\nu) \subseteq I$.
\end{proof}

The statement of the above corollary can be rephrased as follows: every proper algebraic collection of pseudo-orbits is (strictly) contained in the closure of a single finite-dimensional pseudo-orbit. 
If $\dim \mf g < \infty$, this means that $\mf g^*$ contains a dense coadjoint orbit; such Lie algebras are called Frobenius Lie algebras and there are very few of them.
For infinite-dimensional Lie algebras this statement seems quite counterintuitive.
We expect that only a few Lie algebras satisfy it.

The next corollary answers Question 6.8 of \cite{LSS}.

\begin{corollary}\label{Cinfo}
Let $\mf g = W $ or $W_{\geq -1}$ or $W_{\geq 1}$.  
Then $\Sa(\mf g)$ has no  nonzero prime Poisson ideals  of finite height.
\end{corollary}
\begin{proof}
For $d \in \NN$ let $\nu_d \in \mf g^*$ be the local function defined in the proof of Proposition~\ref{Pradprim}.  
The crucial property of $\nu_d$, established in that proof, is that
\beq\label{crucial}
\mbox{if $\GK(\Sa(\mf g)/\CoreP(\mu))<2d$ then $\CoreP(\nu_d)\subseteq \CoreP(\mu)$.}
\eeq
Let $Q$ be a nonzero Poisson prime ideal of $\Sa(\mf g)$ and let $d_1 := \GK \Sa(\mf g)/Q$, which is finite by Theorem~\ref{Tgkf}.  
By the proof of Proposition~\ref{Pradprim}, $\CoreP(\nu_{d_1}) \subseteq Q$.
Then define $d_i$ by induction:  let $d_{i+1} := \GK(\Sa(\mf g)/\CoreP(\nu_{d_i}))+2$.  
By \eqref{crucial}, each $\CoreP(\nu_{d_{i+1}}) \subsetneqq \CoreP(\nu_{d_i})$.
\end{proof}

By Corollary~\ref{cor:Psimple}, if $\zeta \in \kk^\times$ then $(0)$ is the only prime Poisson ideal of $\Sa(\Vir)$ which is contained in $(z-\zeta)$; thus the maximal Poisson ideal $(z-\zeta)$ has height 1 as a prime Poisson ideal.

\section{Subalgebras of finite codimension}\label{Svirfd}

In this section we sharpen earlier results to classify subalgebras of $\Vir$ of small codimension. 
By Proposition~\ref{prop:4.15} we know any such subalgebra contains $z$, so we may reduce to considering the  corresponding subalgebra of $W$; by  Proposition~\ref{prop:4.14} this contains some $W(f)$ with $f \neq 0$. We refine these results and provide more precise statements on subalgebras of $\Vir$ and of $W$ of codimensions 1, 2, 3.

Throughout this section we assume that $\mf k$ is a subalgebra of $W$ of finite codimension.
Let  $f_{\mf k} \in\CC[t]$ 
\label{ind:fk}
be the lowest degree monic polynomial with $\mf k\supset W(f)$, which exists  by Proposition~\ref{prop:4.14}.
 Proposition~\ref{prop:4.14} in fact gives us that 
\beq \label{boff} W(f_{\mf k}) \subseteq \mf k \subseteq W({\rm rad}(f_{\mf k})), \eeq
 where recall that ${\rm rad}(f)=\prod \{ (t-x) | f(x) = 0\}$.
Thus
\beq \label{codima}
\codim_{W} \mf k \geq \deg{\rm rad}(f_{\mf k})=|\{x\in\CC^\times\mid  f_{\mf k}(x) =0\}|.\eeq
By the Euclidean algorithm,  $W(h)\subseteq\mf k$ if and only if $f_{\mf k} \mid h$. 

We immediately obtain a classification of subalgebras of codimension 1, a more conceptual proof of a result originally due to  Ondrus and Wiesner~\cite[Proposition~2.3]{OW}.

\begin{corollary}
\label{cor:codim1}
Let $\mf h$ be a subalgebra of $\Vir$ of codimension 1.  
Then there is $x \in \CC^\times$ so that $$\mf h = (t-x) \CC[t, t^{-1}] \del \oplus \CC z.$$
\end{corollary}
\begin{proof}
As remarked above, by Proposition~\ref{prop:4.15} $z \in \mf h$ so it suffices to prove that $\mf k := \mf h /(z)$, which is a subalgebra of $W$ of codimension 1, is equal to some $W(t-x)$. 
By \eqref{codima} $f_{\mf k}$ must be equal to some $(t-x)^a$ and by \eqref{boff} $\mf k \subseteq W(t-x)$.  
But $W(t-x)$ already has codimension 1.
\end{proof}

The problem of classifying general cofinite-dimensional subalgebras of $\Vir$ is equivalent, by Proposition~\ref{prop:4.15} and \eqref{boff}, to the problem  of classifying subalgebras of an arbitrary Lie algebra of the form $W({\rm rad}(f))/W(f)$;
note that $W(f)$ is a Lie ideal of $W({\rm rad}(f))$ and that $W({\rm rad}(f))/W(f)$ is finite-dimensional and nilpotent.
In subsections~\ref{SScd2} and \ref{SScd3} we give a complete classification of  Lie subalgebras $\mf k\subseteq W$ of codimension 2 and 3, illustrating the complexity of the problem.

\subsection{Notation and concepts}
We begin by establishing some needed notation.  
We call the full list of subalgebras of $W$ of codimension 1 the  {\em  spectrum} of $W$. 
Corollary~\ref{cor:codim1} implies that every such  subalgebra is of the form $W(t-x)$ for an appropriate $x\in\CC^\times$ and thus we can identify the spectrum of $W$ with $\CC^\times$. 
Pick a subalgebra $\mf k$ of $W$ of finite codimension. 
Denote by $\supp(\mf k)$
\label{ind:suppk}
 the list of $x\in\CC^\times$ so that $\mf k\subseteq W(t-x)$. 
By \eqref{boff} $\supp(\mf k)$ is nonempty and finite.  
The  picture is parallel to that of ideals in a commutative $\kk$-algebra, where subalgebras of codimension 1 (maximal subalgebras) are the natural analogues of ideals of codimension 1 (maximal ideals).

It will be useful below to carry out a more detailed analysis of subalgebras with $|\supp(\mf k)|=1$. 
 We establish notation for various invariants of such $\mf k$.

\begin{notation}\label{not:onepoint}
Let $\mf k$ be a finite codimension subalgebra of $W$  with $|\supp(\mf k)|=1$. 
We define the following invariants of $\mf k$.

Let $d:= d(\mf k) := \codim_W\mf k$. 
As $\supp(\mf k)=\{x\}$ for some $x\in\CC^\times$ 
we have $f_{\mf k} = (t-x)^a$ for some $a \geq d$.   
Let $a(\mf k) := a$.
If $a= d$ then $\mf k = W((t-x)^{d})$.
By definition of $f_{\mf k}$, if $a \neq d$ then 
\beq \label{a} (t-x)^{a-1}\del \not \in \mf k. \eeq

Assume that $a \neq d$.
Set $\tilde t:=t-x$.
Let $f_1, f_2, \dots, f_{a-d}$ be elements of $\CC[t, t^{-1}]$ so that the images of the $f_i \del$  give a basis for $\mf k / W(f)$. 
We may write each $f_i$ as $\tilde t^{n_i} k_i$ where $k_i(x) = 1$. 
By cancelling leading terms in the Taylor expansion of $f_i$ around $\tilde t=0$ we may assume that
$$1 \leq n_1 < n_2 < \dots < n_{a-d} < a-1,$$ where  we  used \eqref{a} for the last inequality.
Write $\{ 1, \dots, a-1\} \setminus \{ n_1, \dots, n_{a-d}\} = \{ g_1, \dots, g_{d-1}\}$ where $g_1 < g_2< \dots$.
We say that $$\ldeg(\mf k):=\{n_1-1, \ldots, n_{a-d}-1\}$$ are the {\em leading degrees} 
\label{ind:ldeg}
of $\mf k$. We say that $\sdeg(\mf k):=\{g_1-1, \ldots, g_d-1\}$ are {\em gaps} of~$\mf k$.
\label{ind:gaps}

Note that we do not allow $\sdeg(\mf k) = \{0, \dots, a-2\}$ as this would mean that $\ldeg (\mf k) = \emptyset$, contradicting the assumption that $a(\mf k) \neq d(\mf k)$. 
\end{notation}

\begin{lemma}\label{Lagro}
\begin{itemize}
\item[(a)] 
Let $\mf k$ satisfy $|\supp(\mf k) | = 1$ with $a := a(\mf k) > d:= d(\mf k)$.
 Pick $1\le i\ne j\le a-d$. 
Then either $(n_i-1)+(n_j-1)\ge (a-1)$ or $(n_i-1)+(n_j-1)\in \ldeg(\mf k)$.

\item[(b)]  Let $S = \{ n_1, \dots, n_\ell\}$ be a subset of $\{1,\ldots, a-2\}$ satisfying the conclusion of (a). 
Then there exists a subalgebra  $\tilde{\mf k}$ of $W$ with $a(\tilde{\mf k})=a$ and $\ldeg(\tilde{\mf k})=S$.
\end{itemize}
\end{lemma}
\begin{proof}(a). Consider $$\tilde f_{i, j}\del:=[f_i\del, f_j\del]=[\tilde t^{n_i}k_i\del, \tilde t^{n_j}k_j\del]=(j-i)\tilde t^{n_i+n_j-1}k_ik_j+\tilde t^{n_i+n_j}(k_ik_j'-k_i'k_j) \in \mf k.$$ 
It is easy to verify that $k_{i, j}=\tilde{f}_{i, j}/\tilde t^{n_i+n_j-1} \in\CC[t, t^{-1}]$ and $k_{i, j}(x)\ne0$. This implies~(a). 

For (b) it is clear that the space $\tilde{\mf k}$ defined as the span of $W(\tilde t^a)$ and $\tilde t^i\del$ with $i\in S$ is a Lie subalgebra of~$W$. 
\end{proof}

The reason for subtracting 1 in the definition of the leading degrees of $\mf k$ is that 
Lemma~\ref{Lagro} shows that  $\ldeg(\mf k)$ has the structure of a partial semigroup under addition. 

We deduce from Lemma~\ref{Lagro} the following fact.
\begin{lemma}\label{Lgaps2}
Let $\mf k$ satisfy $|\supp(\mf k) | = 1$ with $a := a(\mf k) > d:= d(\mf k)$.
If $g_1=1$ then $g_i\le 2i-1$ for all $i, 1\le i\le d-1$. 
If $g_1\ne 1$ then $g_i\le 2i+1$ for all $i, 1\le i\le d$. \end{lemma}
\begin{proof} First assume that $g_1=1$. 
Next, assume to the contrary that $g_i=2i-1+\delta$ with $\delta>0$. 
Consider the list of pairs
\beq\label{pq} (2, 2i-2+\delta), (3, 2i-3+\delta), \ldots, (i, i+\delta). \eeq

For each pair $(p,q)$ in \eqref{pq} we have $p+q-1=g_i$; further, 
Lemma~\ref{Lagro} implies that either $p-1$ or $q-1$ belongs to $\sdeg(\mf k)$ for each pair $(p, q)$ from \eqref{pq}, it must be in $\{g_2, \dots, g_{i-1}\}$ as $g_1 = 1< p, q$. 
There are  $i-1$ pairs in \eqref{pq} but only $i-2$ gaps  from $g_2$ to $g_{i-1}$. 
This contradiction completes the case $g_1=1$. In the case $g_1>1$ we have to do the same thing with a minor modification: we have to add $g_1$ to the list $\{g_2, \ldots, g_{i-1}\}$.
\end{proof}
\begin{remark}Note that $g_d=a-1\le 2d+1$ and hence one can enumerate all the pairs $(S, a)$ satisfying the conclusion of condition (a) of Lemma~\ref{Lagro} for a given codimension $d$. 
We believe  that such pairs $(S, a)$ are in bijection with 
irreducible components of the moduli space of subalgebras $\mf k$ of codimension $d$ with $|\supp(\mf k)|=1$. \end{remark}

\subsection{Subalgebras of codimension 2}\label{SScd2}
The goal of this subsection is to show that all subalgebras of $W$ of codimension 2 are listed in the following table.
$$
\begin{array}{|c|c|c|c|}
\hline
\multicolumn{4}{|c|}{\mbox{Subalgebras of $W$ of codimension 2}}\\
\hline \rm{Code} &\sdeg(\mf k)&f_{\mf k}&\begin{tabular}{c}Additional\\ generators\end{tabular}\\

\hline W((t-x)(t-y))&-&(t-x)(t-y)&-\\

\hline
W_{x; \alpha}^{2; 1}&1&(t-x)^3&\begin{array}{c}(t-x)\del+\alpha(t-x)^2\del\end{array}\\
\hline
W_{x; \alpha}^{2; 2}&2&(t-x)^4&\begin{array}{c}(t-x)\del+\alpha(t-x)^3\del,\\
(t-x)^2\del\end{array}\\
\hline
\end{array}
$$
\label{ind:codim2}
Here $x, y, \alpha$ are parameters taking values in $\CC$ with $x, y\ne0$. 

The proof consists of the following two statements.
\begin{proposition}\label{Prcodim2a}
\begin{itemize}
\item[(a)] If $\codim_W \mf k=2$ then $|\supp(\mf k)|\in \{1, 2\}$.

\item[(b)] If $|\supp(\mf k)|=2$ then $\mf k=W((t-x)(t-y))$ and $\supp(\mf k)=\{x, y\}$ for some $x \neq y\in\CC^\times$.
\end{itemize}
\end{proposition}

\begin{proposition}\label{Prcodim2b}
\begin{itemize} \item[(a)] If $\codim_W\mf k=2$ and $| \supp(\mf k)| = 1$ (so $\supp(\mf k)=\{x\}$  for some $x \in \kk^\times$) then either $\mf k = W((t-x)^2)$ or $\sdeg(\mf k)$ is $\{1\} $ or $ \{2\}$.

\item[(b)] If $\sdeg(\mf k)=\{1\}$ then $\mf k=W_{x; \alpha}^{2; 1}$ for a unique $\alpha\in\CC$.

\item[(c)] If $\sdeg(\mf k)=\{2\}$ then $\mf k=W_{x; \alpha}^{2; 2}$ for a unique $\alpha\in\CC$.
\end{itemize}
\end{proposition}
\begin{proof}[Proof of Proposition~\ref{Prcodim2a}] Theorem~\ref{prop:4.14} implies $|\supp(\mf k)|\ge1$, and  \eqref{codima} implies that  $|\supp(\mf k)|\le2$.  This proves~(a). 

Assume $\supp(\mf k)=\{x, y\}$ for some distinct $x, y\in\CC^\times$. 
We have $\mf k\subseteq W((t-x)(t-y))$ and 
$$\codim_W W((t-x)(t-y))=\codim_W\mf k=2.$$ 
Hence $\mf k=W((t-x)(t-y))$ and (b) is complete.
\end{proof}
\begin{proof}[Proof of Proposition~\ref{Prcodim2b}]
We may assume that $\mf k \neq W((t-x)^2)$.  
Adopt the terminology of Notation~\ref{not:onepoint}.
We have 
 $|\sdeg(\mf k)|=\codim(\mf k)-1$ and hence $|\sdeg(\mf k)|=\{g_1-1\}$ for a certain positive integer $g_1$.
 Note that $g_1 = a(\mf k) - 1$ by \eqref{a}.
 Thanks to Lemma~\ref{Lagro} we have $g_1\le 3$. 
This proves (a). 

 Note that $g_1 \neq 1$ as by convention $\sdeg(\mf k) $ cannot be equal to $\{0, \dots, a-2\}$.
Assume $g_1=2$, so 
 $a(\mf k)=3$. 
 Recall that $\tilde t=t-x$.
Let $f_1 = \tilde{t}k_1$ be as in Notation~\ref{not:onepoint}; as $a(\mf k) - d(\mf k) = 1$, 
the image of $f_1\del$ in $\mf k / W(\tilde t^3)$ provides a generator of this one-dimensional vector space. 
We can assume that $f_1$ is monic in $\tilde t$ and contains no terms in $\tilde t$ of degree $3$ or more. 
Thus $f_1=\tilde t+\alpha \tilde t^2$ for some $\alpha\in\CC$ and  
 $\mf k=W_{x; \alpha}^{2; 1}$.  Further, $\alpha$ is unique as $\dim \mf k /W(\tilde t^3) = 1$.

Assume $g_1=3$, so  $a(\mf k)=4$. 
Let $f_1 = \tilde t k_1, f_2 = \tilde t^2 k_2 \in\CC[t]=\CC[\tilde t]$ be 
as in Notation~\ref{not:onepoint}, so the images of $f_1\del, f_2\del$ in $\mf k / W((t-x)^3)$ give a basis. 
We can assume that $f_1, f_2$ contain no terms in $\tilde t$ of degree $4$ or more and $f_1$ contains no terms in $\tilde t$ of degree $2$. 
Thus $f_1=\tilde{t}+\alpha \tilde{t}^3, f_2=\tilde{t}^2+\beta \tilde t^3$ for some $\alpha, \beta\in\CC$. 
We have
$$\mf k \ni [f_1\del, f_2\del]=\tilde t^2\del+ 2\beta \tilde t^3\del \mod W(\tilde t^4).$$
Hence $[f_1\del, f_2\del]=f_2\del$ modulo $ W(\tilde t^4)$. 
This implies $\beta=0$ and therefore $\mf k= W_{x; \alpha}^{2; 2}$.  Again $\alpha$ is unique.
\end{proof}

\begin{remark}\label{rem:codim2W-1}
Similar proofs give a classification of subalgebras of $W_{\geq -1}$ of codimension 2:  these are either of the form
$W_{\geq-1}((t-x)(t-y))$ or may be written $(W_{\geq-1})^{2;1}_{x;\alpha}$ or $(W_{\geq-1})^{2;2}_{x;\alpha}$, where these last two are defined similarly to the analogous subalgebras of $W$.  
(Here of course we allow $x,y$ to be $0$.)
These subalgebras are all deformations of $W_{\geq-1}(t^2) = W_{\geq 1}$.

Deformations of $W_{\geq 1}$ are classified in \cite{F1} (see also \cite{FF}).  
These papers show that up to isomorphism there are three such deformations, denoted in \cite{F1} by $L_1^{(1)}$, $L_1^{(2)}$, $L_1^{(3)}$.
It can be shown that
\[ L_1^{(1)} \cong W(t(t-y)), \quad L_1^{(2)} \cong (W_{\geq-1})^{2;1}_{0;\alpha}, \quad L_1^{(3)} \cong (W_{\geq-1})^{2;2}_{0;\alpha}\]
for appropriate $y, \alpha$. 
We thank Lucas Buzaglo for explaining this to us.
\end{remark}

\subsection{Subalgebras of codimension 3} \label{SScd3}
It can also be shown that all subalgebras of $W$ of codimension 3 are listed in the following table.
Because the methods are similar to those in Subsection~\ref{SScd2} we omit the proof. 
$$
\begin{array}{|c|c|c|c|}
\hline
\multicolumn{4}{|c|}{\mbox{Subalgebras of $W$ of codimension 3}}\\
\hline {\rm Code} &\sdeg(\mf k)&f_{\mf k}&\begin{tabular}{c}Additional generators\\ or description\end{tabular}\\

\hline W(f_{\mf k})&-&(t-x)(t-y)(t-z)&-\\

\hline W_{x, y; \alpha, \beta}^{3A}&-&(t-x)^2(t-y)^2&\begin{array}{c}(t-x)(t-y)(\alpha t+\beta)\del,\\ \alpha x+\beta, \alpha y+\beta \ne 0, x\ne y\end{array}\\

\hline
W_{x, y; \alpha}^{3B1}&-&(t-x)^3(t-y)&\begin{array}{c}W^{2; 1}_{x; \alpha}\cap W(t-y),\\x\ne y\end{array}\\

\hline
W_{x, y; \alpha}^{3B2}&-&(t-x)^4(t-y)&\begin{array}{c}W^{2; 2}_{x; \alpha}\cap W(t-y),\\x\ne y\end{array}\\

\hline
W_{x; \alpha}^{3C1}&0, 2&(t-x)^4& 
(t-x)^2\del+\alpha(t-x)^3\del\\

\hline
W_{x; \alpha, \beta}^{3C2}&1, 2&(t-x)^4&
(t-x)\del+\alpha(t-x)^2\del+\beta(t-x)^3\del\\

\hline
W_{x; \alpha, \beta}^{3C3}&1, 3&(t-x)^5&\begin{array}{c} (t-x)\del+\alpha(t-x)^2\del+\beta(t-x)^4\del,\\
(t-x)^3\del-\alpha(t-x)^4\del\end{array}\\
\hline

W_{x; \alpha, \beta}^{3C4}&1, 4&(t-x)^6&\begin{array}{c}(t-x)\del+\alpha(t-x)^2\del+\beta(t-x)^5\del,\\
(t-x)^3\del-\alpha^2(t-x)^5\del,\\
(t-x)^4\del-2\alpha(t-x)^5\del\\\end{array}\\
\hline

W_{x; \alpha, \beta}^{3C5}&2,  3&(t-x)^5&\begin{array}{c}(t-x)\del+\alpha(t-x)^3\del+\beta(t-x)^4\del,\\
(t-x)^2\del+\frac{\alpha}2(t-x)^4\del\end{array}\\
\hline

\end{array}
$$
\label{ind:codim3}

The notation here is as follows:

$\bullet$ $\mf k$  is  given in the final column as either an intersection of two explicitly given subalgebras or is spanned by $W(f_\mf k)$ and a few more explicit generators;

$\bullet$ $\alpha, \beta$ are parameters taking arbitrary values in $\CC$ for all cases but $W^{3A}_{x, y; \alpha, \beta}$ (for this case the required restrictions on $\alpha, \beta$ are given in the table);

$\bullet$ $x, y$ are parameters taking arbitrary values in $\CC^\times$ except when restrictions are given in the table.

It is easy to verify that a subalgebra $\mf k$ of $W$ of codimension 3 belongs to only one type and moreover the parameters defining $\mf k$ are unique if $\mf k\ne W_{x, y; \alpha, \beta}^{3A}$ and unique up to scaling $(\alpha, \beta)\to (\lambda \alpha, \lambda \beta)$ in the case $\mf k = W_{x, y; \alpha, \beta}^{3A}$.

\begin{remark} \label{Rgstr} 
$(1)$ Recall the analogy between subalgebras of $W$ of codimension 1 and maximal ideals.  It is natural to ask whether or not an analogue of the Lasker-Noether primary decomposition theorem holds in this setting. However, this  statement fails as we can easily see that 
the Lie algebra $W^{3A}_{x, y; \alpha, \beta}$ is not an intersection of subalgebras $\mf k_x$ and $\mf k_y$ with $\supp(\mf k_x)=\{x\}$ and $\supp(\mf k_y)=\{y\}$. 

\noindent $(2)$
Recall that every subalgebra $\mf k$ of finite codimension in $W$ lies between $W(f_\mf k)$ and $W({\rm rad}(f_\mf k))$. 
One can  construct a sequence of polynomials $$h_0={\rm rad} f_\mf k, h_1, h_2,\ldots, h_s=f_\mf k$$ with $\deg h_{i+1}=\deg h_i+1$ and $h_i\mid h_{i+1}$ and consider the filtration of $W({\rm rad}f_\mf k)$ by the Lie ideals $W(h_i)$. 
The associated graded algebra $$\oplus_{i\ge0}([W(h_i)\cap\mf k]/[W(h_{i+1})\cap\mf k])$$ is isomorphic to a graded algebra $\mathord{\vtop{\offinterlineskip\lineskip.2ex\halign{\hfil#\hfil\cr \the\textfont1 $\mf k$\cr\the\textfont0\char126\cr}}}
$ 
with
\begin{center}$W(h_0)\supseteq\mathord{\vtop{\offinterlineskip\lineskip.2ex\halign{\hfil#\hfil\cr \the\textfont1 $\mf k$\cr\the\textfont0\char126\cr}}}\supseteq W(f_\mf k)$, $\codim_W\mf k=\codim_W\mathord{\vtop{\offinterlineskip\lineskip.2ex\halign{\hfil#\hfil\cr \the\textfont1 $\mf k$\cr\the\textfont0\char126\cr}}}$ and $f_\mf k=f_{\mathord{\vtop{\offinterlineskip\lineskip.2ex\halign{\hfil#\hfil\cr \the\textfont1 $\mf k$\cr\the\textfont0\char126\cr}}}}$.\end{center} 
The subalgebras $\mathord{\vtop{\offinterlineskip\lineskip.2ex\halign{\hfil#\hfil\cr \the\textfont1 $\mf k$\cr\the\textfont0\char126\cr}}}$ can be described in purely combinatorial terms and thus they give a collection of discrete invariants for $\mf k$.
This generalises the notation of gaps and leading degrees for subalgebras with one-point support.
\end{remark}

\subsection{Some general comments on finite codimension subalgebras of $W$}
The lattice of subalgebras $\mf k$ with $\supp(\mf k) = \{x\}$ and $a(\mf k)\le a$ can be naturally identified with the subalgebras of $$W(t-x)/W((t-x)^a))\cong W(t)/W(t^a);$$
in particular the isomorphism class of this lattice is independent of $x$. 

A similar result holds true in a greater generality. 
Pick $s\ge 0, a\ge1$ and distinct $x_1, \ldots, x_s\in\CC^\times$; set $h:=(t-x_1)\ldots(t-x_s)$. 
Consider subalgebras $\mf k$ satisfying $W(h)\supset \mf k\supset W(h^a)$. 
It is clear that $W(h^a)$ is a Lie ideal of $W(h)$ and the quotient $W(h)/W(h^a)$ is a finite-dimensional solvable Lie algebra. 
A version of the Chinese remainder theorem implies that 
\begin{equation}\label{Ecrmw+}W(h)/W(h^a)\cong \oplus_i W(t-x_i)/W((t-x_i)^a)\cong [W(t)/W(t^a)]^{\oplus s}.\end{equation} In particular, the isomorphism class of $W(h)/W(h^a)$ depends only on $a$ and $s$ but not on the particular choice of the $x_i$. 
This immediately gives the following corollaries.
\begin{corollary}Let $\mf k$ be a subalgebra of finite codimension of $\mf g=W, W_{\ge-1}, W_{\ge1}$ or $\Vir$. Then there are subalgebras $\mf k^+, \mf k^-$ with $$\mf k^-\subseteq\mf k\subseteq\mf k^+\mathrm{~and~}\codim_\mf g(\mf k^-)+1=\codim_\mf g(\mf k)=\codim_\mf g(\mf k^+)-1.$$\end{corollary}
\begin{proof}A similar statement is well-known for subalgebras of solvable Lie algebras so \eqref{Ecrmw+} implies the desired result if $\mf k\ne W(f_\mf k)$. 
If $\mf k=W(f_\mf k)$ the result follows from a similar fact on ideals in $\CC[t, t^{-1}]$.\end{proof}
\begin{corollary}The lattices of subalgebras of finite codimension of $\Vir$, of $ W$, and of $W_{\ge-1}$ are all isomorphic.\end{corollary}
\begin{proof}Lemma~\ref{prop:4.15} implies that the lattices of subalgebras of finite codimension are isomorphic for $W$ and $\Vir$; hence we left to show that the lattices are isomorphic for $W$ and $W_{\ge-1}$. 

Theorem~\ref{prop:4.14} implies that the lattices of subalgebras of finite codimension for both $W$ and $W_{\ge-1}$ are direct limits of the sublattices of subalgebras $\mf k$ containing $W((t-x_1)^a\cdots(t-x_s)^a)$ for all tuples $( a; x_1, \ldots, x_s)$; the only difference between $W$ and $W_{\ge-1}$ here is that in the first case $x_i\ne0$. 
These lattices are isomorphic for $W, W_{\ge-1}$ and the embeddings between them are the same for $W$ and $W_{\ge-1}$ (they essentially depend only on the integer-valued parameters $s$ and $a$ and on the cardinality of $\CC$).
\end{proof}

\section{Implications of our results for $\Ua(\mf g)$}\label{SUg}

In this final section, we shift, for the first time in this paper, to considering the universal enveloping algebra $\Ua(\mf g)$ of one of our Lie algebras of interest.
We
 apply a version of the orbit method to relate 
 the Poisson primitive ideals 
 $\ker p_\gamma =  \CoreP(\chi_{x;\alpha, \gamma})$ of $\Sa(\mf g)$ to primitive ideals of $\Ua(W)$ obtained as kernels of maps to the (localised) Weyl algebra.  
 We end with some conjectures about ideals in $\Ua(\Vir)$, $\Ua(W)$, and $\Ua(W_{\geq-1})$.
 
 %

%
%

\subsection{Constructing primitive ideals through the orbit method} \label{SSorbit}
 Kirillov's orbit method  for nilpotent and solvable Lie algebras attaches to $\chi\in\mf g^*$ the annihilator in $\Ua(\mf g)$ of the module induced from a polarization of $\chi$ and the induced  character. 
We apply the same construction to $\mf g=W$ and $\chi=\chi_{x; \alpha, \gamma}$.
We will denote the corresponding induced $W$-modules by $M_{x; \gamma}$.

A description of the annihilators of $M_{x, \gamma}$ is given in Proposition~\ref{Prometb}. 
The main result here is that  $$ \Ann_{\Ua(W)}M_{x; \gamma}$$ depends only on $\gamma$; thanks to Theorem~\ref{Texpw} (or Lemma~\ref{lem:Jgamma}) the same holds for $\CoreP(\chi_{x; \alpha; \gamma})$. 
This shows that the constructions of Kirillov's orbit method give rise to a map from a certain class of Poisson primitive ideals of $\Sa(W)$ to a certain class of primitive ideals of $\Ua(W)$, which are known in the literature \cite{CM, SW2} as kernels of maps to the localised Weyl algebra. 
We believe this map extends to a surjection from Poisson primitive ideals of $\Sa(\mf g)$ to primitive ideals of $\Ua(\mf g)$; this is the subject of ongoing research. 

Throughout this section, we write the localised Weyl algebra   as   $A = \kk [t,t^{-1}, \del]$, with $\del t = t \del + 1$.

We first describe a polarization for $\chi_{x; \alpha, \gamma}$. 
Let $x, \alpha , \gamma \in \kk$ with $x \neq 0, (\alpha , \gamma) \neq (0,0)$ and let $\chi= \chi_{x;\alpha , \gamma}$. Recall the computation of $W^\chi$ in Lemma~\ref{lem:Wchi}, and consider the Lie subalgebra $W(t-x)$ of $W$, which contains $W^{\chi}$.
We have $$\dim W/W^\chi = 2\implies W(t-x) = \CC (t-x) \del \oplus W^\chi.$$  
As $B_\chi((t-x) \del, (t-x)\del) = 0$,  thus $W(t-x)$ is a totally isotropic subspace of $(W, B_\chi)$; by dimension count it is maximal totally isotropic.  
Thus $W(t-x)$ is a {\em polarization} of $W$ at $\chi$, as in \cite[1.12.8]{Dixmier}.
Further, $W(t-x)$ is the unique polarization of $W$ at $\chi$: since any polarization of $W$ at $\chi$ must be a codimension 1 subalgebra of $W$, by Corollary~\ref{cor:codim1} it must be equal to some $W(t-y)$ and it is easy to see that we must have $y=x$. 

Note that $\chi$ is a character of $W(t-x)$; let 
 $\CC m_{x;\gamma}$ be the corresponding 1-dimensional representation of $W(t-x)$, with basis element $m_{x; \gamma}$.
 (The restriction $\chi|_{W(t-x)}$, which sends $p \del \mapsto \gamma p'(x)$, depends only on $x$ and $\gamma$ --- this is the reason to omit $\alpha $ in the notation $\CC m_{x; \gamma}$.  )
Put $$M_{x; \gamma} := \Ua(W) \otimes_{\Ua(W(t-x))} \CC m_{x;\gamma}.$$
Since $e_{-1} =  \del \not \in W(t-x)$, thus $W = \CC e_{-1} \oplus W(t-x)$ and by the Poincar\'e-Birkhoff-Witt theorem the set $\{ e_{-1}^k m_{x;\gamma} | k \in \NN\} $ is  a basis for~$M_{x; \gamma}$.

We now give an alternative construction of $M_{x; \gamma}$.  Set $$N_x := \kk[t, t^{-1}, (t-x)^{-1}]/\kk[t, t^{-1}].$$ 
For every $x$, $N_x$ is a simple (faithful) left $A$-module. 

\begin{remark}
The space $N_x$ can be thought of as a space of distributions on $\kk^\times$; for, setting $\delta_x = (t-x)^{-1} \in N_x$ we have $(t-x) \delta_x = 0$ so $\delta_x$ behaves like  a $\delta$-function at $x$.  The elements $\del^k \delta_x$ form a basis of $N_x$.
\end{remark}

Recall that, for any $\gamma \in \CC$, the map
\begin{equation}\label{Epgm} \pi_\gamma: W \to A, \quad f\del  \mapsto f \del + \gamma f'\end{equation}
is  a Lie algebra homomorphism; see \cite{CM}. 
Thus $\pi_\gamma$ extends to define a ring homomorphism $\Ua(W) \to A$.  
Note that the map $p_\gamma$ defined in \eqref{pgamma} is {\em not} the associated graded map attached to $\pi_\gamma$ even though they  are clearly analogous.

The images of $\pi_\gamma$ have been computed in  \cite[Lemma~2.1]{CM}, and we give them here.

\begin{lemma}\label{Lpiim} We have 
\begin{itemize} \item[(a)]  $\im(\pi_0) = \kk \oplus A \del$, and $\im \pi_1 = \kk \oplus \del A$.
\item[(b)]  $\im(\pi_\gamma) = A$ if $\gamma\ne0, 1$.
\end{itemize}
\end{lemma}

\begin{remark}\label{rem:kerpi} The restriction of $\pi_\gamma$ to $\Ua(W_{\geq 1})$ was  considered, under slightly different notation, in \cite{SW2}; see \cite[Remark~3.14]{SW2}.  It was shown there that the ideal $\ker \pi_0|_{W_{\geq 1}} = \ker \pi_1|_{W_{\geq 1}}$ is not finitely generated as a left or right ideal of $\Ua(W_{\geq 1})$.
\end{remark}

For every $\gamma\in\CC$ the map $\pi_\gamma$ from~\eqref{Epgm} induces the structure of a $W$-module on $N_x$; we denote the space $N_x$ with the corresponding $W$-module structure $\cdot_\gamma$  by $N_x^\gamma$.  
\begin{proposition}\label{Prometb} 
Let $x\neq 0, \alpha, \gamma \in \kk$.
Then $N_x^\gamma\cong M_{x; \gamma-1}$. 
Moreover, $N_x^\gamma$ is a simple $W$-module if and only if $\gamma\ne1$.
There is an exact sequence 
\[ 0 \to N_x^0 \to N_x^1 \to \kk \to 0.\]

For all $\gamma$ the annihilator of $N_x^\gamma$ is equal to $\ker \pi_\gamma$, which is primitive.\end{proposition}
\begin{proof}
Let $p\in\CC[t, t^{-1}]$ with $(t-x)|p$. By taking the  Taylor expansion of $p $ about $x$ one may verify that 
\beq\label{eqp}(p\del)\cdot_\gamma\delta_x=(\gamma-1) p'(x)\delta_x.\eeq
This immediately implies that there exists a surjective module homomorphism $M_{x; \gamma-1}\to N_x^{\gamma} $ which sends $m_{x; \gamma-1}\mapsto\delta_x$. 
The basis element $e_{-1}^k m_{x; \gamma-1}$ maps to $\del^k \delta_x$, so this map is an isomorphism.

Assume $\gamma \neq 0,1$. 
In this case $\pi_\gamma$ is surjective by Lemma~\ref{Lpiim}, so simplicity of  $N_x^\gamma$ follows from simplicity of $N_x$ as an $A$-module.

We claim that $N_x^0$ is simple.  To see this, let $ 0\neq n \in N_x$, which we recall is a simple $A$-module. 
By construction of $N_x$ there is some $0 \neq f(t) \in \kk[t, t^{-1}]$ so that $fn =0$, and thus $\del n \neq 0$ as $A\del+ Af(t) = A$.
Thus $A \del n = N_x$, proving the claim since by Lemma~\ref{Lpiim} $A\del \subseteq \pi_0(\Ua(W))$.

Finally, we consider $N_x^1$.    
Lemma~\ref{Lpiim} implies $\del A \triangleleft \im \pi_1 = \kk \oplus \del A$, and thus $\tilde{N} := \del A N_x = \del N_x$ is a submodule of $N_x^1$, with  $N_x^1/\tilde{N} \cong \kk$.
We claim that $\tilde{N}$ is simple.  Let $0 \neq n \in \tilde{N}$.  
As $N_x$ is a simple $A$-module, $\del A n = \tilde{N}$ and so $\Ua(W) \cdot_1 n = (\kk \oplus \del A) n = \tilde{N}$, as needed.  The reader may verify that as in \eqref{eqp}
\[ p \del \cdot_1 \del \delta_x = -p'(x) \del \delta_x,\]
and so $\tilde{N} \cong M_{x; -1} \cong N_x^0$.
The claim about annihilators follows from the fact that $N_x$ is a faithful module over the (simple) ring $A$.
That $\ker \pi_\gamma$ is primitive is immediate for $\gamma \neq 1$; for $ \gamma=1$ it follows from the fact that $\ker \pi_1 = \ker \pi_0$.
\end{proof}
\begin{remark}
\begin{enumerate}
\item[(a)]Note that the primitive ideal $\ker \pi_\gamma$ is completely prime. 
We do not know of  a primitive (or prime) ideal of $\Ua(W)$ which is not completely prime.

\item[(b)] We believe that the ideals $\ker \pi_\gamma$ above  are all of the primitive ideals of $\Ua(W)$ of Gelfand-Kirillov codimension 2. 

\item[(c)]  By Remark~\ref{rem:kerpi}, $\ker \pi_0 = \ker\pi_1$.  However, $\ker p_0 \neq \ker p_1$.  
To see this, note that if $\ker p_0 = \ker p_1 $ then $\chi_{1;1,0} \in V(\ker p_1) = \overline{\OO(\chi_{1;1,1})}$.
Thus either $\OO(\chi_{1;1,0} ) = \OO(\chi_{1;1,1})$ or $\dim \OO(\chi_{1;1,0}) < \dim \OO(\chi_{1;1,1})$.  Neither is true.

Thus the orbit method does not give a bijection from Poisson primitive ideals of $\Sa(W)$ to primitive ideals of $\Ua(W)$.  
\end{enumerate}
\end{remark}

\subsection{Conjectures for $ \Ua(\mf g)$}\label{SSconjectures}
We have focused almost entirely on the symmetric algebra of $\mf g$, where $\mf g$ is one of $\Vir$, $W$, or several related Lie algebras.  
However, our results are at minimum suggestive for the enveloping algebras of these Lie algebras.  
In this final subsection, we make several conjectures for $\Ua(\mf g)$.  
Broadly speaking, these are instances of the meta-conjecture:

\begin{center}
{\em The ideal  structure of $\Ua(\mf g)$ is closely analogous to the Poisson structure of $\Sa(\mf g)$. }
\end{center}

For each conjecture, we give the Poisson result which suggested it to us.

\begin{conjecture}[cf. Corollary~\ref{cor:Psimple}]
If $\zeta \neq 0$ then $\Ua(\Vir)/(z-\zeta)$ is simple.
\end{conjecture}
\begin{conjecture}[cf. Corollaries~\ref{Csameprimspec}, \ref{Csameprimespec}]
Restriction gives a bijection between primitive (respectively, prime) ideals of $\Ua(W)$ and $\Ua(W_{\ge-1})$, and  a homeomorphism  ${\rm PSpec}_{\rm prim}\Ua(W) \stackrel{\sim}{\to} {\rm PSpec}_{\rm prim}\Ua(W_{\geq -1})$.
\end{conjecture}
\begin{conjecture}[cf. Proposition~\ref{Pradprim}] Every proper prime ideal of $\Ua(W)$ contains a proper primitive ideal.\end{conjecture}

\begin{conjecture}[cf. Proposition~\ref{PGK2}] The $\Ker \pi_\gamma$ are all of the prime ideals of $\Ua(W)$ of co-GK-dimension 2.
\end{conjecture}
\begin{conjecture}[cf. Subsection~\ref{SSorbit}]
 Kirillov's orbit method, i.e. the assignment 
\begin{center}a local function $\chi\to$ a polarization of $\chi\to$ the annihilator of the induced module\end{center} always produces a primitive ideal, is independent of polarization, and depends only on $\OO(\chi)$.  There is thus an induced map  ${\rm PSpec}_{prim}\Sa(\mf g) \to \Spec_{\rm prim}\Ua(\mf g)$ for $\mf g=\Vir, W, W_{\ge-1}$.
This map is surjective onto $\Spec_{\rm prim} \Ua(\mf g)$.
\end{conjecture}

These conjectures are the subject of ongoing research.

 \section*{Index of notation}\label{index}
 
\begin{multicols}{2}
{\small  \baselineskip 14pt
 %
 %

$(a, b)_f$ \hfill\pageref{Esres} 

$a(\mf k)$ \hfill\pageref{not:onepoint}

$\AA(\lambda)$ \hfill\pageref{ind:Alambda}, \pageref{ind:Alambda2}

$B_\chi(x,y)$ \hfill\pageref{ind:Bchi}

$\chi_{x; \alpha_0,\ldots, \alpha_n}$ \hfill\pageref{Elf}

$d(\mf k)$ \hfill\pageref{not:onepoint}

$D_\chi$ \hfill\pageref{ind:Dchi}

$D(\lambda(\chi))$ \hfill\pageref{ind:Dlambda}

$\Dil_x$ \hfill\pageref{ind:Dil}

$\DLoc_x$ \hfill\pageref{ind:DLoc}

$\DLoc_x^{\le n}$ \hfill\pageref{ind:DLocn}

$\DLoc_x^{k+}$, $\DLoc_x^{\le n, k+}$ \hfill \pageref{ind:DLock}

$\widehat{\DLoc}^\lambda$ \hfill \pageref{ind:whDLoclambda}

$\widehat{\DLoc}^{\le n}$,  $\widehat{\DLoc} $ \hfill\pageref{ind:whDloc}

$\widehat{\DLoc}^\lambda$-orbit quotient \hfill\pageref{ind:orbitquotient}

$D(u_1,  \ldots, u_n; v_1,  \ldots, v_n)$  \hfill\pageref{ind:D}

$e_i(x)$, $e_i(x)^*$ \hfill\pageref{ind:ei}

$\End_{t\to s}(\cdot)$ \hfill\pageref{ind:End}

 $\ev_\chi$ \hfill\pageref{ind:evchi}
 
 $f_{\mf k}$ \hfill\pageref{ind:fk}
 
 $\sdeg(\mf k)$,  gaps of  a subalgebra \hfill\pageref{ind:gaps}
 
 $\mf g^\chi$ \hfill\pageref{ind:isotropy}
 
 $I(n)$ \hfill\pageref{detideal}
 
 $I(X) \subseteq A$ \hfill \pageref{ind:I}
 
 $J^\lambda_{\mf g} := I(V^\lambda_{\mf g})$ \hfill\pageref{ind:Jlambda}
 
 $\lambda(\chi)$, order partition of $\chi$ \hfill \pageref{ind:lambda}
 
 $\lambda(Q)$, generic order partition of $Q$ \hfill \pageref{thm:paramprimes}
 
$ \ldeg(\mf k)$, leading degrees of  a subalgebra \hfill\pageref{ind:ldeg}

 $\lie_x^{\le n}$ \hfill\pageref{ind:lien}

 $\widehat{\lie}^{\leq n}$ \hfill\pageref{ind:whlie}

$\Loc^\lambda_{\mf g}$ \hfill\pageref{ind:Loclambda}

 $\Loc^{\le n}_{\mf g}$, 
 $\widetilde{\Loc}_{\mf g}^{\leq n}$, 
 $\Loc_x$, 
 $\Loc_x^{\le n}$ \hfill\pageref{ind:Loc1}
 
 local function \hfill\pageref{def:local}

 $\mf{m}_\chi$  \hfill\pageref{ind:mchi}, \pageref{ind:mchi2}

 $\MSpec A$ \hfill\pageref{ind:Mspec}
 
 $\overline{\mb O(\chi)}$, orbit closure of $\chi$ \hfill\pageref{ind:orbitclosure}
 
order of a local function \hfill\pageref{ind:order}

$p_\gamma: \Sa(W) \to \kk[t, t^{-1}, y]$ \hfill \pageref{pgamma}
 
$ \pi_{\mf g}^{\leq n}$, 
$\phi_{\mf g }^{\leq n}$ \hfill\pageref{ind:phig}

 
 $\CoreP(\chi)$, Poisson core \hfill\pageref{ind:poissoncore}, \pageref{ind:poissoncore2}
 
 
 
 
 Poisson morphism \hfill\pageref{ind:poissonmor}
 
 

$\PSpec_{\rm prim}A$, $\PSpec A$ \hfill\pageref{ind:Pspec}

pseudo-orbit, $\mb O(\chi)$ \hfill\pageref{def:pseudoorbit}

${\rm rad}(f)$ \hfill\pageref{ind:rad}

$\mf S_\lambda$ \hfill\pageref{ind:Slambda}

$\Sigma^\lambda_{\mf g}$ \hfill\pageref{ind:Sigma}

$\Shift_z$, $\Shifts$ \hfill\pageref{ind:Shifts}

support of a local function \hfill\pageref{ind:support}

support of a subalgebra, $\supp(\mf k)$ \hfill\pageref{ind:suppk}

 
 $u_\chi, \mf u_\chi$ \hfill\pageref{ind:uchi}
 
$\Vir$ \hfill\pageref{ind:Vir}

$V^\lambda_{\mf g}$ \hfill\pageref{ind:Vlambda}, \pageref{ind:Vlambda2}

$V(N) \subseteq \MSpec(A)$ \hfill\pageref{ind:V}

$\VP(N)  \subseteq  \PSpec A $ \hfill\pageref{ind:VP}

$W$ \hfill\pageref{ind:W}

$W(g)$ \hfill\pageref{ind:Wg}, \pageref{ind:Wg2}

$W_{\geq-1}$ \hfill\pageref{ind:W-1}

$W_{\geq -1}(f)$ \hfill\pageref{ind:W-1f}

$W_{\geq1}$ \hfill\pageref{ind:W1}

$W_{x; \alpha}^{2; 1}$, $W_{x; \alpha}^{2; 2}$ \hfill\pageref{ind:codim2}

$ W_{x, y; \alpha, \beta}^{3A}$, $W_{x, y; \alpha}^{3B1}$, etc.
\hfill\pageref{ind:codim3}

$Y(Q)$ \hfill \pageref{thm:paramprimes}
 
}
\end{multicols}

\end{document}